\documentclass[10pt,leqno]{amsart}
\usepackage{graphicx}
\usepackage{indentfirst,csquotes}

\usepackage{graphicx}
\usepackage{amsmath,amsfonts,amssymb,amsthm}
\usepackage{xcolor}
\usepackage{booktabs}
\usepackage{array}
\usepackage{url}
\usepackage{verbatim}
\usepackage{textcomp}
\usepackage{indentfirst,csquotes}
\usepackage{paralist}
\usepackage{etoolbox}
\usepackage{titlesec}
\usepackage{fancyhdr}

\usepackage{algorithm}
\usepackage{algorithmic}

\usepackage{subfig}

\usepackage[colorlinks=true,linkcolor=black,citecolor=black,urlcolor=black,filecolor=black]{hyperref}

\newtheorem{theorem}{Theorem}[section]
\newtheorem{definition}[theorem]{Definition}
\newtheorem{example}{Example}
\newtheorem{lemma}[theorem]{Lemma}
\newtheorem{proposition}[theorem]{Proposition}

\newtheorem{assumption}[theorem]{Assumption}

\newtheorem{remark}[theorem]{Remark}

\titleformat{\section}
  {\normalfont\Large\bfseries}
  {\thesection}
  {1em}
  {}

\titlespacing*{\section}{0pt}{1ex}{1ex}

\titleformat{\subsection}
  {\normalfont\large\bfseries}
  {\thesubsection}
  {1em}
  {}

\titleformat{\subsubsection}
  {\normalfont\normalsize\bfseries}
  {\thesubsubsection}
  {1em}
  {}

\titlespacing*{\section}{0pt}{2ex}{1ex}
\titlespacing*{\subsection}{0pt}{1.5ex}{0.8ex}
\titlespacing*{\subsubsection}{0pt}{1.2ex}{0.6ex}

\usepackage{lipsum}

\begin{document}
\title{Shallow neural network yields regularization for ill-posed inverse problems} 
\author[Wang et al.]{Lan Wang$^1$, Qiao Zhu$^1$, and Ye Zhang$^{1,3,*}$} \thanks{$^1$School of Mathematics and Statistics, Beijing Institute of Technology, Beijing 100081, China.} 
\thanks{$^3$MSU-BIT-SMBU Joint Research Center of Applied Mathematics, Shenzhen MSU-BIT University, Shenzhen 518172, China.} 
\thanks{$^*$Author to whom any correspondence should be addressed. Email: \texttt{ye.zhang@smbu.edu.cn}.}


\begin{abstract}
In this paper, we develop  regularization theory for neural network approximations of general ill-posed operator equations with noisy data. Within the framework of iterative regularization, we propose two expanding neural network based schemes under different a priori assumptions. Instead of prescribing a fixed architecture, ENNs adaptively select the number of neurons through an a posteriori stopping rule, so that the network size serves as a regularization parameter balancing approximation accuracy and stability with respect to data noise. We prove their regularizing properties and quantitative relationships between the  network size and the noise level. Within the framework of variational regularization, we propose a neural network-based Tikhonov scheme and derive both convergence and convergence-rate results under mild assumptions. The estimates depend on the noise level, the network size, and the underlying smoothness expressed through general variational source conditions. Numerical experiments demonstrate the effectiveness and robustness of the algorithms, and show that, for highly noisy data, relatively small network architectures can already produce stable reconstructions.
\end{abstract}


\maketitle

\section{Introduction}
In this work, we study neural network methods for solving general operator equations
\begin{equation}
\label{1}
A(f)=g, 
\end{equation}
where $A \in \mathbb{M}(\mathcal{X}, \mathcal{Y})$, and $\mathbb{M}(\mathcal{X}, \mathcal{Y})$ denotes the space of all bounded (possibly non-linear) operators between two reflexive Banach spaces of functions. Suppose that instead of exact right-hand side $g$, we are given only perturbed element $g^\delta$, with a noise level $\delta>0$:
\begin{equation}
\label{NoiseLevel}
\|g^{\delta}-g\|_\mathcal{Y} \le \delta.
\end{equation}

Since problem \eqref{1} is generally ill-posed, directly fitting the noisy data $g^\delta$ may lead to instability, and one may employ a  regularization method  to construct a stable approximation $f^\delta$ to $f^\dagger$. Following the concept of regularization \cite{Yagola1995}, the approximation $f^\delta$ is called a regularized solution if it converges to $f^\dagger$ in a prescribed topology, as $\delta\to0$. Specifically, given the noisy data $g^\delta$, our goal is to construct a regularized solution $f^\delta$ to the inverse problem \eqref{1}-\eqref{NoiseLevel}. 
In this work, we seek stable finite-dimensional regularized solutions constructed by neural networks as $\delta\to0$.

The approximation theory of neural networks characterizes approximation capabilities for prescribed function classes and quantifies their dependence on network architecture. Classical universal approximation results show that single-hidden-layer feedforward networks with suitable activation functions can approximate continuous functions on compact sets \cite{cybenko1989approximation,hornik1989multilayer,hornik1991approximation}. Barron \cite{barron1993universal} derived quantitative bounds for one-hidden-layer neural networks, showing that functions with suitable Barron-type regularity can be approximated at the rate $\mathcal O(n^{-1/2})$ in $L^2$, where $n$ is the number of hidden neurons. Subsequent studies have revealed the advantages of depth in neural
network approximation \cite{telgarsky2016benefits,eldan2016power} and
established quantitative approximation rates for ReLU networks in terms of their width and depth for various function classes \cite{lu2021deep,shen2022optimal,yarotsky2020phase,hon2022simultaneous}.
These approximation results provide important guidance on how the expressive power of neural networks grows with architectural complexity. However, these results primarily concern direct function approximation and do not account for the instability introduced by an ill-posed observation operator. In ill-posed inverse problems, the unknown solution must be recovered from noisy indirect data through a forward operator. Hence, increasing the network size improves approximation capability but may also amplify the instability caused by data perturbations.


Closer to the present work are neural network methods that are explicitly connected to regularization theory. The NETT approach defines a data-driven regularizer through a trained neural network and incorporates it into a generalized Tikhonov functional, for which well-posedness, convergence, and error estimates have been established \cite{li2020nett}. Another related direction is the deep image prior, in which the architecture of an untrained convolutional generator network itself serves as an implicit image prior \cite{ulyanov2018deep}. Dittmer et al. further investigated deep image prior methods for ill-posed inverse problems and interpreted them as a form of regularization induced by the network architecture \cite{dittmer2020regularization}. These approaches demonstrate that neural networks can contribute to regularization either through learned regularizers or through architecture-induced priors. In contrast, the present paper employs neural networks directly to represent regularized solutions, rather than to define an external regularizer or to impose a fixed generator prior. More precisely, the approximation in this work is obtained through an unsupervised learning approach: it is learned directly from the given noisy observation $g^\delta$, without relying on external training samples or labeled exact solutions. Moreover, we do not prescribe a fixed architecture in advance. Instead, we investigate whether suitably constrained and adaptively selected classes of shallow neural networks can serve as admissible regularization spaces for ill-posed inverse problems. In this framework, the network width is treated as an architecture-dependent regularization parameter and is selected from the noisy data. This perspective naturally leads to the following questions:
\begin{itemize}
\item[(\textbf{Q1}):] Does there exist a neural network approximation that provides a regularized solution to the ill-posed problem?
\item[(\textbf{Q2}):] Is a larger neural network always preferable for solving ill-posed inverse problems with noisy data, or is there a stability limitation on the size or complexity of the neural network? What is the relation between noise level and size of neural network?
\item[(\textbf{Q3}):] If such a limitation exists, how can an appropriate network size be selected adaptively from the noisy data? 
\item[(\textbf{Q4}):] What is an appropriate size of a neural network for a variational regularization method based on neural networks that achieves optimal convergence rates with respect to the noise level for general ill-posed operator problems (\ref{1})?
\end{itemize}
The questions (Q1)--(Q3) are addressed by the expanding neural network methods (ENNs) developed in Section~\ref{theory}. Within the framework of iterative regularization theory \cite{engl1996regularization,BK2008}, and inspired by the method of extending compacts \cite{YagolaTitarenko2002,DorofeevYagola2004,ZhangGulliksson2019}, we propose two expanding neural network methods to construct regularized solutions using shallow neural networks, where the number of neurons is adaptively selected by an a posteriori stopping rule, and the selected network size serves as a regularization parameter that balances approximation accuracy and stability in the presence of noisy data.
To address (Q4), we further develop a neural network-based Tikhonov regularization scheme (Tikhonov-NN) in Section~\ref{Tikhonov}. Within the framework of variational regularization theory, in a spirit similar to PINNs for approximating solution functions of PDEs, we study neural network approximations for inverse problems in a general abstract setting and derive unified convergence results under mild assumptions.

Our analysis of these unified convergence results is carried out for two-layer neural networks. A natural functional framework for such networks is the so-called \textit{Barron space}. Li et al. \cite{LiYuanyuan} established approximation results for two-layer ReLU-type networks in Barron spaces: for functions in $\mathcal B_1$, such networks can  achieve an $H^1$ approximation rate $\mathcal O(n^{-1/2})$. This result relates the network width, the Barron representation cost, and the approximation error, and thus provides the theoretical foundation for using bounded two-layer neural network classes as admissible solution spaces in our regularization analysis.

The remainder of the paper is structured as follows. In Section \ref{Preliminaries}, we provide preliminaries on Barron spaces. In Section \ref{theory}, we discuss iterative regularization, and propose two expanding neural network methods. In Section \ref{Tikhonov}, we discuss Tikhonov regularization. We present numerical experiments in Section \ref{simulation}, and give concluding remarks in Section \ref{sec:Con}.

\section{Preliminaries of Barron spaces}
\label{Preliminaries}
Functions in a Barron space take the following form
\begin{equation}
\label{continuef}
f(\mathbf{x})=\int_{P} a  \sigma(\boldsymbol{b}^T \mathbf{x}+c) \rho(\mathrm{d} a, \mathrm{d} \boldsymbol{b}, \mathrm{d} c), \quad \mathbf{x} \in \Omega\subset \mathbb{R}^d,
\end{equation}
where  $P=\mathbb{R}\times \mathbb{R}^d\times \mathbb{R}$, $\rho$ is a probability distribution on $(P,\Sigma_P)$ with $\Sigma_P$ being a Borel $\sigma$-algebra on $P$, $\sigma$ is the activation function, and $\Omega$ is a bounded domain. We focus on the ReLU activation $\sigma(z):=\max(0,z):=(z)_{+}$. Then given the probability distribution $\rho$, we define
\begin{equation*} 
\|f\|_{\mathcal{B}_{p,\rho}}=\left(\mathbb{E}_\rho\left[|a|^p\left(\|\boldsymbol{b}\|_1+|c|\right)^p\right]\right)^{1 / p}=\left(\int_{P}|a|^p(\|\boldsymbol{b}\|_{1}+|c|)^{p})\rho(\mathrm{d} a,\mathrm{d}\boldsymbol{b},\mathrm{d} c)\right)^{\frac{1}{p}}, \quad 1 \le p \le+\infty.
\end{equation*}
The Barron norm $\|f\|_{\mathcal{B}_p}$ of a function $f$ of the form (\ref{continuef}) is defined by
\begin{equation}
\label{Bpnorm}
    \|f\|_{\mathcal{B}_{p}}=\inf_{\rho} \|f\|_{\mathcal{B}_{p,\rho}}. 
\end{equation}
The Barron spaces $\mathcal{B}_p$ are defined as the set of functions that can be represented by $(\ref{continuef})$ with a finite Barron norm: 
\begin{equation}
\label{Barron}
\mathcal{B}_p := \{ f:~ f \text{~satisfies~} \eqref{continuef} \text{~and~} \|f\|_{\mathcal{B}_{p}}<\infty \}. 
\end{equation}
The Barron space $\mathcal B_p$ consists of continuous functions, and is a subset of $C(\overline\Omega)$. Then the following properties hold for Barron spaces.
\begin{proposition}
[{\cite{Weinan2019BarronSA}}]For any $1\le p\le \infty$, there hold $\mathcal{B}_1=\mathcal{B}_{p}$ as sets and $\|f\|_{\mathcal{B}_1}=\|f\|_{\mathcal{B}_{p}}$.
\end{proposition}

\begin{proposition}[{\cite{LiYuanyuan}}]
\label{proBarron}
The set $\mathcal{B}_1$ equipped with the functional $f \mapsto\|f\|_{\mathcal{B}_1}$ is a normed space. Moreover, the space $\mathcal{B}_1$ is a subspace of $H^1(\Omega)$, and $\mathcal{B}_1$ embeds continuously into $H^1(\Omega)$.
\end{proposition}

Two-layer neural networks are empirical analogues and have the form
\begin{equation}
\label{repfn}
f_n(\mathbf x)
=
\frac{1}{n}
\sum_{j=1}^n
a_j(\boldsymbol b_j^T\mathbf x+c_j)_+,
\quad
(a_j,\boldsymbol b_j,c_j)\in
\mathbb R\times\mathbb R^d\times\mathbb R .
\end{equation}
For $r>0$, let $M_r
:=
\left\{
(a,\boldsymbol b,c)\in\mathbb R\times\mathbb R^d\times\mathbb R:
|a|\le r,\ 
\|\boldsymbol b\|_1+|c|=1
\right\}$, which is compact in
$\mathbb R\times\mathbb R^d\times\mathbb R$. We define the class of
two-layer neural networks with width $n$ and parameter radius $r$ by
\begin{equation}
\label{2}
X_{n,r}
:=
\bigg\{
f_n:
f_n(\mathbf x)
=
\frac{1}{n}
\sum_{j=1}^n
a_j
\left(\boldsymbol b_j^T\mathbf x+c_j\right)_+,
\quad
(a_j,\boldsymbol b_j,c_j)\in M_r
\bigg\}.
\end{equation}
For every $f_n\in X_{n,r}$, define the empirical probability measure $\rho_n:=\frac1n\sum_{j=1}^n
\delta_{(a_j,\boldsymbol b_j,c_j)}$ where \(\delta_z\) denotes the Dirac point mass at \(z\). Then $f_n$ admits an integral representation $f_n(\mathbf x)=\int_P
a(\boldsymbol b^T\mathbf x+c)_+
\,\rho_n(da,d\boldsymbol b,dc)$ for $\mathbf x\in\Omega$.
In practice, the radius $r$ may depend on the width $n$, i.e., $r=r(n)$, with
nondecreasing $\{r(n)\}_{n\ge1}$.
Let
\begin{equation}\label{eq:Xninfty_union}
X_{n,\infty}
:=\bigcup_{r<\infty}X_{n,r}.
\end{equation}

We use the following compactness results in the convergence analysis. See Appendix \ref{App:pre} for the proofs.

\begin{lemma}
\label{lemma compact set}
For any fixed $n\in \mathbb{N}$ and $r>0$, $X_{n,r}$ is sequentially compact in $C(\overline\Omega)$ and $L^2(\Omega)$.
\end{lemma}

\begin{lemma}
\label{convergesubsq}
Any bounded sequence in Barron space $\mathcal{B}_1$ or Sobolev space $H^1(\Omega)$ has a subsequence that converges in $L^2(\Omega)$.
\end{lemma}
\begin{proof}
It holds since $\mathcal{B}_1$ continuously embeds into $H^1(\Omega)$, and $H^1(\Omega)$ compactly embeds into $L^2(\Omega)$. 
\end{proof}

\begin{lemma}
\label{lem:barron_compact_lsc}
Let $\Omega$ be bounded. Let
\[
f_m(\mathbf x)
=
\frac1{n_m}\sum_{j=1}^{n_m}
a_j^{(m)}
\left((\boldsymbol b_j^{(m)})^T\mathbf x+c_j^{(m)}\right)_+,\quad \|\boldsymbol b_j^{(m)}\|_1+|c_j^{(m)}|=1.
\]
If $q_m:=
\frac1{n_m}\sum_{j=1}^{n_m}|a_j^{(m)}|
\le Q$ for some constant $Q>0$ independent of $m$, then $\{f_m\}$ is relatively compact in $C(\overline\Omega)$. Moreover, if a subsequence $f_{m_k}$ converges to $f$ in $C(\overline\Omega)$, then $f\in\mathcal B_1$ and $
\|f\|_{\mathcal B_1} \le\liminf_{k\to\infty} q_{m_k}$.
\end{lemma}

The next result shows that functions defined in the Barron space $\mathcal{B}_1$ can be approximated by two-layer neural networks (\cite[Theorem 8]{LiYuanyuan} and \cite[Theorem 4]{Weinan2019BarronSA}).
\begin{proposition}\cite{LiYuanyuan}
\label{Approximtion theorem}
Let $f \in \mathcal{B}_1$, and let $\rho$ be its representing probability. Define
\[
c_{\rho}(f)=\|f\|_{\mathcal{B}_{1,\rho}}=\int
|a|(\|\boldsymbol b\|_1+|c|)
\,d\rho(a,\boldsymbol b,c).
\]
Assume that $c_\rho(f)<\infty$. Then, for any $n \in \mathbb{N}$, there exists a probability $\rho_n$, $f_n\in X_{n,c_{\rho}(f)}$ with $\left\|f_n\right\|_{\mathcal{B}_1} \le c_{\rho_n}\left(f_n\right) \le c_{\rho}(f)$, and $C(\Omega, d)>0$ such that
\[
\|f-f_n\|_{H^1(\Omega)}
\le
\frac{C(\Omega,d)}{\sqrt n}\,c_\rho(f).
\]
\end{proposition}

\section{Expanding neural network methods}
\label{theory}

In this section, based on the method of expanding compacts \cite{YagolaTitarenko2002,DorofeevYagola2004}, we propose a new iterative regularization method, i.e., \textit{expanding neural network method} (ENNs). It converges to the unique exact solution $f^{\dagger}$ to $A(f) = g$, if $A$ is a continuous injective operator. We consider two cases, depending  on whether certain a priori information on $f^\dagger$ is available.

\subsection{Case 1: given an energy bound of \texorpdfstring{$c_{\rho}(f^{\dagger})$}.}
\label{CaseI}

Let $b:\mathbb{N}^+\to \mathbb{N}^+$ be a monotonically increasing function. Assume that there exist $C_b, k_0\ge 1$ such that $ b(k+1)\le C_b\, b(k)$ holds for all $k\ge k_0$. The first iterative regularization scheme is given in Algorithm \ref{alg:expanding_nn}. It requires the \emph{a priori} information $c_\rho(f^\dagger)$ of $f^\dagger$, where $\rho$ is a representing probability measure of $f^\dagger$.

\begin{algorithm}[H]
\caption{ENN1: expanding neural network for solving (\ref{1}) with knowledge of $c_{\rho}(f^{\dagger})$}
\label{alg:expanding_nn}
\begin{algorithmic}[1]
\REQUIRE $b(k)$, $r(n)$,  $\tau>1$, and $\delta$
\ENSURE Reconstruction $f_{n(\delta),k(\delta)}^{\delta}$, $n(\delta)$ and $k(\delta)$

\STATE \textbf{Initialization:} $k \gets 1$, $n \gets b(1)$
\STATE Compute 
    \begin{equation}\label{OPT-ENN1}
        f_{n,k}^{\delta} \in
        \arg\min\limits_{f \in X_{n,r(n)}}
        \big\|A(f)-g^{\delta}\big\|_{\mathcal{Y}}\ ,\quad J_{n,k}^{\delta} = \big\|A(f_{n,k}^{\delta})-g^{\delta}\big\|_{\mathcal{Y}} 
    \end{equation}

\WHILE{$J_{n,k}^{\delta} > \tau\delta$}
    \STATE $k \gets k+1$
    \STATE $n \gets b(k)$
    \STATE Update $f_{n,k}^{\delta}$ and $J_{n,k}^{\delta}$ by resolving \eqref{OPT-ENN1}
           on the expanded class of neural networks $X_{n,r(n)}$
\ENDWHILE

\STATE Set $k(\delta) \gets k$, $n(\delta) \gets n$
\RETURN $f_{n(\delta),k(\delta)}^{\delta}$, $n(\delta)$, $k(\delta)$
\end{algorithmic}
\end{algorithm}

The update rule $n \gets b(k)$ in Algorithm \ref{alg:expanding_nn} provides a strategy for increasing the network complexity. Typical choices include $b(k) = k + 1$ (linear growth) or $b(k) = k + m$ for some fixed $m$ (fixed-step growth). Below $b^{-1}$ denotes the generalized inverse of $b$ defined by $b^{-1}(t):=\min\{k\in\mathbb N^+: b(k)>t\}$. For each $n$, the approximate solution $f_{n,k}^{\delta}$ is obtained by solving a minimization problem over a bounded compact set $X_{n,r(n)}$, which guarantees the existence of a regularized solution. The set $X_{n,r(n)}$ expands with increasing $n$, so the algorithm yields wider NNs until it is terminated by  the discrepancy principle. 

Now, we give the convergence of Algorithm \ref{alg:expanding_nn}. If the upper bound of $c_{\rho}(f^{\dagger})$ for $f^\dagger$ is known and $r(n)$ satisfies mild conditions, then the sequence generated by Algorithm \ref{alg:expanding_nn} terminates after a finite number of steps and converges to $f^\dagger$. Moreover, we derive asymptotic estimates on the stopping index and the number of neurons, under H\"older continuity of $A$.
\begin{definition}
\label{Holdercontinuous}
The operator $A$ is said to be locally H\"older continuous at $f^\dagger$ of order
$\theta\in(0,1]$ with respect to the $H^1(\Omega)$ if there exist
$L_A>0$ and $\gamma>0$ such that
\[
  \|A(f)-A(f^\dagger)\|_{\mathcal Y}
  \le L_A \, \|f-f^\dagger\|_{H^1(\Omega)}^{\theta},
  \quad\forall f \text{ with } \|f-f^\dagger\|_{H^1(\Omega)}<\gamma.
\]
\end{definition}

\begin{theorem}
\label{main}
 Let $A:L^2(\Omega)\to\mathcal Y$ be continuous and injective, and let $f^\dagger\in\mathcal B_1$ satisfy $A(f^\dagger)=g$. Let $\rho$ be a representing probability measure of $f^\dagger$ with $c_\rho(f^\dagger)<\infty$, and suppose the noisy data $g^{\delta}$ satisfy \eqref{NoiseLevel}. Let $\{r(n)\}_{n\ge1}\subset\mathbb{R}_+$ be a nondecreasing sequence such that
$r(n)\nearrow B$ as $n\to\infty$ with $
B>c_\rho(f^\dagger)$.
Then the following statements hold.
\begin{itemize}
\item For any $\delta >0$, the iteration process in Algorithm \ref{alg:expanding_nn} terminates after a finite number $k(\delta)$ steps with finitely many $n(\delta)$ neurons. Moreover, if $A$ is locally H\"older at $f^\dagger$ of order $\theta$, then there exists $C>0$, independent of $\delta$, such that
\begin{equation*}
    k(\delta) \le b^{-1}\!\big(C\,\delta^{-2/\theta}\big)\quad\mbox{and}\quad
n(\delta) = \mathcal{O}\big(\delta^{-2/\theta}\big),
\quad \text{as } \delta\to0.
\end{equation*}
\item The approximate solutions $f_{n(\delta),k(\delta)}^{\delta}$ converge to $f^{\dagger}$ in $C(\overline\Omega)$ as $\delta \to 0$.
\end{itemize}
 \end{theorem}

\begin{proof}
For each fixed $n$, Lemma~\ref{lemma compact set} shows that $X_{n,r(n)}$ is compact in $L^2(\Omega)$. Since $A:L^2(\Omega)\to\mathcal Y$ is continuous, the functional $J(f)=\|A(f)-g^{\delta}\|_{\mathcal{Y}}$ is continuous on $X_{n,r(n)}$. Hence, by Weierstrass theorem, $J$ attains its minimum on $X_{n,r(n)}$.
Since $r(n)\nearrow B$ and $B>c_\rho(f^\dagger)$, there exists $N_1$ such that $r(n)>c_\rho(f^\dagger)$ for all $n\ge N_1$. By the continuity of $A$ at $f^{\dagger}$ with respect to $L^2(\Omega)$-norm and the continuous embedding $H^1(\Omega)\hookrightarrow L^2(\Omega)$, for every $\delta>0$, there exists $\varepsilon_1(\delta)>0$ such that
\[
\|A(f)-A(f^\dagger)\|_{\mathcal Y}\le (\tau-1)\delta
\quad\text{whenever}\quad
\|f-f^\dagger\|_{H^1(\Omega)}\le \varepsilon_1(\delta).
\]
Let $N(\delta)=\max\{N_1,(\frac{C(\Omega,d)c_\rho(f^\dagger)}{\varepsilon_1(\delta)})^2,b(k_0)\}$. For every integer $n>N(\delta)$, Proposition~\ref{Approximtion theorem} provides
$f_n\in X_{n,c_\rho(f^\dagger)}\subset X_{n,r(n)}$ such that $\|f^\dagger-f_n\|_{H^1(\Omega)}<\varepsilon_1(\delta)$. 
Define $k^{*}:=\min\{k\in \mathbb{N}^{+}|b(k)>N(\delta)\}$ and $\overline N:=b(k^*)$. Then $\overline N>N(\delta)$. Since $N(\delta) \ge b(k_0)$ and $b$ is strictly increasing, we have $k^*>k_0$. By the minimality of $k^*$, $b(k^*-1)\le N(\delta)<b(k^*)=\overline N$. Let $f_{\overline N}\in X_{\overline N,c_\rho(f^\dagger)}
\subset X_{\overline N,r(\overline N)}$ be the approximating network with width $n=\overline N$. Since $f_{\overline N,k^*}^{\delta}$ is a minimizer over $X_{\overline N,r(\overline N)}$, we have
\begin{align}
J_{\overline{N},k^*}^{\delta} &=\left\|A(f_{\overline{N},k^*}^{\delta})-g^{\delta}\right\|_{\mathcal{Y}} \le \left\|A(f_{\overline{N}})-g^{\delta}\right\|_{\mathcal{Y}}  \nonumber \\
\label{JNlestaudel}
&\le \left\|A(f_{\overline{N}})-A(f^\dagger)\right\|_{\mathcal{Y}} + \left\|A(f^\dagger)-g^{\delta}\right\|_{\mathcal{Y}} \le(\tau-1)\delta + \left\|g-g^{\delta}\right\|_{\mathcal{Y}} \le \tau\delta.
\end{align}
Therefore, the discrepancy principle is satisfied no later than $k^*$. By the definitions of $n(\delta)$ and $k(\delta)$ in Algorithm~\ref{alg:expanding_nn}, it follows that $k(\delta)\le k^*$ and $n(\delta)\le \overline N=b(k^*)$. Since $k^*-1\ge k_0$, the growth condition on $b$ and the minimality of $k^*$ yield $n(\delta)\le b(k^*)\le C_b b(k^*-1) \le C_b N(\delta)$. Since $\varepsilon_1(\delta)\to0$ as $\delta\to0$, there exists $C>0$ independent of $\delta$, such that $N(\delta)\le C(1/{\varepsilon_1(\delta)})^2$ for all sufficiently small $\delta$. Hence $n(\delta)=\mathcal O\left((1/\varepsilon_1(\delta))^2\right)$. Moreover, since $b^{-1}$ is nondecreasing, we obtain
$k(\delta)\le k^*
= b^{-1}(N(\delta))
\le
b^{-1}\!\left(C\left(1/{\varepsilon_1(\delta)}\right)^2\right)$. If $A$ is locally H\"older of order $\theta$ at $f^\dagger$, then, for sufficiently small $\delta$, we may choose $\varepsilon_1(\delta)=
((\tau-1)\delta /L_A)^{1/\theta}
<\gamma$. 
Therefore, $n(\delta)=\mathcal O(\delta^{-2/\theta}),k(\delta)\le b^{-1}\!\left(C\delta^{-2/\theta}\right)$.

It remains to prove convergence. Suppose, to the contrary, that $f_{n(\delta),k(\delta)}^\delta$ does not converge to $f^\dagger$ in $C(\overline \Omega)$. Then there exist $\varepsilon>0$ and a sequence $\delta_\ell\searrow0$ such that, with $f_\ell:=f_{n(\delta_\ell),k(\delta_\ell)}^{\delta_\ell}$,
\begin{equation}
    \label{contr_InEq1}
    \left\|
f_\ell-f^\dagger
\right\|_{C(\overline\Omega)}
\ge \varepsilon .
\end{equation}
For each $\ell$, $f_\ell\in X_{n(\delta_\ell),r(n(\delta_\ell))}$, denote its outer coefficients by $\{a_j^{(\ell)}\}_{j=1}^{n(\delta_\ell)}$. The definition of $X_{n,r}$ gives $q_\ell:=
{n(\delta_\ell)}^{-1}
\sum_{j=1}^{n(\delta_\ell)}
|a_j^{(\ell)}|
\le r(n(\delta_\ell))
\le B<\infty$. By Lemma~\ref{lem:barron_compact_lsc}, there exists a subsequence, denoted by $\{f_{\ell_i}\}$, and some $f^*\in C(\overline\Omega)\cap\mathcal B_1$ such that $f_{\ell_i}\to f^*$ in $C(\overline{\Omega})$, and hence also in $L^2(\Omega)$. By the discrepancy principle and the noise assumption,
\[
\|A(f_{\ell_i})-g\|_{\mathcal Y}
\le
\|A(f_{\ell_i})-g^{\delta_{\ell_i}}\|_{\mathcal Y}
+
\|g^{\delta_{\ell_i}}-g\|_{\mathcal Y}
\le
(\tau+1)\delta_{\ell_i}\to0.
\]
The continuity of $A:L^2(\Omega)\to\mathcal Y$ therefore gives $A(f^*)=g$. Since $A$ is injective and $A(f^\dagger)=g$, we obtain $f^*=f^\dagger$ in $L^2(\Omega)$. As both $f^*$ and $f^\dagger$ are continuous on $\overline\Omega$, they coincide on $\overline\Omega$. Consequently, $f_{\ell_i}\to f^\dagger$ in $C(\overline\Omega)$, which contradicts \eqref{contr_InEq1}. Hence $f_{n(\delta),k(\delta)}^\delta \to f^\dagger$ in $C(\overline\Omega)$ as $\delta\to0$.
\end{proof}

\begin{remark}
For typical choices of $b=b(k)$ we obtain explicit rates. For instance,
if $b(k) \asymp k^{p}$
for some $p>0$, then $k(\delta) = \mathcal{O}\big(\delta^{-2/(p\theta)}\big)$; 
if $b(k) \asymp e^{q k}$ for some $q>0$, then $k(\delta) = \mathcal{O}\big(\log(1/\delta)\big)$. 
\end{remark}

\subsection{Case 2: without knowledge of \texorpdfstring{$c_{\rho}(f^{\dagger})$}.}

Now, we consider the case without any prior knowledge of the quantity $c_{\rho}(f^{\dagger})$, for which we use a penalty term $\mathcal{R}$ in the optimization problem. It ensures the uniform boundedness of the minimizer throughout the iterations.

\begin{algorithm}[H]
\caption{ENN2: a modified expanding neural network method for solving equation (\ref{1})}
\label{alg:modified_enn}
\begin{algorithmic}[1]
\REQUIRE $b(k)$, $r(n)$, $\tau>1$, $c_0>0$, $\theta>0$, $\mathcal{R}$ and $\delta$
\ENSURE Reconstruction $f_{n(\delta),k(\delta)}^{\delta}$, $n(\delta)$, $k(\delta)$

\STATE \textbf{Initialization:} $k \gets 1$, $n \gets b(1)$
\STATE Set $\beta_n = c_0 n^{-\theta/2}$ and compute
  \begin{equation}
  \label{OPT-ENN2}
      f_{n,k}^{\delta}
      =
      \arg\min\limits_{f \in X_{n,r(n)}}
      \Big( \left\|A(f)-g^{\delta}\right\|_{\mathcal{Y}} + \beta_n \mathcal{R}(f) \Big),
      \quad
      J_{n,k}^{\delta}
      = \left\|A(f_{n,k}^{\delta})-g^{\delta}\right\|_{\mathcal{Y}}.
  \end{equation}

\WHILE{$J_{n,k}^{\delta} > \tau\delta$}
    \STATE $k \gets k+1$
    \STATE $n \gets b(k)$
    \STATE Update $\beta_n = c_0 n^{-\theta/2}$
    \STATE Update $f_{n,k}^{\delta}$ and $J_{n,k}^{\delta}$ by resolving \eqref{OPT-ENN2}
          on the expanded class of neural networks $X_{n,r(n)}$
\ENDWHILE

\STATE Set $n(\delta)\gets n$, $k(\delta)\gets k$
\RETURN $f_{n(\delta),k(\delta)}^{\delta}$, $n(\delta)$, $k(\delta)$
\end{algorithmic}
\end{algorithm}

The convergence of Algorithm \ref{alg:modified_enn} is given in Theorem~\ref{main2}: the algorithm can produce a sequence of regularized solutions converging to $f^\dagger$, without additional conditions on $r(n)$. 

\begin{theorem}
\label{main2}
Let $f^\dagger\in\mathcal B_1$ satisfy $A(f^\dagger)=g$, and let the noisy data $g^{\delta}$ satisfy \eqref{NoiseLevel}. Suppose that $A$ is continuous, injective, and locally H\"older at $f^{\dagger}$ of order $\theta$. Let $r:\mathbb N\to(0,\infty)$ be nondecreasing and satisfy $r(n)\to+\infty$ as $n\to\infty$. Moreover, one of the following holds.
\begin{itemize}
\item[\textnormal{(i)}] 
$\mathcal{R}=\mathcal{R}_H$ is weakly lower semicontinuous in $H^1(\Omega)$.
It is also locally H\"older continuous at $f^\dagger$ of order $\theta$, with
constants $L_{\mathcal{R}_H}>0$ and $\gamma_2>0$. Moreover, there exists
$C_{\mathcal{R}_H}>0$ such that $C_{\mathcal{R}_H}\|f\|_{H^1(\Omega)}
\le
\mathcal{R}_H(f)
<+\infty$.
\item[\textnormal{(ii)}] \(\mathcal{R}=\mathcal{R}_B\), where for a parametrized network $f_n(\mathbf x)
=
\frac1n\sum_{j=1}^n
a_j(\boldsymbol b_j^T\mathbf x+c_j)_+$ with $\|\boldsymbol b_j\|_1+|c_j|=1$, we set $\mathcal R_B(f_n):=c_{\rho_n}(f_n)=
\frac1n\sum_{j=1}^n |a_j|$. Here \(\rho_n\) is the empirical measure associated with the chosen parametrization. 
\end{itemize}
Then, the following statements hold:
\begin{itemize}
\item For any $\delta >0$ the iteration process in Algorithm \ref{alg:modified_enn} terminates after finitely many steps $k(\delta)$ with finitely many neurons $n(\delta)$. Moreover, there exists $C>0$, independent of $\delta$, such that
\begin{equation}
    \label{desire_estimate_kn}
    k(\delta) \le b^{-1}\!\big(C\,\delta^{-2/\theta}\big),
\qquad
n(\delta) = \mathcal{O}\big(\delta^{-2/\theta}\big),
\quad \text{as } \delta\to0.
\end{equation}

\item 
The approximate solutions generated by Algorithm~\ref{alg:modified_enn}
satisfy as $\delta\to0$:
\begin{align*}
&f_{n(\delta),k(\delta)}^\delta \to f^\dagger
\quad\text{in }L^2(\Omega),
\qquad \text{if } \mathcal R=\mathcal R_H,\\
&f_{n(\delta),k(\delta)}^\delta \to f^\dagger
\quad\text{in }C(\overline\Omega),
\qquad \text{if } \mathcal R=\mathcal R_B.
\end{align*}
\end{itemize}
\end{theorem}

\begin{proof}
Fix $n\in\mathbb N$ and $\delta>0$. Let
$\mathcal{L}_{n}^{\delta}(f)
:=\|A(f)-g^{\delta}\|_{\mathcal{Y}}+\beta_n \mathcal{R}(f)$ and $\beta_n = c_0 n^{-\frac{\theta}{2}}$. For $r(n)\in(0,\infty)$, we first show the existence of a minimizer for \eqref{OPT-ENN2}.
Let $m_n^\delta := \inf_{f\in X_{n,r(n)}} \mathcal L_n^\delta(f)$, and choose a minimizing sequence $\{f_{n,\ell}\}_{\ell\in\mathbb N}\subset X_{n,r(n)}$ such that $\mathcal L_n^\delta(f_{n,\ell})\to m_n^\delta$ as $\ell\to\infty$. Pick any $f_{n,0}\in X_{n,r(n)}$ with $\mathcal L_n^\delta(f_{n,0})<\infty$ (e.g. $f_{n,0}=0$), and without loss of generality, we assume $\mathcal L_n^\delta(f_{n,\ell})\le \mathcal L_n^\delta(f_{n,0})+1$ for all $\ell\in \mathbb N$.
This directly implies $\beta_n \mathcal R(f_{n,\ell}) \le \mathcal L_n^\delta(f_{n,\ell})\le \mathcal L_n^\delta(f_{n,0})+1$, and hence
\[
\mathcal R(f_{n,\ell}) \le \frac{\mathcal L_n^\delta(f_{n,0})+1}{\beta_n}
= \frac{\mathcal L_n^\delta(f_{n,0})+1}{c_0}\, n^{\theta/2}
=: C_{n,\delta} <\infty.\]
We discuss the two choices of $\mathcal R$ separately.

\noindent\emph{Case A: $\mathcal R=\mathcal R_H$.} Since $\mathcal R_H(f)\ge C_{\mathcal R_H}\|f\|_{H^1(\Omega)}$, the estimate $\mathcal R_H(f_{n,\ell})\le C_{n,\delta}$ implies that $\{f_{n,\ell}\}$ is bounded in $H^1(\Omega)$. Hence, there exists a subsequence $\{f_{n,\ell_i}\}$ and some $f_n^*\in H^1(\Omega)$ such that $f_{n,\ell_i}\rightharpoonup f_n^*$ in $H^1(\Omega)$, and $f_{n,\ell_i}\to f_n^*$ in $L^2(\Omega)$.
Since $X_{n,r(n)}$ is closed in $L^2(\Omega)$, we have $f_n^*\in X_{n,r(n)}$. By the continuity of $A:L^2(\Omega)\to\mathcal Y$ and the weak lower
semicontinuity of $\mathcal R_H$, we have $\mathcal L_n^\delta(f_n^*)
\le \liminf_{i \to\infty} \mathcal L_n^\delta(f_{n,\ell_i})
= m_n^\delta.$
Hence $f_n^*$ is a minimizer of $\mathcal L_n^\delta$ over $X_{n,r(n)}$.

\noindent\emph{Case B: $\mathcal R=\mathcal R_B$.}
Write $f_{n,\ell}(\mathbf x)=\frac1n\sum_{j=1}^n a_j^{(\ell)}((\boldsymbol b_j^{(\ell)})^T\mathbf x+c_j^{(\ell)})_+$ with
$\|\boldsymbol b_j^{(\ell)}\|_1+|c_j^{(\ell)}|=1$.
From $\mathcal R_B(f_{n,\ell})\le C_{n,\delta}$ and the definition of $\mathcal{R}_B$, we obtain
\[
|a_j^{(\ell)}|
\le \sum_{i=1}^{n}|a_i^{(\ell)}|= n\cdot \,\frac{1}{n}\sum_{i=1}^{n}|a_i^{(\ell)}|\le  n\,\mathcal R_B(f_{n,\ell})
\le nC_{n,\delta}
=:R_{n,\delta},
\qquad j=1,\dots,n.
\]
Define the radius $\bar r_{n,\delta}:=\min\{r(n),R_{n,\delta}\}\in(0,\infty).$ Then $f_{n,\ell}\in X_{n,\bar r_{n,\delta}}$ for all $\ell$. By Lemma~\ref{lemma compact set} and its proof, there exist a subsequence $\{f_{n,\ell_i}\}$ and some $f_n^*\in X_{n,\bar r_{n,\delta}}\subset X_{n,r(n)}$ such that $f_{n,\ell_i}\to f_n^*$ in $C(\overline \Omega)$ with the associated parameters $\{(a_j^{(\ell_i)},\boldsymbol b_j^{(\ell_i)},c_j^{(\ell_i)})\}$ converging as well. Hence $\mathcal R_B(f_{n,\ell_i})\to \mathcal{R}_B(f_n^*)$.
Moreover, since $A:L^2(\Omega)\to\mathcal Y$ is continuous, it follows that $\mathcal L_n^\delta(f_n^*)
=
\lim_{i\to\infty}\mathcal L_n^\delta(f_{n,\ell_i})
=
\inf_{f\in X_{n,r(n)}}\mathcal L_n^\delta(f).$
Thus $f_n^*$ is a minimizer of $\mathcal L_n^\delta$ over $X_{n,r(n)}$.

Next, we show that Algorithm~\ref{alg:modified_enn} terminates after finitely many iterations for every $\delta>0$. Since $f^\dagger\in\mathcal B_1$, fix a representing probability
measure $\rho$ of $f^\dagger$ such that $c_\rho(f^\dagger)<\infty.$ Define $r^{-1}(t):=\min\{n\in\mathbb N^+:r(n)\geq t\}$, which is finite since $r(n)\to+\infty$. 

For the case $\mathcal R=\mathcal R_H$, let
$\eta_H(\delta):=\min\{((\tau-1)\delta/[2(L_A+L_{\mathcal R_H})])^{1/\theta},\gamma_1,\gamma_2\}$ and
\[
N_H(\delta):=\max\left\{r^{-1}(c_\rho(f^\dagger)),\,b(k_0),\,c_0^{2/\theta},\,\left(\frac{C(\Omega,d)c_\rho(f^\dagger)}{\eta_H(\delta)}\right)^2,\,\left(\frac{2c_0\mathcal R_H(f^\dagger)}{(\tau-1)\delta}\right)^{2/\theta}\right\}.
\]
By Proposition~\ref{Approximtion theorem}, for every integer $n>N_H(\delta)$, there exists a two-layer NN
$f_n\in X_{n,c_\rho(f^\dagger)}\subset X_{n,r(n)}$ such that $\|f_n-f^\dagger\|_{H^1(\Omega)}\le
C(\Omega,d)c_{\rho}(f^\dagger)n^{-1/2}<\eta_H(\delta)$. Since $c_0n^{-\theta/2}\le1$, the local H\"older continuity of $A$ and $\mathcal R_H$, together with the definition of
$N_H(\delta)$, gives
\begin{align*}
\mathcal L_n^\delta(f_n)&=\|A(f_n)-g^{\delta}\|_{\mathcal Y}
+c_0n^{-\theta/2}\mathcal R_H(f_n)\\
&\le
L_A\|f_n-f^\dagger\|_{H^1(\Omega)}^\theta+\|g-g^{\delta}\|_{\mathcal{Y}}
+c_0n^{-\theta/2}
\left[
\mathcal R_H(f^\dagger)
+L_{\mathcal R_H}\|f_n-f^\dagger\|_{H^1(\Omega)}^\theta
\right] \\
&\le
(L_A+L_{\mathcal R_H})\eta_H(\delta)^\theta+\delta
+c_0n^{-\theta/2}\mathcal R_H(f^\dagger)
\le\tau\delta.
\end{align*}

For the case $\mathcal R=\mathcal R_B$, let
$\eta_B(\delta):=\min\{((\tau-1)\delta/(2L_A))^{1/\theta},\gamma_1\}$ and
\[
N_B(\delta):=\max\left\{r^{-1}(c_\rho(f^\dagger)),\,b(k_0),\,\left(\frac{C(\Omega,d)c_\rho(f^\dagger)}{\eta_B(\delta)}\right)^2,\,\left(\frac{2c_0c_\rho(f^\dagger)}{(\tau-1)\delta}\right)^{2/\theta}\right\}.
\]
 By Proposition~\ref{Approximtion theorem}, for every integer $n>N_B(\delta)$, there exists a two-layer NN $f_n\in X_{n,c_\rho(f^\dagger)}\subset X_{n,r(n)}$ and an associated empirical measure $\rho_n$ such that $\|f_n-f^\dagger\|_{H^1(\Omega)}<\eta_B(\delta)$ and
$\mathcal R_B(f_n)=c_{\rho_n}(f_n)\le c_\rho(f^\dagger)$. Hence
\begin{align*}
\mathcal L_n^\delta(f_n)
&\le
L_A\|f_n-f^\dagger\|_{H^1(\Omega)}^\theta
+\|g-g^\delta\|_{\mathcal Y}
+c_0n^{-\theta/2}\mathcal R_B(f_n) \\
&\le
L_A\eta_B(\delta)^\theta
+\delta
+c_0n^{-\theta/2}c_\rho(f^\dagger)
\le
\tau\delta.
\end{align*}

Let $N(\delta):=N_H(\delta)$ if $\mathcal R=\mathcal R_H$, and $N(\delta):=N_B(\delta)$ if $\mathcal R=\mathcal R_B$. In either case, $N(\delta)=\mathcal O(\delta^{-2/\theta})$. Hence there exists a constant $C_{\max}>0$, independent of $\delta$, such that $N(\delta)\le C_{\max}\delta^{-2/\theta}$ for all sufficiently small $\delta$. Let $k^*:=\min\{k\in\mathbb N^+: b(k)>N(\delta)\}$ and $\overline N:=b(k^*)$. Then $\overline N>N(\delta)$,
$k^*>k_0$, and $b(k^*-1)\le N(\delta)<b(k^*)=\overline N.$ Let $f_{\overline N}\in
X_{\overline N,c_\rho(f^\dagger)}\subset X_{\overline N,r(\overline N)}$ be the corresponding approximating network constructed above. By the minimizing property of $f_{\overline N,k^*}^\delta$ and the preceding estimates,
\[
J_{\overline N,k^*}^\delta
\le
\mathcal L_{\overline N}^\delta
(f_{\overline N,k^*}^\delta)
\le
\mathcal L_{\overline N}^\delta(f_{\overline N})
\le
\tau\delta.
\]
Therefore, the discrepancy principle is satisfied no later than
iteration $k^*$. By the stopping rule in
Algorithm~\ref{alg:modified_enn}, it follows that
$k(\delta)\le k^*$ and $n(\delta)\le\overline N=b(k^*)$. Since $k^*-1\geq k_0$, the growth condition on $b$ and the minimality of $k^*$ yield $n(\delta)
\le
b(k^*)
\le
C_bb(k^*-1)
\le
C_bN(\delta)
=
\mathcal O(\delta^{-2/\theta}).$ Moreover, since $b^{-1}$ is nondecreasing, we have $k(\delta)
\le
k^*
=
b^{-1}(N(\delta))
\le
b^{-1}\!\left(C_{\max}\delta^{-2/\theta}\right).$
This proves \eqref{desire_estimate_kn} and the finite termination of Algorithm~\ref{alg:modified_enn}.

Finally, we prove the regularizing property of the algorithm. Suppose, to the contrary, that the asserted convergence fails. Then there exist $\varepsilon>0$ and a sequence $\delta_\ell\searrow0$ such that, with
$f_\ell:=f_{n(\delta_\ell),k(\delta_\ell)}^{\delta_\ell}$ and $n_\ell:=n(\delta_\ell)$,
\begin{equation}
    \label{contradictRBH}
    \|f_\ell-f^\dagger\|_{L^2(\Omega)}\geq\varepsilon
\quad\text{if }\mathcal R=\mathcal R_H,
\qquad
\|f_\ell-f^\dagger\|_{C(\overline\Omega)}\geq\varepsilon
\quad\text{if }\mathcal R=\mathcal R_B.
\end{equation}
Passing to a subsequence if necessary, either $\{n_\ell\}$ is bounded or $n_\ell \to\infty$.

We first consider the case where $\{n_\ell\}$ is bounded. Then, there exists $N_0\in\mathbb N$ such that $n_\ell\le N_0$ for all $\ell$, and hence
$f_\ell\in\bigcup_{n=1}^{N_0}X_{n,r(n)}$. Since each
$X_{n,r(n)}$ is compact in $C(\overline\Omega)$ by Lemma~\ref{lemma compact set}, this finite union is also compact in $C(\overline\Omega)$. Therefore, after passing to a further subsequence, there exists $f^*\in C(\overline\Omega)$ such that $f_\ell\to f^*$ in $C(\overline\Omega)$, and hence also in $L^2(\Omega)$. By the discrepancy principle and the noise assumption,
\[
\|A(f_\ell)-g\|_{\mathcal Y}
\le
\|A(f_\ell)-g^{\delta_\ell}\|_{\mathcal Y}
+\|g^{\delta_\ell}-g\|_{\mathcal Y}
\le(\tau+1)\delta_\ell\to0.
\]
The continuity of $A:L^2(\Omega)\to\mathcal Y$ gives
$A(f^*)=g$. Since $A$ is injective and $A(f^\dagger)=g$, we obtain $f^*=f^\dagger$ in $L^2(\Omega)$. As both functions are continuous on $\overline\Omega$, they coincide on $\overline\Omega$. Thus $f_\ell\to f^\dagger$ in $C(\overline\Omega)$, which yields a contradiction in either case.

It remains to consider the case $n_\ell\to \infty$. By Proposition~\ref{Approximtion theorem}, for each $\ell$, there exists a two-layer NN $f_{n_\ell}\in X_{n_\ell,c_\rho(f^\dagger)}$ such that $c_{\rho_{n_\ell}}(f_{n_\ell})\le c_\rho(f^\dagger)$ and 
\begin{equation}
    \label{appresult_ndelta}
    \left\|f^{\dagger}-f_{n_\ell}\right\|_{H^1(\Omega)}\le C(\Omega, d)c_\rho(f^\dagger)n_\ell^{-1/2}.
\end{equation}
Since $n_\ell\to\infty$ and $r(n)\to\infty$, we have $f_{n_\ell}\in X_{n_\ell,r(n_\ell)}$ for all sufficiently large $\ell$. Moreover, for all sufficiently large $\ell$, the local H\"older continuity assumptions on $A$ and, when $\mathcal R=\mathcal R_H$, on $\mathcal R_H$
are applicable to $f_{n_\ell}$.

For $\mathcal{R} = \mathcal{R}_H$, by the minimizing property of $f_\ell$, the inequality $\mathcal{R}_H(f)\ge C_{\mathcal{R}_H}\|f\|_{H^1(\Omega)}$, \eqref{appresult_ndelta}, and the local H\"older continuity of $A$ and $\mathcal R_H$, we have
\begin{align*}
&c_0C_{\mathcal R_H}n_\ell^{-\theta/2}
\|f_\ell\|_{H^1(\Omega)}
\le
\mathcal L_{n_\ell}^{\delta_\ell}(f_\ell)
\le
\mathcal L_{n_\ell}^{\delta_\ell}(f_{n_\ell}) =\| A(f_{n_\ell}) - g^{\delta_\ell} \|_{\mathcal{Y}} + c_0 \, n_\ell^{-\theta/2} \, \mathcal{R}_H(f_{n_\ell})\\
&\qquad \quad\le
\left[
L_A
\left(C(\Omega,d)c_\rho(f^\dagger)\right)^\theta
+c_0\mathcal R_H(f^\dagger)
\right]n_\ell^{-\theta/2}
+\delta_\ell +
c_0L_{\mathcal R_H}
\left(C(\Omega,d)c_\rho(f^\dagger)\right)^\theta
n_\ell^{-\theta}.
\end{align*}
Since $n_\ell \le C_b N(\delta_\ell)$, and $ N(\delta_\ell)\le C_{\max}\delta_\ell^{-2/\theta}$ for all sufficiently small $\delta_\ell$, we have $\delta_\ell (n_\ell)^{\theta/2} \le (C_bC_{\max})^{\theta/2}$. 
Multiplying the preceding estimate by $n_\ell^{\theta/2}$ shows that $\{f_\ell\}$ is bounded in
$H^1(\Omega)$. By Lemma~\ref{convergesubsq}, there exists a subsequence denoted by $\{f_{\ell_i}\}$, and some $f^*\in L^2(\Omega)$ such that $f_{\ell_i}\to f^*$ in $L^2(\Omega)$. The discrepancy principle, the noise assumption, and the continuity of $A$ imply $A(f^*)=g$. By the injectivity of $A$, $f^*=f^\dagger$. Hence $f_{\ell_i} \to f^\dagger$ in $L^2(\Omega)$, contradicting \eqref{contradictRBH}.

For $\mathcal R=\mathcal R_B$, using \eqref{appresult_ndelta}, the minimizing property of $f_\ell$, the local H\"older continuity of $A$, $\mathcal R_B(f_{n_\ell})=c_{\rho_{n_\ell}}(f_{n_\ell})\le c_\rho(f^\dagger)$, and the bound $\delta_\ell (n_\ell)^{\theta/2} \le (C_bC_{\max})^{\theta/2}$, we obtain
\begin{align*}
    \mathcal R_B(f_\ell)&
\le 
c_0^{-1}n_\ell^{\theta/2}\mathcal L_{n_\ell}^{\delta_\ell}(f_\ell)
\le
c_0^{-1}n_\ell^{\theta/2}\mathcal L_{n_\ell}^{\delta_\ell}(f_{n_\ell})\le
c_0^{-1}n_\ell^{\theta/2}
\|A(f_{n_\ell})-g^{\delta_\ell}\|_{\mathcal Y}
+\mathcal R_B(f_{n_\ell})\\&\le
c_0^{-1}
\left[
L_A(C(\Omega,d)c_\rho(f^\dagger))^\theta
+
(C_bC_{\max})^{\theta/{2}}
\right]
+c_\rho(f^\dagger).
\end{align*}
Hence the sequence $\{\mathcal R_B(f_\ell)\}$ is uniformly bounded for all sufficiently small $\delta_\ell$. Let $\{a_j^{(\ell)}\}_{j=1}^{n_\ell}$ be the outer coefficients of $f_\ell$, then $q_\ell:=n_\ell^{-1}\sum_{j=1}^{n_\ell}|a_j^{(\ell)}|=\mathcal R_B(f_\ell)$ is also uniformly bounded. By Lemma~\ref{lem:barron_compact_lsc}, there exists a subsequence, denoted by $\{f_{\ell_i}\}$, and some $f^*\in C(\overline \Omega)\cap\mathcal B_1$ such that $f_{\ell_i}\to f^*$ in $C(\overline \Omega)$. As above, the limit must be $f^\dagger$. Hence, $f_{\ell_i}
\to f^\dagger$ in $C(\overline \Omega)$, which contradicts \eqref{contradictRBH}.

Therefore,
$f_{n(\delta),k(\delta)}^\delta\to f^\dagger$ in
$L^2(\Omega)$ if $\mathcal R=\mathcal R_H$, and
$f_{n(\delta),k(\delta)}^\delta\to f^\dagger$ in
$C(\overline\Omega)$ if $\mathcal R=\mathcal R_B$, as
$\delta\to0$.
\end{proof}

Thus, Algorithms \ref{alg:expanding_nn}--\ref{alg:modified_enn} provide a strategy for constructing NNs that serve as regularization solutions.
\begin{itemize}
    \item \textbf{Answer to Q1:} The proposed algorithms construct neural network approximations that yield stable regularized solutions for ill-posed inverse problems. The discrepancy principle provides an \textit{a posteriori} stopping rule based on the noisy data, and the convergence results show that the selected approximations converge to the exact solution as the noise level tends to zero. 
    \item \textbf{Answer to Q2:} The theorems quantify how the sufficient network width depends on the noise level. When the Barron-space approximation error is of order \(\mathcal O(n^{-1/2})\) and the forward operator is locally H\"older continuous with exponent \(\theta\), a sufficient width can be chosen as $n(\delta)=\mathcal O(\delta^{-2/\theta})$. This shows that the network size should be coupled with the noise level rather than enlarged independently. The numerical results in Section~\ref{simulation} further illustrate that, for noisy data, the reconstruction error need not decrease monotonically as the number of neurons increases. 
    \item \textbf{Answer to Q3:} The expanding neural network methodology increases the number of neurons gradually and stops at the first architecture along the prescribed expansion path that satisfies the discrepancy principle. In this way, the method adaptively selects the network size from the noisy data while avoiding unnecessary architectural complexity.
\end{itemize}

\section{A neural network-based Tikhonov regularization scheme}
\label{Tikhonov}

In this section, we provide a method in the framework of variational regularization, which incorporates an explicit parameter $\alpha$ so as to provide more direct and stable control over the solution. We propose a neural network-based Tikhonov regularization scheme (Tikhonov-NN) to address problem \eqref{1}. We establish convergence of the regularized solutions and derive explicit convergence rates that characterize the efficiency of the method. For consistency with the previous sections, we use the squared norm as the data fidelity term throughout this section. The extension to a general data fidelity functional $\mathcal D$ is presented in Appendix~\ref{app:generalD}. The idea is to seek two layer neural networks as the regularized solution by minimizing the functional
\begin{equation}
\label{Tikfunc}
f_{n,\alpha}^{\delta}\in \arg\min_{f\in X_{n,\infty}}\mathcal{T}_\alpha(f,g^{\delta}),\quad  \mathcal{T}_\alpha(f,g^{\delta}):=\|A(f)-g^{\delta}\|^2_{\mathcal Y}+\alpha \mathcal{R}_n(f),
\end{equation}
where $X_{n,\infty}$ is defined in \eqref{eq:Xninfty_union}, $\alpha>0$ is the regularization parameter and $\mathcal{R}_n(f):X_{n,\infty}\to\mathbb{R}_+$ is the regularization term, given by
\begin{equation}
\label{penalty}
\mathcal{R}_n(f):= \varphi(c_{\rho_n}(f)), \quad  c_{\rho_n}(f)= \frac{1}{n} \sum_{i=1}^n\left|a_i\right|(\left\|\boldsymbol{b_i}\right\|_1+\left|c_i\right|)=\frac{1}{n} \sum_{i=1}^n\left|a_i\right|,
\end{equation}
where \(\rho_n\) is the associated empirical measure and $\varphi:[0,\infty)\to[0,\infty)$ is continuous and nondecreasing, with $\varphi(t)\to\infty$ as $t\to\infty$.

Now we show that Tikhonov-NN method is a well-posed and convergent scheme: one can construct a sequence that converges to an $\mathcal{R}$-minimizing solution defined in \eqref{deffdagger}, as $n \to \infty$ and $\delta \to 0$.
\begin{theorem}[Well-posedness]
\label{well-posedness}
Let $A:L^2(\Omega)\to\mathcal Y$ be continuous. The following assertions hold.
    \begin{enumerate}
        \item[(a)] For each $n\in\mathbb{N}$, every $g \in \mathcal{Y}$ and every $\alpha>0$, there exists a minimizer $f_n^*$ of $\mathcal{T}_\alpha(f,g)$ in $X_{n,\infty}$.
        
        \item[(b)] Fix $n\in\mathbb{N}$ and $\alpha>0$. If $g_k \to g$ in $\mathcal{Y}$ and $f_{n,k} \in \arg\min_{f\in X_{n,\infty}}\mathcal{T}_\alpha(f,g_k)$, then there exists a subsequence $\{f_{n,k_{\ell}}\}_{\ell=1}^\infty$ and some $f_n^{*}\in \arg\min_{f\in X_{n,\infty}}\mathcal{T}_\alpha(f,g)$ such that $f_{n,k_{\ell}}\to f^{*}_n$ in $C(\overline \Omega)$.
        \item[(c)] Assume that $A(f)=g$ has at least one solution in $\mathcal B_1$, and $A$ is locally H\"older continuous of order $\theta\in(0,1]$ at every $\bar f\in\mathcal B_1$ satisfying $A(\bar f)=g$. Let $\{g_k\}_{k=1}^{\infty}\subset\mathcal{Y}$ satisfy  $\|g-g_k\|_{\mathcal{Y}} \le \delta_k$ for some sequence $\delta_k \searrow 0$ as $k \to \infty$.  For each $n,k\in\mathbb{N}$, choose $\alpha_{n,k}=\alpha(n,\delta_k)>0$ and $f_{n,k} \in \arg\min_{f\in X_{n,\infty}}\mathcal{T}_{\alpha_{n,k}}(f,g_k)$. Suppose that 
\begin{equation}\label{choicealpha}
    \alpha_{n,k} \to 0,
    \qquad
    \alpha_{n,k}^{-1} (n^{-\theta} +\delta_k^2)\to 0
    \quad \text{as } n,k \to\infty.
\end{equation}
Then, for every sequence $\{(n_\ell,k_\ell)\}_{\ell=1}^{\infty}$ satisfying $n_\ell\to\infty$ and $k_\ell\to\infty$, there exists a subsequence, still denoted by $\{f_{n_\ell,k_\ell}\}_{\ell=1}^{\infty}$, such that $f_{n_\ell,k_\ell} \to f^{\dagger}$ in $C(\overline \Omega)$, where $f^{\dagger}$ is an $\mathcal{R}$-minimizing solution of \eqref{1}, i.e.
\begin{equation}\label{deffdagger}
    f^{\dagger} \in \arg \min \left\{
        \mathcal{R}(f):=\varphi(\|f\|_{\mathcal{B}_1})
        \ \big|\ f \in \mathcal{B}_1,\ A(f)=g
    \right\}.
\end{equation} 
\end{enumerate}
\end{theorem}

\begin{proof}
 (a) Fix $n\in\mathbb N$, $g\in\mathcal Y$, and $\alpha>0$. Let $\{f_{n,k}\}_{k\in\mathbb{N}}\subset X_{n,\infty}$ be a minimizing sequence. Taking $f_{n,0}=0$, we have $\mathcal T_\alpha(f_{n,0},g)<\infty$. Hence, without loss of generality, we assume that $\alpha \mathcal{R}_n(f_{n,k})\le \mathcal T_{\alpha}(f_{n,k},g)\le \mathcal T_{\alpha}(f_{n,0},g)+1$ for all $k\in\mathbb N$. It follows that $\mathcal R_n(f_{n,k})
\le \frac{\mathcal T_\alpha(f_{n,0},g)+1}{\alpha}
=:C_{n,\alpha}.$ Writing $\mathcal{R}_n(f_{n,k}) = \varphi\left(\frac{1}{n}\sum_{j=1}^{n}|a_j^{(k)}|\right)$, the monotonicity of $\varphi$ and the  definition of its generalized inverse give
\[
|a_j^{(k)}|\le n\cdot \,\frac{1}{n}\sum_{i=1}^{n}|a_i^{(k)}|
\le n\,\varphi^{-1}(C_{n,\alpha})
=:R_{n,\alpha},
\quad j=1,\dots,n.
\]
Hence $f_{n,k}\in X_{n,R_{n,\alpha}}$ for all $k$. By Lemma~\ref{lemma compact set}, there exists a subsequence $\{f_{n,k_\ell}\}_{\ell=1}^{\infty}$, and some $f_n^*\in X_{n,R_{n,\alpha}}$ such that $f_{n,k_\ell} \to f_n^*$ in $C(\overline \Omega)$, and hence also in $L^2(\Omega)$. Moreover, by the proof of Lemma~\ref{lemma compact set}, the corresponding normalized parameters $(a,\boldsymbol b,c)$ may be chosen to converge. Since $\varphi$ is continuous, we have $\mathcal R_n(f_{n,k_\ell})\to \mathcal R_n(f_n^*)$. The continuity of $A:L^2(\Omega)\to\mathcal{Y}$ also gives $\|A(f_{n,k_\ell})-g\|_{\mathcal Y}^2
\to \|A(f_n^*)-g\|_{\mathcal Y}^2.$ Consequently,
\[
\mathcal T_\alpha(f_n^*,g)
=
\lim_{\ell\to\infty}
\mathcal T_\alpha(f_{n,k_\ell},g)
=
\inf_{f\in X_{n,\infty}}
\mathcal T_\alpha(f,g).
\]
Hence $f_n^*$ is a minimizer of $\mathcal{T}_{\alpha}(\cdot,g)$ on $X_{n,\infty}$.

\noindent(b) By the minimizing property of $f_{n,k}$, we have 
\[
\alpha\mathcal R_n(f_{n,k})
\le
\mathcal T_\alpha(f_{n,k},g_k)
\le
\mathcal T_\alpha(0,g_k).
\]
 Since $g_k\to g$ in $\mathcal Y$, the sequence
$\{\mathcal T_\alpha(0,g_k)\}_{k\in\mathbb N}$ is bounded. Hence $\{\mathcal R_n(f_{n,k})\}$ is uniformly bounded. As in part (a),  there exists a radius $R_{n,\alpha}^1<\infty$ such that $f_{n,k}\in X_{n,R_{n,\alpha}^1}$ for any $k\in\mathbb N$. By Lemma~\ref{lemma compact set}, there exists a subsequence $\{f_{n,k_\ell}\}_{\ell=1}^{\infty}$ and some $f_n^*\in X_{n,R_{n,\alpha}^1}$ such that $f_{n,k_\ell}\to f_n^*$ in $C(\overline \Omega)$. Moreover, the corresponding normalized parameters may be chosen to converge, and hence $\mathcal R_n(f_{n,k_\ell})\to \mathcal R_n(f_n^*)$. Since $A$ is continuous and $g_{k_\ell}\to g$, we also have $\|A(f_{n,k_\ell})-g_{k_\ell}\|_{\mathcal Y}^2\to \|A(f_n^*)-g\|_{\mathcal Y}^2.$ Therefore, $\mathcal T_\alpha(f_{n,k_\ell},g_{k_\ell})\to \mathcal T_\alpha(f_n^*,g)$. For every $f\in X_{n,\infty}$, the minimizing property gives $\mathcal T_\alpha(f_{n,k_\ell},g_{k_\ell})
\le
\mathcal T_\alpha(f,g_{k_\ell}).$ Letting $\ell\to\infty$ yields $\mathcal T_\alpha(f_n^*,g)
\le
\mathcal T_\alpha(f,g)$ for any $f\in X_{n,\infty}$. This proves $f_n^*\in\arg\min_{f\in X_{n,\infty}}\mathcal T_\alpha(f,g)$.

\noindent(c) Fix a sequence $\{(n_\ell,k_\ell)\}_{\ell=1}^{\infty}$ with $n_\ell\to\infty$ and $k_\ell\to\infty$. For simplicity, set
$f_\ell:=f_{n_\ell,k_\ell}$,
$\alpha_\ell:=\alpha_{n_\ell,k_\ell}$,
$\delta_\ell:=\delta_{k_\ell}$, and
$g_\ell:=g_{k_\ell}$. Let $\bar f\in\mathcal B_1$ be an arbitrary solution of $A(\bar f)=g$. For arbitrary $\epsilon\in (0,1)$, by the definition of the Barron norm, there exists a representing probability measure $\rho(\epsilon)$ of $\bar f$ such that $c_{\rho(\epsilon)}(\bar f) \le (1+\epsilon)\|\bar f\|_{\mathcal B_1}\le 2\|\bar f\|_{\mathcal B_1}.$ By Proposition~\ref{Approximtion theorem}, for every $\ell$ there exist $\bar f_{n_\ell}^{\,\epsilon}
\in
X_{n_\ell,c_{\rho(\epsilon)}(\bar f)}$ and a probability $\rho_{n_\ell}({\epsilon})$ of
$\bar f_{n_\ell}^{\,\epsilon}$ such that $c_{\rho_{n_\ell}(\epsilon)}
(\bar f_{n_\ell}^{\,\epsilon})
\le
c_{\rho(\epsilon)}(\bar f)
$ and \[\|\bar f-\bar f_{n_\ell}^{\,\epsilon}\|_{H^1(\Omega)}
\le
C(\Omega,d)c_{\rho(\epsilon)}(\bar f)n_\ell^{-1/2}\le 2C(\Omega,d)\|\bar f\|_{\mathcal B_1}n_\ell^{-1/2}.\]
 The local H\"older continuity of $A$ at $\bar f$ gives, for all sufficiently large $\ell$,
\[
\|A(\bar f_{n_\ell}^{\,\epsilon})-A(\bar f)\|_{\mathcal Y}^2
\le
L_A^2
\|\bar f_{n_\ell}^{\,\epsilon}-\bar f\|_{H^1(\Omega)}^{2\theta}
\le
C_{\bar f}n_\ell^{-\theta},
\]
where $C_{\bar f}:=L_A^2\left(2C(\Omega,d)\|\bar f\|_{\mathcal B_1}\right)^{2\theta}$ is independent of $\epsilon$. Together with $\|g_\ell-g\|_{\mathcal Y}\le\delta_\ell$, we obtain
\[
\|A(\bar f_{n_\ell}^{\,\epsilon})-g_\ell\|_{\mathcal Y}^2
\le
2\|A(\bar f_{n_\ell}^{\,\epsilon})-A(\bar f)\|_{\mathcal Y}^2
+2\|g-g_\ell\|_{\mathcal Y}^2
\le
2C_{\bar f} n_\ell^{-\theta}
+2\delta_\ell^2.
\]
Using the minimizing property of $f_\ell$ and the monotonicity of
$\varphi$,
\[
\|A(f_\ell)-g_\ell\|_{\mathcal Y}^2
+\alpha_\ell\mathcal R_{n_\ell}(f_\ell)
\le
\|A(\bar f_{n_\ell}^{\,\epsilon})-g_\ell\|_{\mathcal Y}^2
+\alpha_\ell\mathcal R_{n_\ell}(\bar f_{n_\ell}^{\,\epsilon})
\le
2C_{\bar f}n_\ell^{-\theta}
+2\delta_\ell^2
+\alpha_\ell\varphi(c_{\rho(\epsilon)}(\bar f)).
\]
Hence, by \eqref{choicealpha}, $\|A(f_\ell)-g_\ell\|_{\mathcal Y}\to0$. Moreover, since $c_{\rho(\epsilon)}(\bar f) \le (1+\epsilon)\|\bar f\|_{\mathcal B_1}\le 2\|\bar f\|_{\mathcal B_1}$, we have
\begin{equation}
    \label{upperR}
    \mathcal R_{n_\ell}(f_\ell)
\le
2\alpha_\ell^{-1}
\bigl(C_{\bar f}n_\ell^{-\theta}+\delta_\ell^2\bigr)
+\varphi((1+\epsilon)\|\bar f\|_{\mathcal B_1})\le 2\alpha_\ell^{-1}
\bigl(C_{\bar f}n_\ell^{-\theta}+\delta_\ell^2\bigr)+\varphi(2\|\bar f\|_{\mathcal B_1}).
\end{equation}
Thus, $\{\mathcal R_{n_\ell}(f_\ell)\}$ is uniformly bounded by a constant
independent of $\epsilon$ for sufficiently large $\ell$. Since
$\mathcal R_{n_\ell}(f_\ell)=\varphi(n_\ell^{-1}\sum_{j=1}^{n_\ell}|a_j^{(\ell)}|)$ and $\varphi(t)\to\infty$ as $t\to\infty$, the coefficient averages $q_\ell:=n_\ell^{-1}\sum_{j=1}^{n_\ell}|a_j^{(\ell)}|$ are also uniformly bounded. By Lemma~\ref{lem:barron_compact_lsc}, there exists a subsequence, still denoted by $\{f_\ell\}$, and some $f^*\in C(\overline \Omega)\cap\mathcal B_1$ such that $f_\ell\to f^*$ in $C(\overline\Omega)$, and $\|f^*\|_{\mathcal B_1}
\le
\liminf_{\ell\to\infty}
q_\ell.$
The continuity of $A:L^2(\Omega)\to\mathcal Y$, together with $g_\ell\to g$ and $\|A(f_\ell)-g_\ell\|_{\mathcal Y}\to0$, implies that $A(f^*)=g$. Since \eqref{upperR} holds for every $\epsilon\in(0,1)$, using \eqref{choicealpha} gives $\limsup_{\ell\to\infty}
\mathcal R_{n_\ell}(f_\ell)
\le
\varphi\bigl(
(1+\epsilon)\|\bar f\|_{\mathcal B_1}
\bigr).$ Therefore, by the monotonicity and continuity of $\varphi$,
\[
\mathcal R(f^*)
=
\varphi(\|f^*\|_{\mathcal B_1})
\le
\liminf_{\ell\to\infty}
\varphi(q_\ell)
=
\liminf_{\ell\to\infty}\mathcal R_{n_\ell}(f_\ell)
\le
\varphi\bigl((1+\epsilon)\|\bar f\|_{\mathcal B_1}\bigr).
\]
 Letting $\epsilon\to0$ and using the continuity of $\varphi$, we obtain $\mathcal R(f^*)
\le
\varphi(\|\bar f\|_{\mathcal B_1})
=
\mathcal R(\bar f).$ Since $\bar f\in \mathcal{B}_1$ was an arbitrary solution of $A(f)=g$, $f^*$ is an $\mathcal R$-minimizing solution. 
\end{proof}

Next we derive convergence rates results. Generally, the convergence rate of $f_{n} \to f^{\dagger}$ as $n\to \infty$ for exact data, and of $f_{n_\ell,k_{\ell}} \to f^{\dagger}$ as $\ell\to \infty$ for noisy data, can be arbitrarily slow if $f^{\dagger}$ lacks smoothness. Thus convergence rates typically require certain smoothness conditions on $f^\dagger$. Commonly used source conditions include range-type \cite{Mathe-2006,ZhangHof2018,YuanZhang2025}, approximate  \cite{Hofmann2009}, and variational \cite{Hofmann2007}. Below we focus on variational source conditions. 

\begin{assumption}
\label{Varsourcecondition}
Let $f^\dagger\in\mathcal B_1$ be an $\mathcal R$-minimizing solution of $A(f)=g$. There exist a radius $\varepsilon>0$, an error measure functional $\mathcal{E}^{\dagger}: \mathcal{B}_1 \to[0, \infty]$ with $\mathcal{E}^{\dagger}\left(f^\dagger\right)=0$ and a concave index function\footnote{$\Phi$ is called an index function if it is continuous and strictly increasing with $\Phi(0)=0$.} $\Phi:[0, \infty) \to[0, \infty)$ such that 
    \begin{equation}
    \label{varsourceassump}
        \mathcal{E}^{\dagger}(f) \le \mathcal{R}(f)-\mathcal{R}\left(f^\dagger\right)+\Phi\bigl(\|A(f)-g\|_{\mathcal Y}^2\bigr),\quad 
        \forall f\in\mathcal{B}_1 \mbox{ with }\|f-f^\dagger\|_{C(\overline \Omega)}<\varepsilon.
    \end{equation}
\end{assumption}

Under Assumption \ref{Varsourcecondition}, we have the following convergence rates. Below $\Phi^{-*}$ denotes the Fenchel conjugate of the inverse function $\Phi^{-1}$, namely $\Phi^{-*}(\alpha):=(\Phi^{-1})^*(\alpha)=\sup_{t\ge0}\{\alpha t-\Phi^{-1}(t)\}$

\begin{theorem}\label{convergencerate}
Let the conditions of Theorem~\ref{well-posedness}(c) hold. Let $\{f_{n_\ell,k_\ell}\}_{\ell=1}^\infty$
be a subsequence obtained in Theorem~\ref{well-posedness}(c), with $f_{n_\ell,k_\ell} \in \arg\min_{f\in X_{n_\ell,\infty}}\mathcal{T}_{\alpha_{n_\ell,k_\ell}}(f,g_{k_\ell})$, such that $f_{n_\ell,k_\ell}\to f^\dagger$ in $C(\overline{\Omega})$. Suppose further that $f^{\dagger}$ satisfies Assumption~\ref{Varsourcecondition}. Then,
\begin{enumerate}
    \item[(i)] for all sufficiently large $\ell$,
\begin{align}
\mathcal E^\dagger(f_{n_\ell,k_\ell})
&\le
2\alpha_{n_\ell,k_\ell}^{-1}
\left(
\delta_{k_\ell}^2+C_{f^\dagger} n_\ell^{-\theta}
\right)
+\Phi(2\delta_{k_\ell}^2)
+\frac{1}{2\alpha_{n_\ell,k_\ell}}
\Phi^{-*}(2\alpha_{n_\ell,k_\ell}),
\label{errorestimate}
\\
\|f_{n_\ell,k_\ell}\|_{\mathcal B_1}
&\le
\varphi^{-1}
\left(
2\alpha_{n_\ell,k_\ell}^{-1}
\left(
\delta_{k_\ell}^2+C_{f^\dagger} n_\ell^{-\theta}
\right)
+\varphi(\|f^\dagger\|_{\mathcal B_1})
\right),
\label{varphiestimate}
\end{align}
where $C_{f^\dagger}:=L_A^2\left(C(\Omega,d)\|f^\dagger\|_{\mathcal B_1}\right)^{2\theta}$.
\item[(ii)] if $n_\ell\asymp\delta_{k_\ell}^{-2/\theta}, \alpha_{n_\ell,k_\ell}
\asymp
{\delta_{k_\ell}^2}
/{\Phi(2\delta_{k_\ell}^2)},$ then $\mathcal E^\dagger(f_{n_\ell,k_\ell})
=
\mathcal O\bigl(\Phi(2\delta_{k_\ell}^2)\bigr).$ For $\Phi(t)=ct^\mu$ with $c>0$ and $0<\mu<1$, the choices $n_\ell\asymp\delta_{k_\ell}^{-2/\theta}$, and $\alpha_{n_\ell,k_\ell}
\asymp\delta_{k_\ell}^{2(1-\mu)}$ yield $\mathcal E^\dagger(f_{n_\ell,k_\ell})
=\mathcal O(\delta_{k_\ell}^{2\mu}).$
\end{enumerate}
\end{theorem}

\begin{proof}
Set
$f_\ell:=f_{n_\ell,k_\ell}$,
$\alpha_\ell:=\alpha_{n_\ell,k_\ell}$,
$\delta_\ell:=\delta_{k_\ell}$, and
$g_\ell:=g_{k_\ell}$. For any $\epsilon\in(0,1)$, by the definition of Barron norm, there exists a representing probability measure $\rho(\epsilon)$ of $f^\dagger$ such that $c_{\rho(\epsilon)}(f^{\dagger})\le (1+\epsilon)\|f^{\dagger}\|_{\mathcal{B}_1}$. By Proposition~\ref{Approximtion theorem}, for each $\ell$, there exists a two-layer NN $f_{n_\ell}^{\dagger,\epsilon}$ and a probability 
$\rho_{n_\ell}(\epsilon)$ of
$f_{n_\ell}^{\dagger,\epsilon}$ such that
$c_{\rho_{n_\ell}(\epsilon)}(f_{n_\ell}^{\dagger,\epsilon})
\le
c_{\rho(\epsilon)}(f^\dagger)$ and \[\|f_{n_\ell}^{\dagger,\epsilon}-f^\dagger\|_{H^1(\Omega)}
\le
C(\Omega,d)c_{\rho(\epsilon)}(f^\dagger)
n_\ell^{-1/2}\le C(\Omega,d)(1+\epsilon)
\|f^\dagger\|_{\mathcal B_1}n_\ell^{-1/2}.\]
Hence, for all sufficiently large $\ell$, independently of $\epsilon$, the local H\"older continuity of $A$ at $f^\dagger$ and the noise estimate $\|g_\ell-g\|_{\mathcal Y}\le\delta_\ell$ yield
\[
\|A(f_{n_\ell}^{\dagger,\epsilon})-g_\ell\|_{\mathcal Y}^2\le 2L_A^2\|f_{n_\ell}^{\dagger,\epsilon}-f^{\dagger}\|_{H^1(\Omega)}^{2\theta}+2\|g-g_\ell\|_{\mathcal{Y}}^2 \le
2(1+\epsilon)^{2\theta}
C_{f^\dagger}n_\ell^{-\theta}
+
2\delta_\ell^2,
\]
where $C_{f^\dagger}=L_A^2\left(C(\Omega,d)\|f^\dagger\|_{\mathcal B_1}\right)^{2\theta}$. Moreover, the monotonicity of $\varphi$ gives $\mathcal R_{n_\ell}(f_{n_\ell}^{\dagger,\epsilon})
\le
\varphi(c_\rho(f^\dagger))
\le
\varphi\bigl((1+\epsilon)
\|f^\dagger\|_{\mathcal B_1}\bigr).$ Therefore, by the minimizing property of $f_\ell$,
\begin{equation*}
    \|A(f_\ell)-g_\ell\|_{\mathcal Y}^2
+
\alpha_\ell\mathcal R_{n_\ell}(f_\ell)\le \mathcal{T}_{\alpha_\ell}(f_{n_\ell}^{\dagger,\epsilon},g_\ell)\le 
2(1+\epsilon)^{2\theta}
C_{f^\dagger}n_\ell^{-\theta}
+
2\delta_\ell^2
+
\alpha_\ell
\varphi\left((1+\epsilon)\|f^{\dagger}\|_{\mathcal{B}_1}\right).
\end{equation*}
Since $\mathcal R(f_\ell)\le\mathcal R_{n_\ell}(f_\ell)$ and the preceding estimate holds for every $\epsilon\in(0,1)$, letting $\epsilon\to 0$ and using the continuity of $\varphi$ yield
\begin{equation}
\label{prorate1}
\|A(f_\ell)-g_\ell\|_{\mathcal Y}^2
+\alpha_\ell\mathcal R(f_\ell)
\le
2C_{f^\dagger} n_\ell^{-\theta}
+2\delta_\ell^2
+\alpha_\ell\mathcal R(f^\dagger).
\end{equation}
Dropping the first term on the left-hand side of \eqref{prorate1} gives \eqref{varphiestimate}. On the other hand, \eqref{prorate1} also implies
\[
\begin{aligned}
\alpha_\ell
\bigl(\mathcal R(f_\ell)-\mathcal R(f^\dagger)\bigr)
\le{}&
2C_{f^\dagger} n_\ell^{-\theta}
+2\delta_\ell^2
-\|A(f_\ell)-g_\ell\|_{\mathcal Y}^2.
\end{aligned}
\]
Since $f_\ell\to f^\dagger$ in $C(\overline\Omega)$,
Assumption~\ref{Varsourcecondition} applies for all sufficiently large $\ell$. Moreover, $\|A(f_\ell)-g\|_{\mathcal Y}^2 \le 2\|A(f_\ell)-g_\ell\|_{\mathcal Y}^2 +2\delta_\ell^2.$ Combining these estimates with the variational source condition \eqref{varsourceassump} gives
\[
\begin{aligned}
\alpha_\ell\mathcal E^\dagger(f_\ell)
\le&
2C_{f^\dagger} n_\ell^{-\theta}
+2\delta_\ell^2
-\|A(f_\ell)-g_\ell\|_{\mathcal Y}^2
+\alpha_\ell\Phi\left(
2\|A(f_\ell)-g_\ell\|_{\mathcal Y}^2
+2\delta_\ell^2
\right)
\\
\le&
2C_{f^\dagger} n_\ell^{-\theta}
+2\delta_\ell^2
+\alpha_\ell\Phi(2\delta_\ell^2)
+\frac12\Phi^{-*}(2\alpha_\ell).
\end{aligned}
\]
Here, the last inequality follows from the subadditivity of the concave
index function $\Phi$ and Young's inequality $\alpha\Phi(2t)
\le t+\frac12\Phi^{-*}(2\alpha)$ for $t\ge 0$. Dividing the last estimate by $\alpha_\ell$ proves
\eqref{errorestimate}. Finally, balancing $\delta_\ell^2$ and $n_\ell^{-\theta}$ gives $n_\ell\asymp\delta_\ell^{-2/\theta}$. Choosing
$\alpha_\ell\asymp
\delta_\ell^2/\Phi(2\delta_\ell^2)$ then yields
$\mathcal E^\dagger(f_\ell)
=\mathcal O(\Phi(2\delta_\ell^2))$.
The case $\Phi(t)=ct^\mu$ follows directly.
\end{proof}

Theorem~\ref{convergencerate} implies a sufficient choice for the number of neurons: Under a general variational source condition, the Tikhonov-NN scheme achieves the convergence rate \(\Phi(2\delta_{k_\ell}^2)\), which is comparable with the classical rate. Moreover, Theorem~\ref{convergencerate} quantifies the requisite network width to realize the rate, i.e., $n_\ell\asymp\delta_{k_\ell}^{-2/\theta}$.

\section{Numerical experiments}
\label{simulation}

In this section, we present numerical examples to complement the theoretical results. 

For Examples~\ref{Ex1}--\ref{Ex2}, we generate the exact function $f^{\dagger}$ in the form  \cite{LiYuanyuan}
\begin{equation}
\label{exactf}
f^{\dagger}(t)=\sum_{j=1}^T \phi_j K_d\left(\mathbf x, \ell^j\right), \quad \mathbf x \in[0,1]^d,
\end{equation}
where $K_d$ is the polynomial kernel defined by $K_d(\mathbf x, \ell)=\langle \mathbf x, \ell \rangle_{\mathbb{R}^d}+1$. The specific choices of $T$, $d$, $\ell^j$, and $\phi_j$ are given in the corresponding examples. By construction, $f^\dagger$ belongs to the Barron space $\mathcal B_1$, cf. Remark \ref{finbaron}.
Then we compute the exact data $g$ by $g=A(f^\dagger)$ and generate the noisy data $g^\delta$ by
$g^\delta=g+\delta\xi$,
where $\delta>0$ denotes the noise level and $\xi$ is uniformly distributed in
$[-1,1]$.

\begin{remark}
\label{finbaron}
    The exact solution $f^\dagger$ generated by \eqref{exactf} belongs to
the Barron space $\mathcal B_1$. Indeed, it can be written as $f^{\dagger}(\mathbf x)=u^T \mathbf x+v$ with $u:=\sum_{j=1}^T \phi_j \ell^j\in \mathbb{R}^d$ and $v:=\sum_{j=1}^T \phi_j$. Define the uniform empirical measure $\rho_{d+1}$
on the following $d+1$ atoms: $(a_k,\boldsymbol b_k,c_k)=((d+1) u_k,e_k,0)$, $k=1,\ldots,d$, and $(a_0,\boldsymbol b_0,c_0)=((d+1) v,\boldsymbol 0,1)$.
Then, for $\mathbf x\in[0,1]^d$, since $\mathbf x_k\ge0$, we have
\[
\frac{1}{d+1}\sum_{k=1}^d (d+1) u_k(e_k^T\mathbf x)_+
+
\frac{1}{d+1} (d+1) v(\boldsymbol 0^T\mathbf x+1)_+
=
\sum_{k=1}^d u_k \mathbf x_k+v
=
f^\dagger(\mathbf x).
\]
Thus $f^\dagger$ admits a finite Barron representation. Moreover, the representation cost is
\[
c_{\rho_{d+1}}(f^{\dagger})=\|f^{\dagger}\|_{\mathcal{B}_{1,\rho_{d+1}}}
=\mathbb{E}_{\rho_{d+1}}\left[
|a|\left(\|\boldsymbol b\|_1+|c|\right)
\right]
=\sum_{k=1}^d |u_k|+|v|
=\Big\|\sum_{j=1}^T \phi_j \ell^j\Big\|_1+\Big|\sum_{j=1}^T \phi_j\Big|.
\]
Consequently, $\|f^\dagger\|_{\mathcal B_1}\le c_{\rho_{d+1}}(f^\dagger)<\infty$, and hence $f^\dagger\in\mathcal B_1$.
\end{remark}

We approximate $f^\dagger$ by two-layer neural networks of the form $f(\mathbf x)
=
\frac1n\sum_{j=1}^n
a_j(\boldsymbol b_j^T\mathbf x+c_j)_+$.
For ENN1 (Algorithm~\ref{alg:expanding_nn}), we choose an increasing sequence $r(n)$ bounded by $B> c_{\rho_{d+1}}(f^{\dagger})=\|\sum_{j=1}^T \phi_j \ell^j\|_{1}+|\sum_{j=1}^T\phi_j|$. For each fixed width $n$, we solve
 \begin{equation*}
\min_{\substack{(a_j,\boldsymbol{b}_j,c_j) \in M_{r(n)},\\ j=1,\dots,n}}
\ \left\|\,A\!\left(\frac{1}{n} \sum_{j=1}^n a_j \left(\boldsymbol{b}_j^T \mathbf{x}+c_j\right)_{+}\right) - g^\delta \right\|_{\mathcal{Y}},
\end{equation*}
with $M_{r(n)} := \{ (a_{j},\boldsymbol{b}_j,c_j)\in \mathbb{R}\times \mathbb{R}^d \times \mathbb{R},\ |a_j|\le r(n),\ \|\boldsymbol{b}_j\|_1+|c_j| =1 \}$. For ENN2 (Algorithm~\ref{alg:modified_enn}), we take $\mathcal{R}(f)=\mathcal{R}_H(f)=C_{\mathcal{R}_H}\|f\|_{H^1(\Omega)}$, which is $C_{\mathcal{R}_H}$-Lipschitz continuous with respect to the
$H^1(\Omega)$ norm. Then we solve
    \begin{equation*}
\min_{\substack{(a_j,\boldsymbol{b}_j,c_j) \in M_{r(n)},\\ j=1,\dots,n}}
\ \left\|\,A\!\left(\frac{1}{n}\sum_{j=1}^n a_j\, (\boldsymbol{b}_j^T \mathbf{x} + c_j)_{+}\right) - g^\delta \right\|_{\mathcal{Y}}
\;+\; \beta_n C_{\mathcal{R}_H} \left\|\frac{1}{n}\sum_{j=1}^n a_j\, (\boldsymbol{b}_j^T \mathbf{x} + c_j)_{+}\right\|_{H^1(\Omega)}
,
\end{equation*}
where we choose $\beta_n=c_0n^{-\frac{\theta}{2}}$, and $c_0=C_{\mathcal{R}_H}=0.01$. Further, for Tikhonov-NN, we solve \begin{equation*}
\min_{\substack{(a_j,\boldsymbol{b}_j,c_j) \in M_{\infty},\\ j=1,\dots,n}}
\ \left\|\,A\!\left(\frac{1}{n}\sum_{j=1}^n a_j\, (\boldsymbol{b}_j^T \mathbf{x} + c_j)_{+}\right) - g^\delta \right\|^2_{\mathcal{Y}}
\;+\; \alpha \left( \frac{1}{n}\sum_{j=1}^n |a_j|(\left\|\boldsymbol{b_j}\right\|_1+\left|c_j\right|) \right)^{2}
,
\end{equation*}
with $\alpha$ chosen by $\alpha=(\delta+n^{-1})^{0.8}$.

The numerical experiments are conducted with the following settings. The exact solution $f^\dagger$ is evaluated on an $N$-point grid, while the noisy data $g^\delta$ are represented by $M$ discrete measurements. The resulting non-convex optimization problems are solved using the Adam optimizer with a fixed learning rate of $10^{-3}$. The network width starts from $n=50$ and is increased by $20$ at each stage, with $n\le n_{\max}$. A warm-start strategy is employed, in which the parameters of the existing neurons are retained and only those of the newly added neurons are randomly initialized. The reconstruction accuracy is measured by the relative $L^2$-error
\begin{equation}\label{eq:rel_L2_error}
    \text{Relative } L^2 \text{ Error} := \frac{\|f_n^\delta - f^\dagger\|_{L^2([0,1]^d)}}{\|f^\dagger\|_{L^2([0,1]^d)}}.
\end{equation}

For Examples~\ref{Ex1}--\ref{Ex2}, we set $N=101$, with $M=51$ for
Example~\ref{Ex1} and $M=101$ for Example~\ref{Ex2}, and take
$n_{\max}=1000$. At each width, ENN1 and ENN2 are trained for 600
iterations, while Tikhonov-NN is trained for 1200 iterations. For
Example~\ref{Ex3}, we set $N=31^2=961$, $M=124$, and $n_{\max}=500$. At each width, ENN1 and ENN2 are trained for 1500 iterations, while Tikhonov-NN is trained for 2500 iterations. Three seeds are used for each example: $111,666,3333$ for Example~\ref{Ex1}, $678,765,987$ for Example~\ref{Ex2}, and $20,30,40$ for Example~\ref{Ex3}. They are chosen a priori, agnostic to the performance metrics, and used consistently across all algorithms. For Examples~\ref{Ex1}--\ref{Ex2}, the seed determines the randomly generated exact solution $f^\dagger$ and hence the corresponding exact data $g=A(f^\dagger)$, as well as the noise realization and neural network initialization; for Example~\ref{Ex3}, the exact conductivity is fixed, so the seed affects only the noise realization and network initialization. This setting allows us to examine whether the observed error trends persist across different problem instances and stochastic realizations.

\begin{example}
\label{Ex1}
Fredholm integral equation of the first kind:
    \begin{equation*}
    A(f)=\int_0^1 K(t,s)f(s)\, ds=g(t),\quad K(t, s)= \begin{cases}s(1-t), & 0\le s \le t\le 1, \\  t(1-s), & 0\le t<s\le 1.\end{cases}
\end{equation*}
The operator $A$ is linear, compact, injective and self-adjoint on \(L^2[0,1]\). Its eigenvalues are \(\lambda_n = 1/(n^2\pi^2)\) with eigenfunctions \(\sin(n\pi t)\) and $\|A\|_{\mathcal L(L^2(0,1),L^2(0,1))}=1/\pi^2$. There holds
\[
\|A(f)-A(f^\dagger)\|_{L^2(0,1)}
\le
\frac{1}{\pi^2}\|f-f^\dagger\|_{L^2(0,1)}
\le
\frac{1}{\pi^2}\|f-f^\dagger\|_{H^1(0,1)}.
\]
That is, $A$ is Lipschitz at $f^\dagger$ with $L_A=1/\pi^2$. Thus, $A$ satisfies the assumptions of Theorems~\ref{main}-\ref{well-posedness}. The exact solution $f^\dagger$ is generated according to \eqref{exactf} with $T=10$ and $d=1$, where $\ell^j$ and $\phi_j$, $j=1,\ldots,T$, are independently and uniformly sampled from $[-1,1]$.

\end{example}
\begin{figure}[htbp]
	\centering
	\begin{minipage}{0.3\linewidth}
		\centering
		\includegraphics[width=\linewidth]{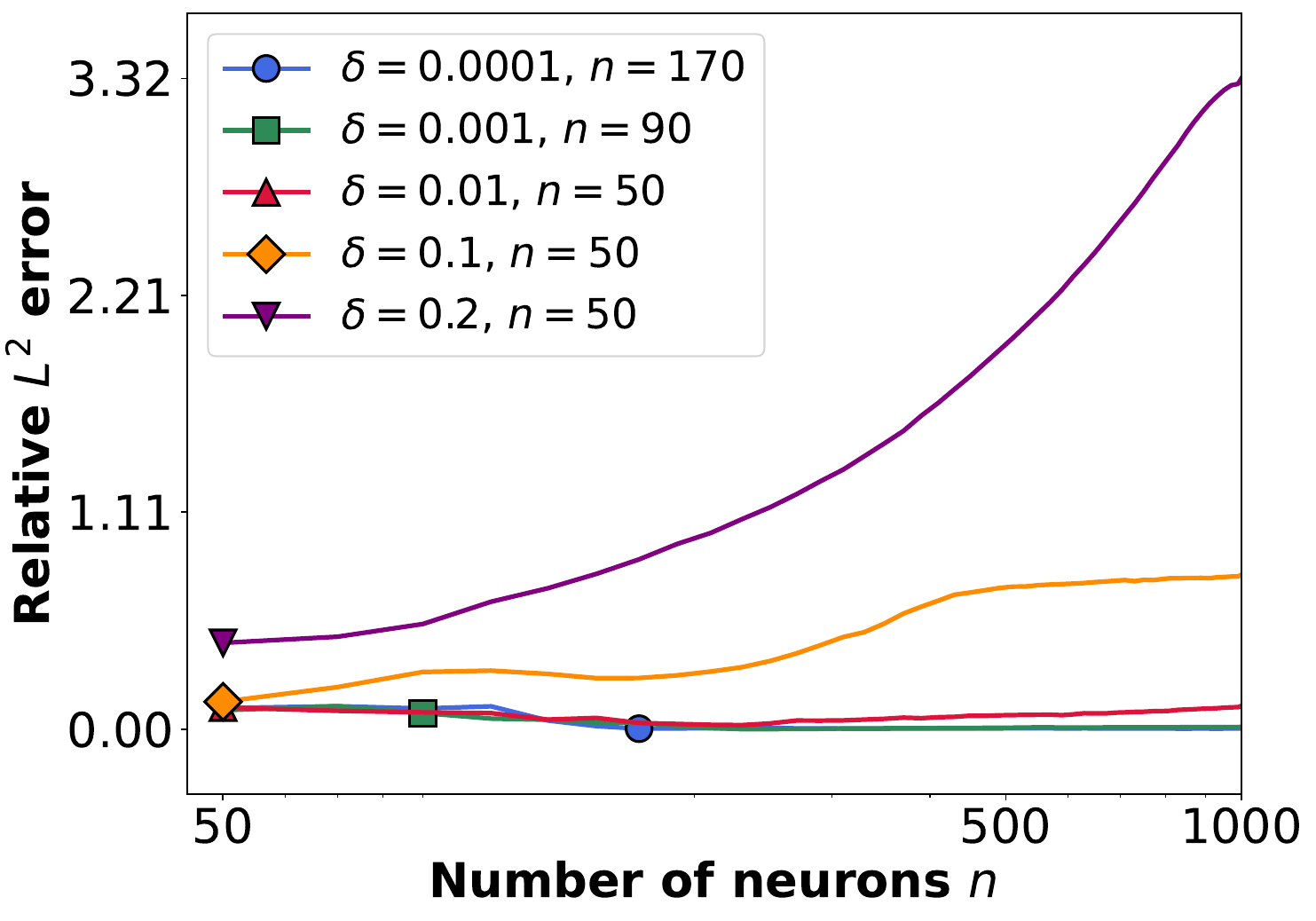}
	\end{minipage}
	\begin{minipage}{0.3\linewidth}
		\centering
		\includegraphics[width=\linewidth]{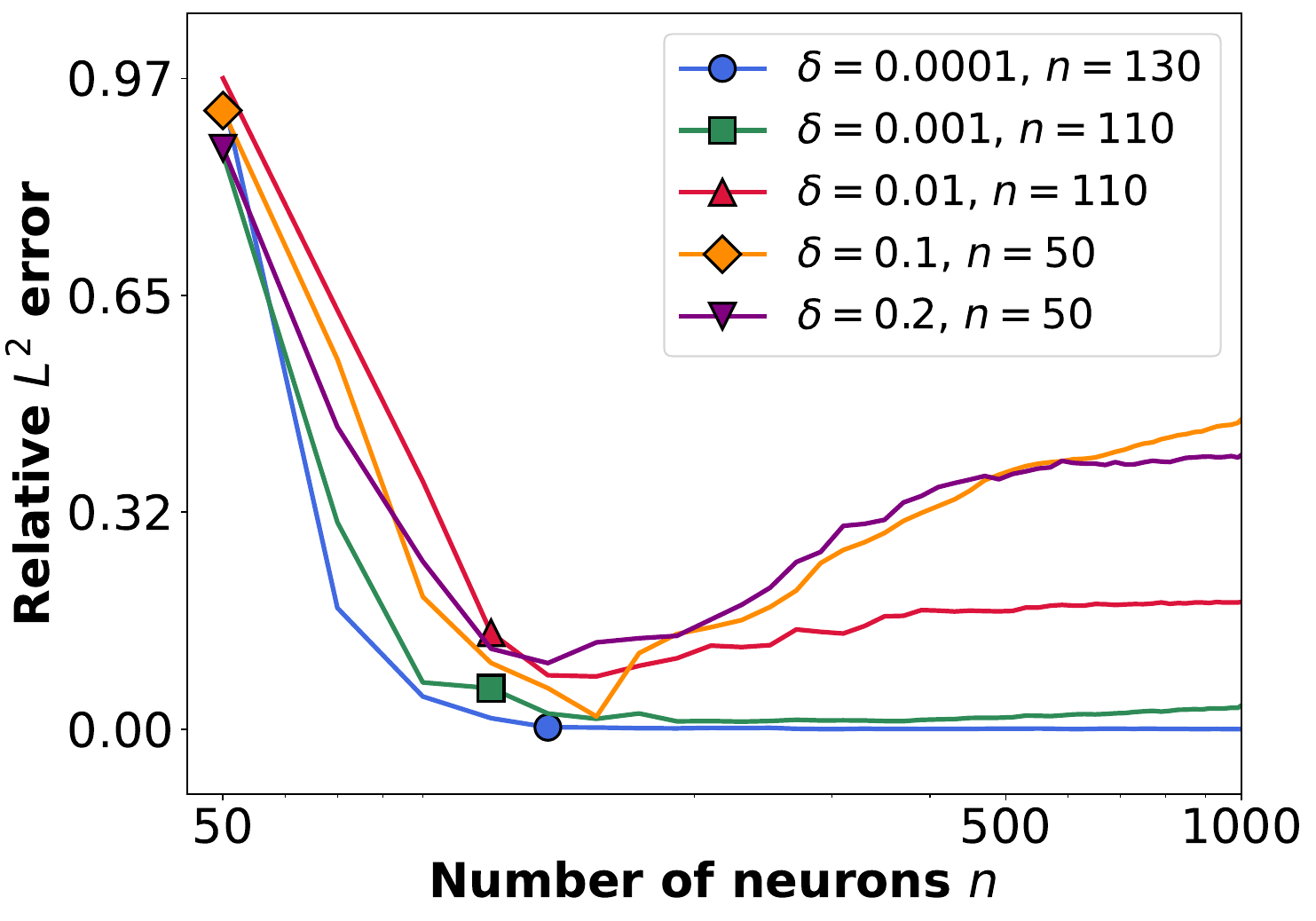}
	\end{minipage}
	\begin{minipage}{0.3\linewidth}
		\centering
		\includegraphics[width=\linewidth]{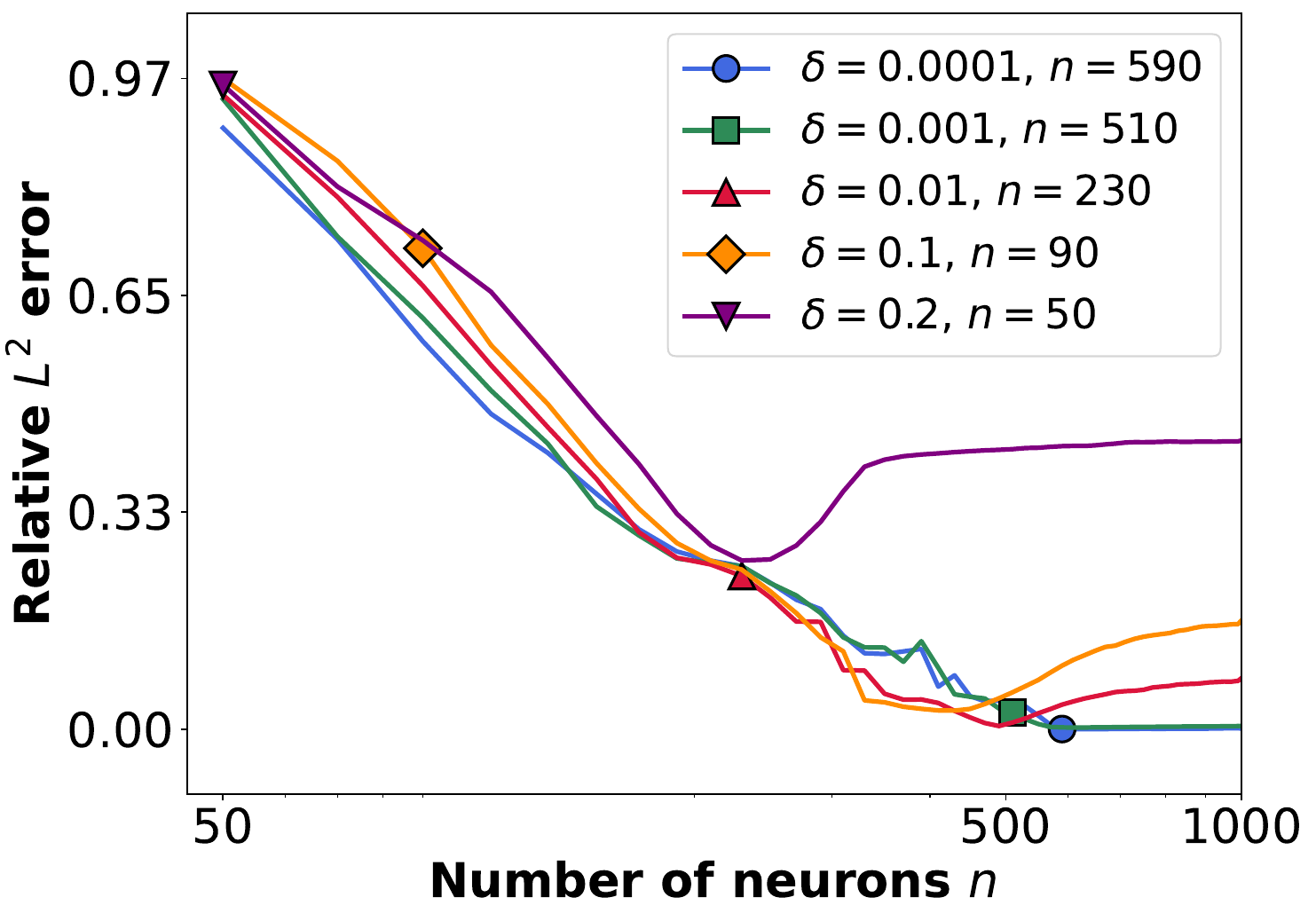}
	\end{minipage}

	\begin{minipage}{0.3\linewidth}
		\centering
		\includegraphics[width=\linewidth]{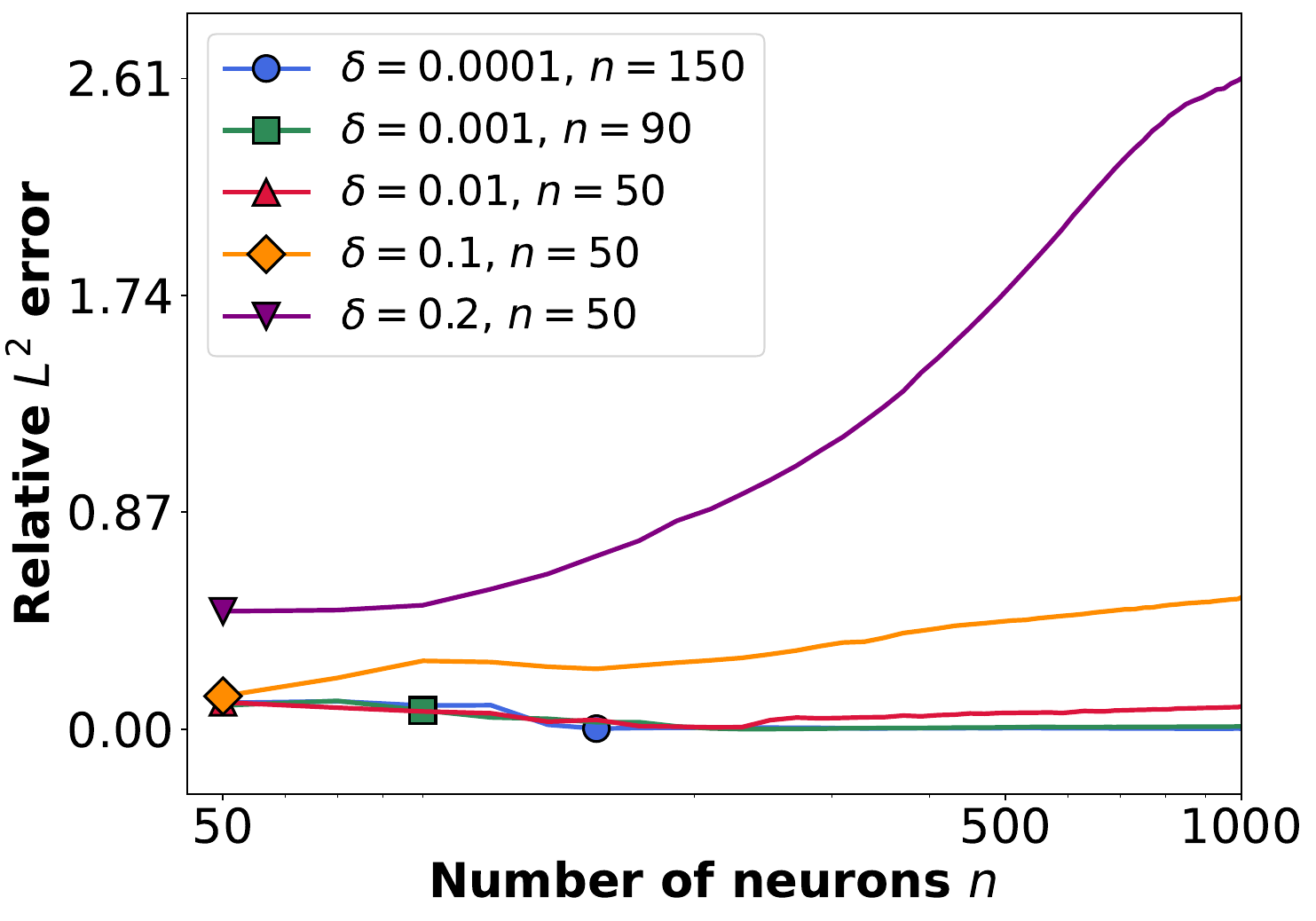}
	\end{minipage}
	\begin{minipage}{0.3\linewidth}
		\centering
		\includegraphics[width=\linewidth]{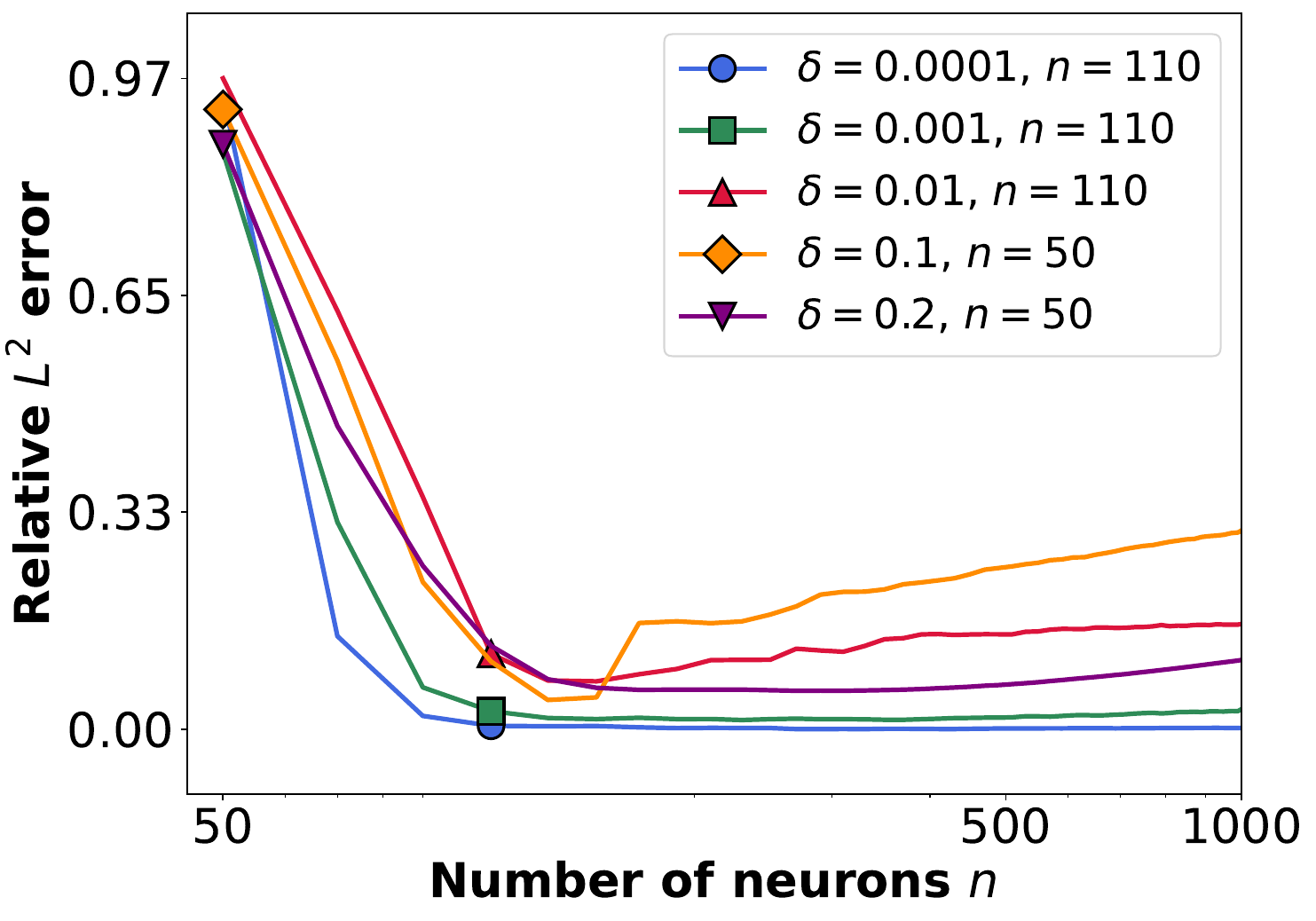}
	\end{minipage}
	\begin{minipage}{0.3\linewidth}
		\centering
		\includegraphics[width=\linewidth]{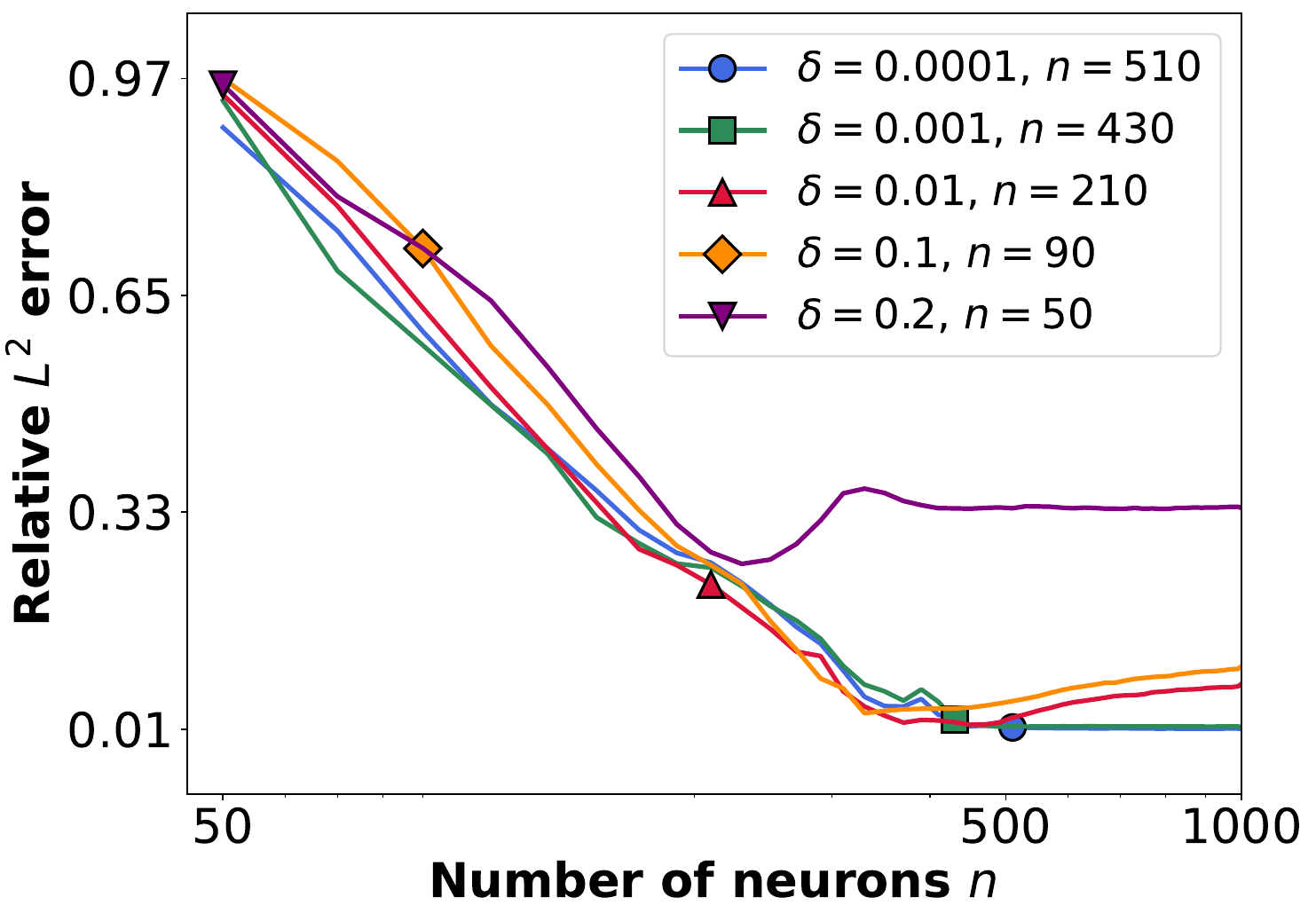}
	\end{minipage}

	\begin{minipage}{0.3\linewidth}
		\centering
		\includegraphics[width=\linewidth]{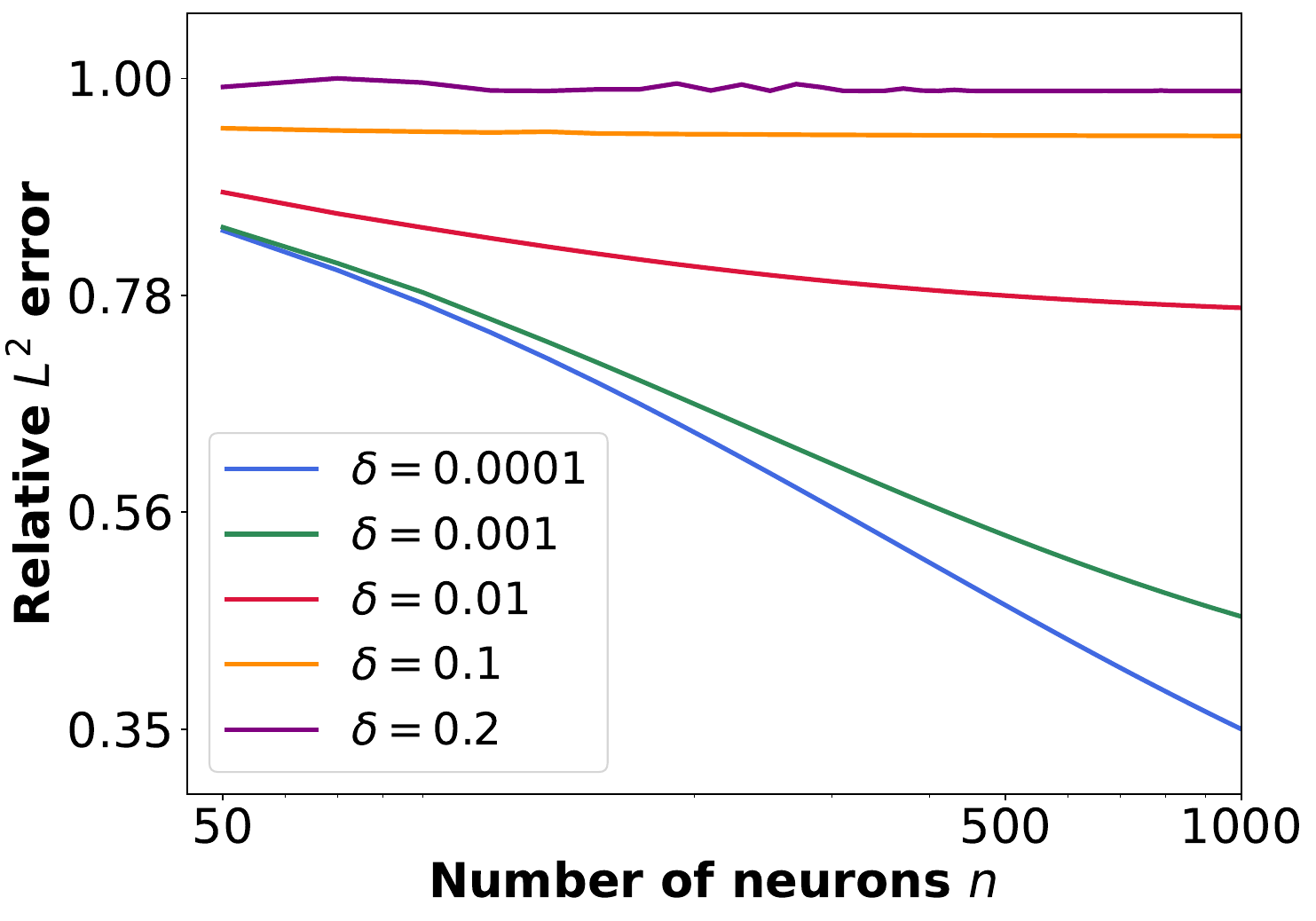}
	\end{minipage}
	\begin{minipage}{0.3\linewidth}
		\centering
		\includegraphics[width=\linewidth]{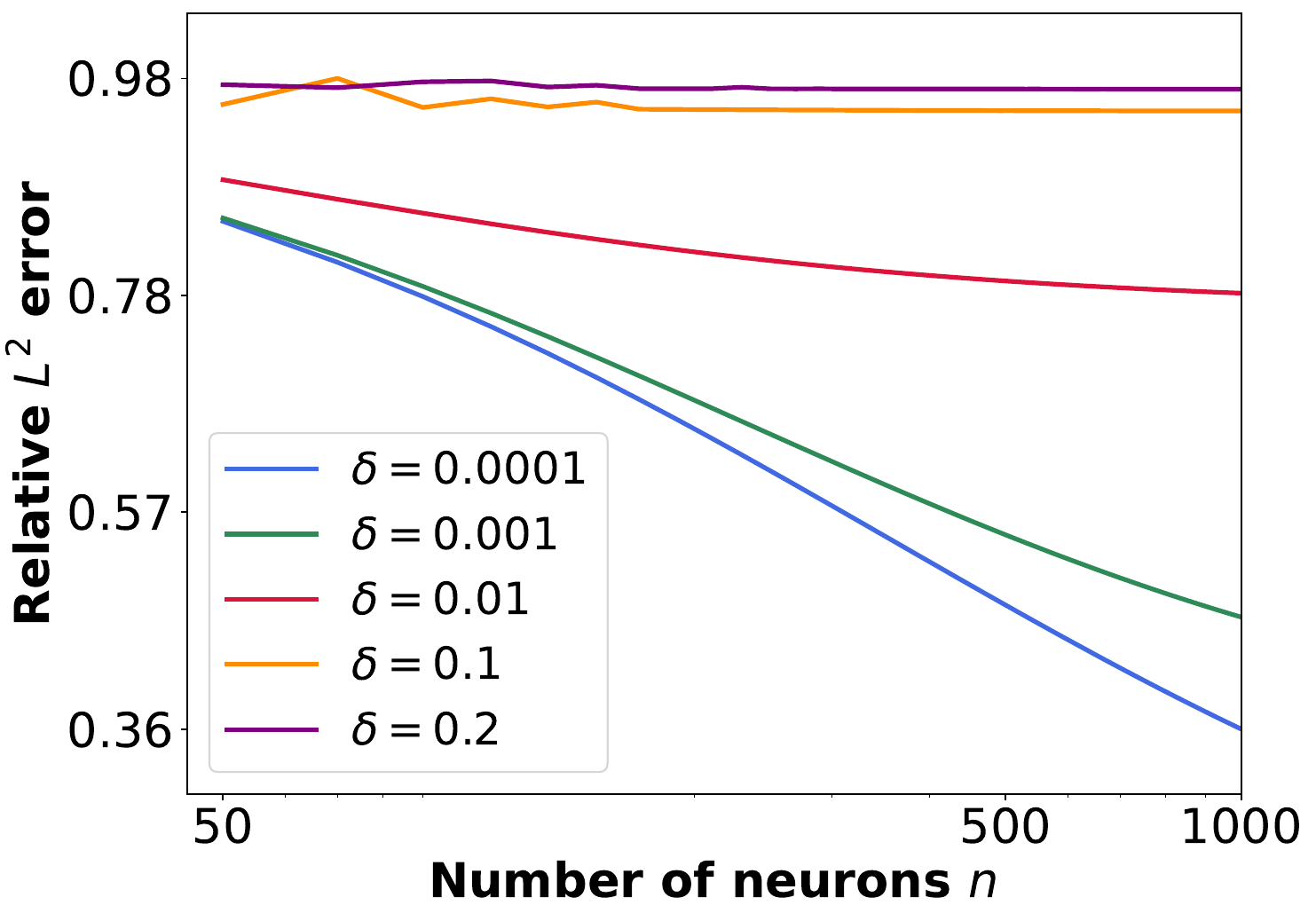}
	\end{minipage}
	\begin{minipage}{0.3\linewidth}
		\centering
		\includegraphics[width=\linewidth]{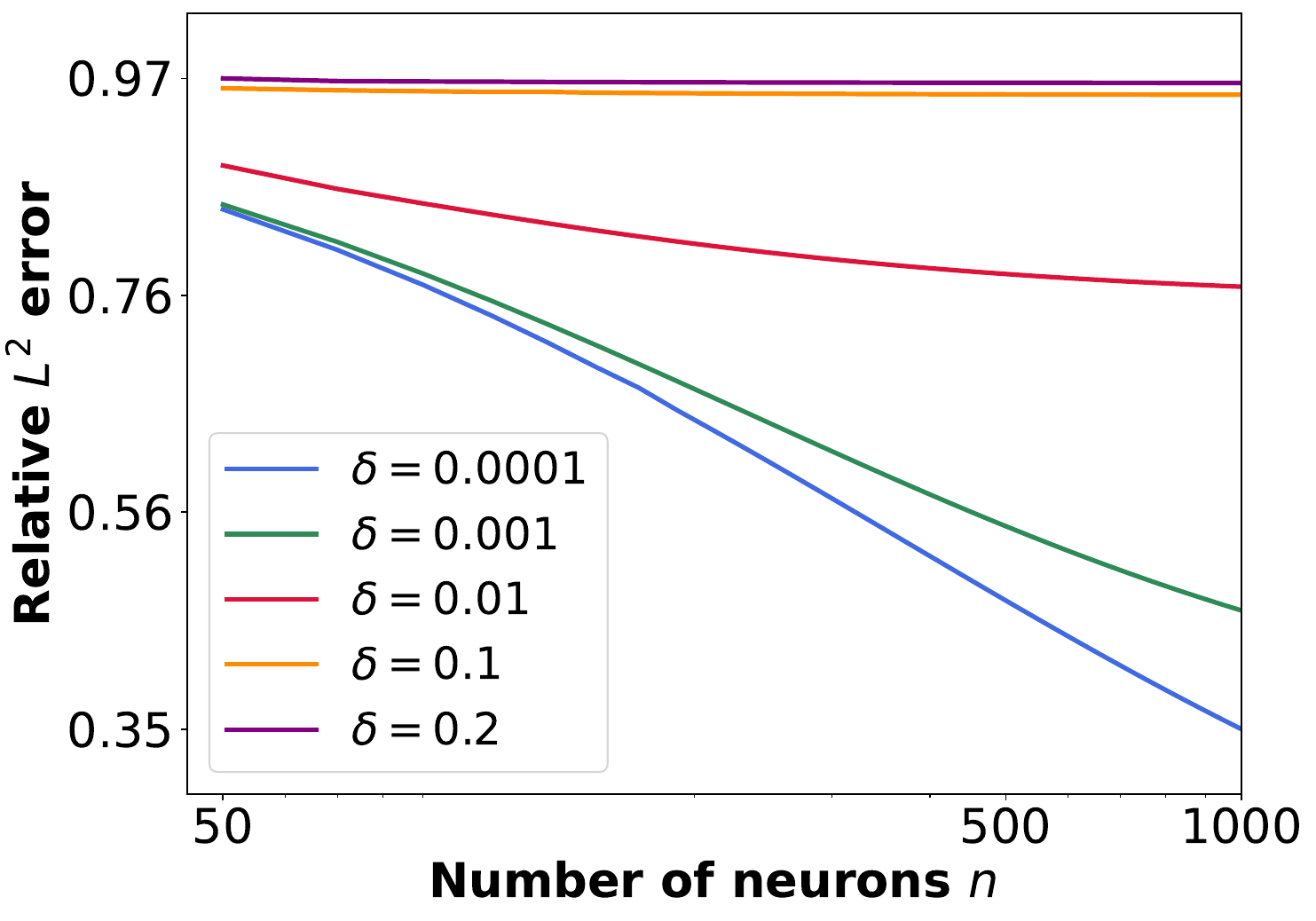}
	\end{minipage}

	\caption{1D linear integral equation: Relative $L^2$ error of $f_n^\delta$ obtained by ENN1 (top row), ENN2 (middle row), and Tikhonov-NN (bottom row). Each column corresponds to a different random seed: 111, 666, 3333 (from left to right). The marker on each curve highlights the first network architecture that satisfied the stopping criterion for each noise level, with $\tau=1.0001$. The corresponding value of $n$ (the stopping number) is indicated in the legend.}
	\label{fig:noa_error_comparison}
\end{figure}

\begin{example}
\label{Ex2}
Auto-convolution Volterra integral equations (AVIEs) of the first kind:
    \begin{equation}
\label{IntegralEq}
    A(f)=\int_{0}^{t}  f(t-s)f(s) \,ds  = g(t), \quad t\in [0,1].
\end{equation}
The auto-convolution operator $A$ is nonlinear and continuous. On a finite interval, it is not injective on the whole nonnegative cone. Following \cite{Deng2023,Fleischer1999}, injectivity holds after restricting the admissible set to
\[
D_0^+
:=
\left\{
f\in L^2(0,1): f(t)\ge0\ \text{a.e. on }[0,1],
\text{ and } f \text{ does not vanish a.e. on any interval }[0,\varepsilon],
\varepsilon>0
\right\},
\]
which is satisfied for nonnegative continuous function $f\in C[0,1]$ with $f(0)>0$. Further, for any $f,f^{\dagger}\in H^1(\Omega)$ with $\|f-f^{\dagger}\|_{H^1(\Omega)}\le \gamma$, the following local Lipschitz estimate holds
\begin{equation*}
\|A(f)-A(f^{\dagger})\|_{L^2(0,1)}\; \le\; \bigl(\|f\|_{L^1(0,1)}+\|f^{\dagger}\|_{L^1(0,1)}\bigr)\,\|f-f^{\dagger}\|_{L^2(0,1)}\le C_\Omega \left( 2\|f^\dagger\|_{H^1(\Omega)} + \gamma \right) \|f - f^\dagger\|_{H^1(\Omega)}
\end{equation*}
Thus, after restricting the admissible set to the positive cone, Example~\ref{Ex2} satisfies the injectivity and local H\"older continuity assumptions required in Theorems~\ref{main}-\ref{well-posedness} with $\theta=1$. To ensure that $f^\dagger\in D_0^+$, we generate $f^\dagger$ according to \eqref{exactf} with $T=5$ and $d=1$, where $\ell^j$ and $\phi_j$, $j=1,\ldots,T$, are independently and uniformly sampled from $[0,1]$, with $\sum_{j=1}^T\phi_j>0$.
\end{example}

\begin{figure}[htbp]
	\centering
    \begin{minipage}{0.3\linewidth}
		\centering
		\includegraphics[width=\linewidth]{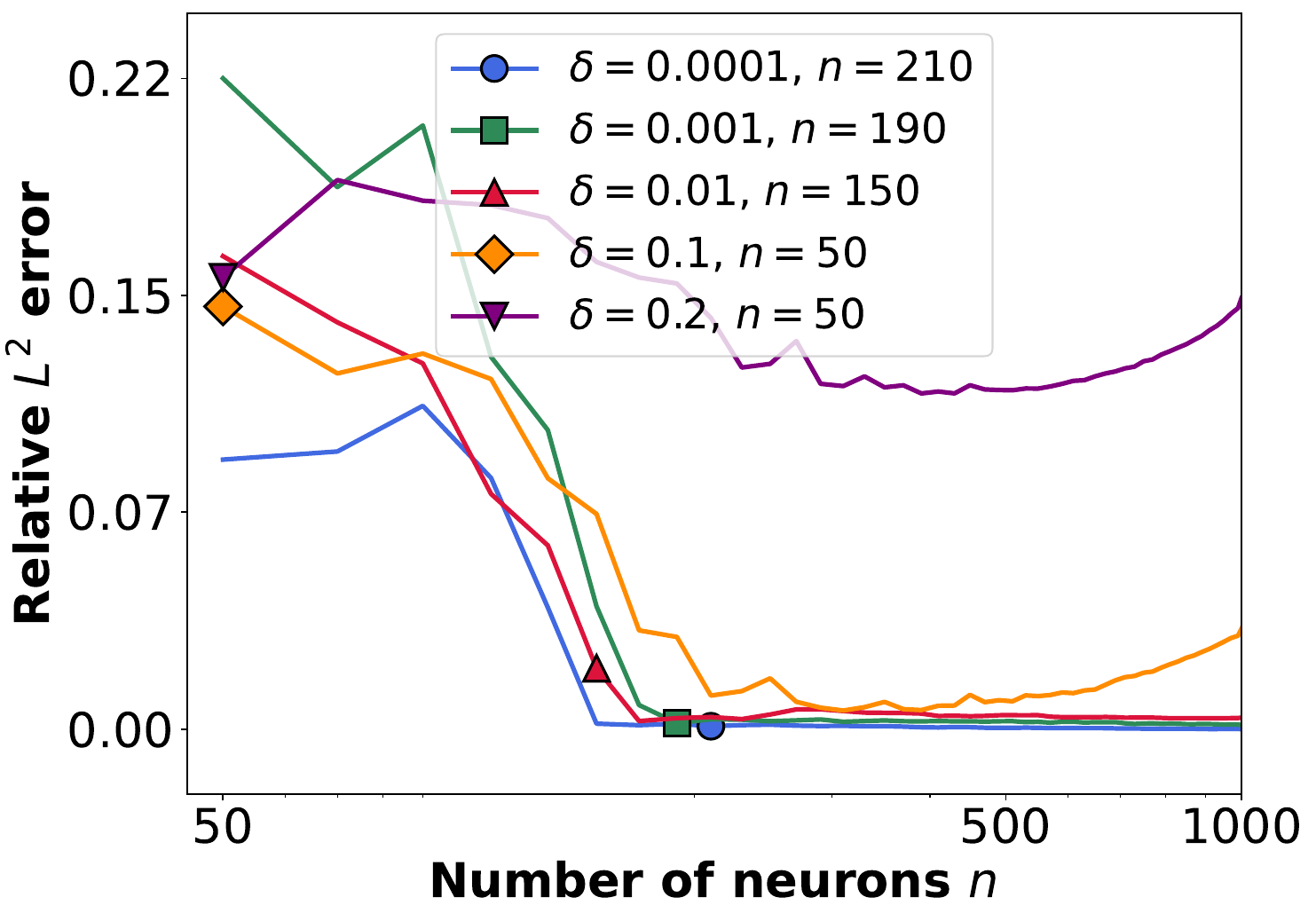}
	\end{minipage}
	\begin{minipage}{0.3\linewidth}
		\centering
		\includegraphics[width=\linewidth]{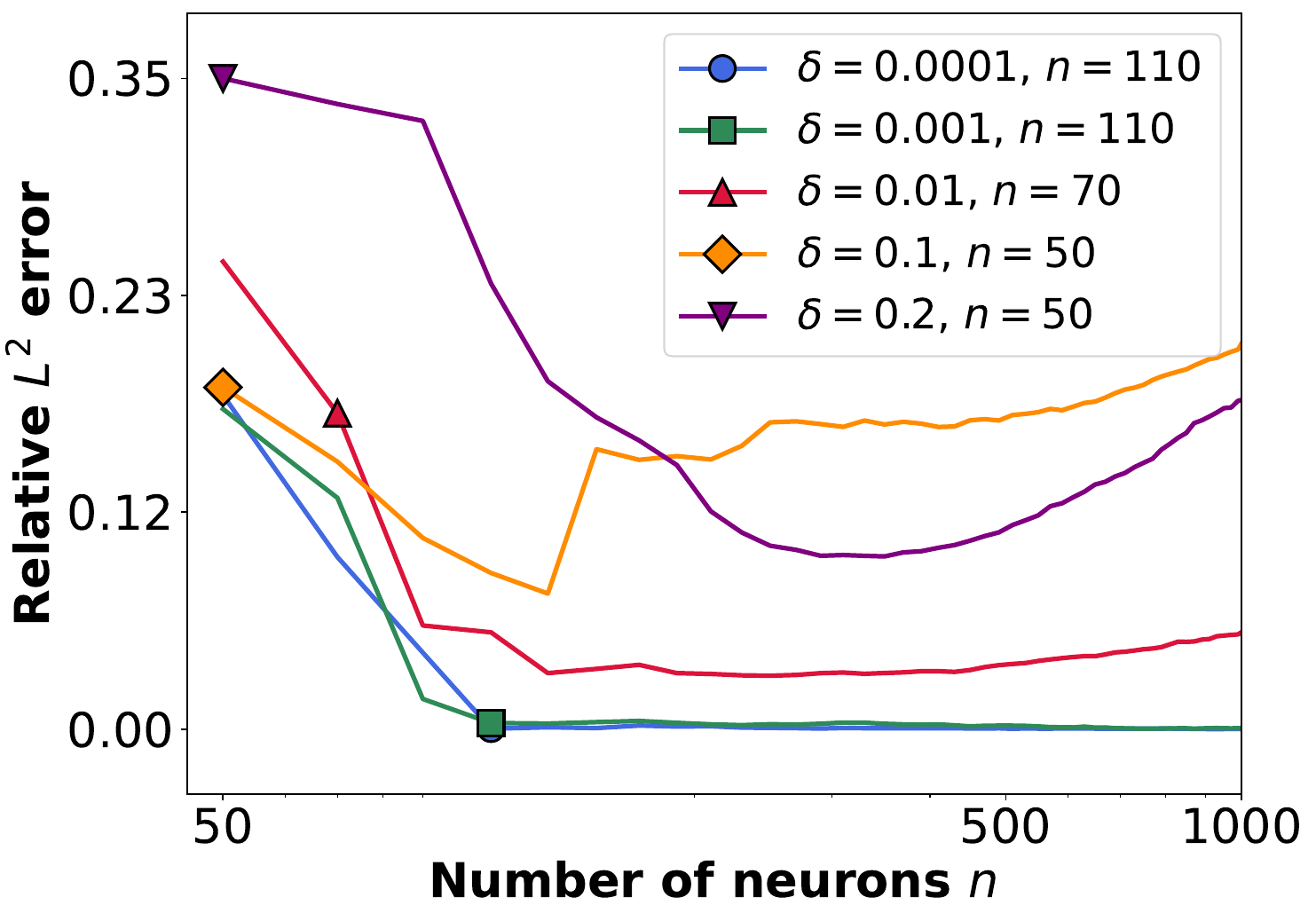}
	\end{minipage}
    \begin{minipage}{0.3\linewidth}
		\centering
		\includegraphics[width=\linewidth]{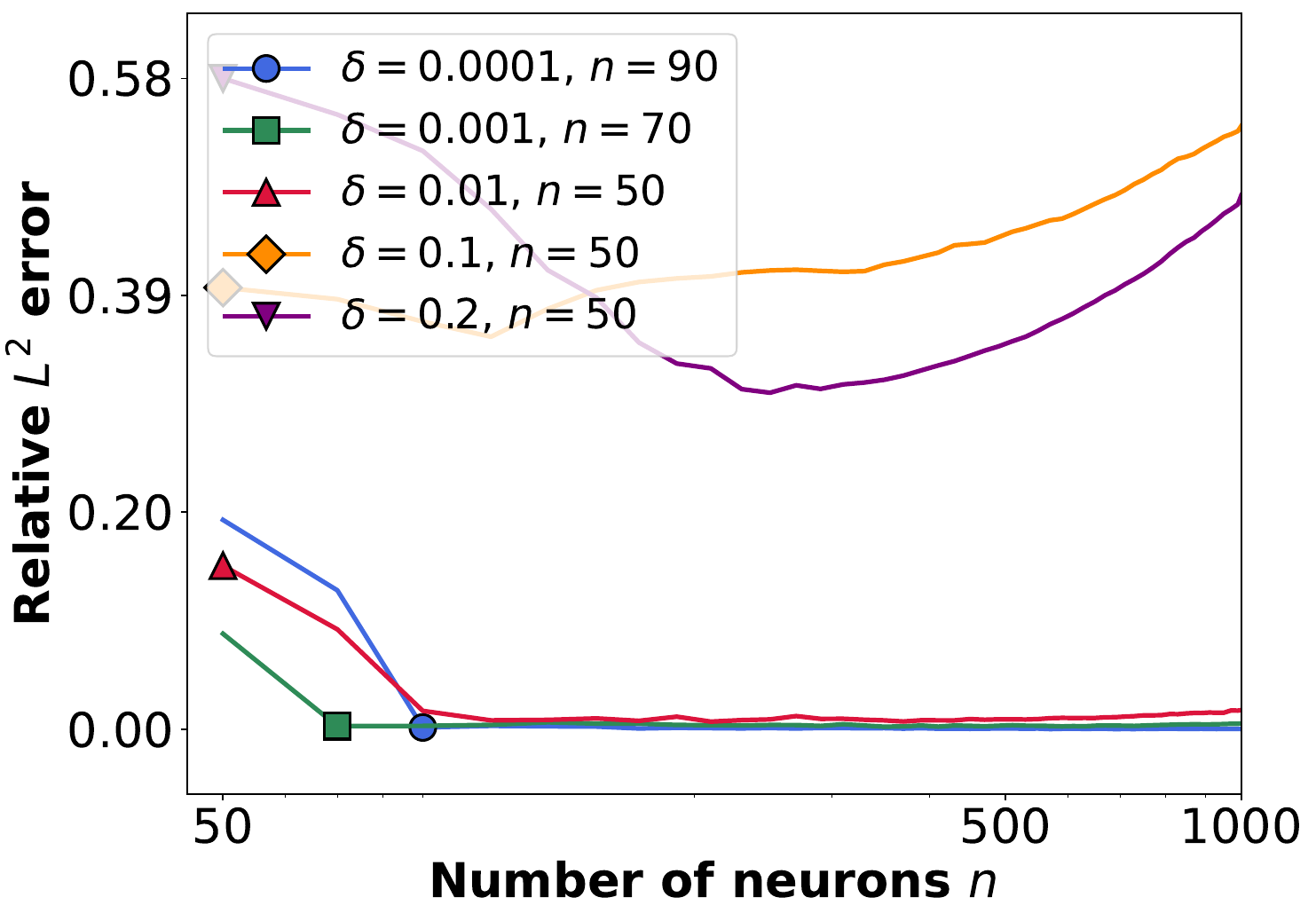}
	\end{minipage}

    \begin{minipage}{0.3\linewidth}
		\centering
		\includegraphics[width=\linewidth]{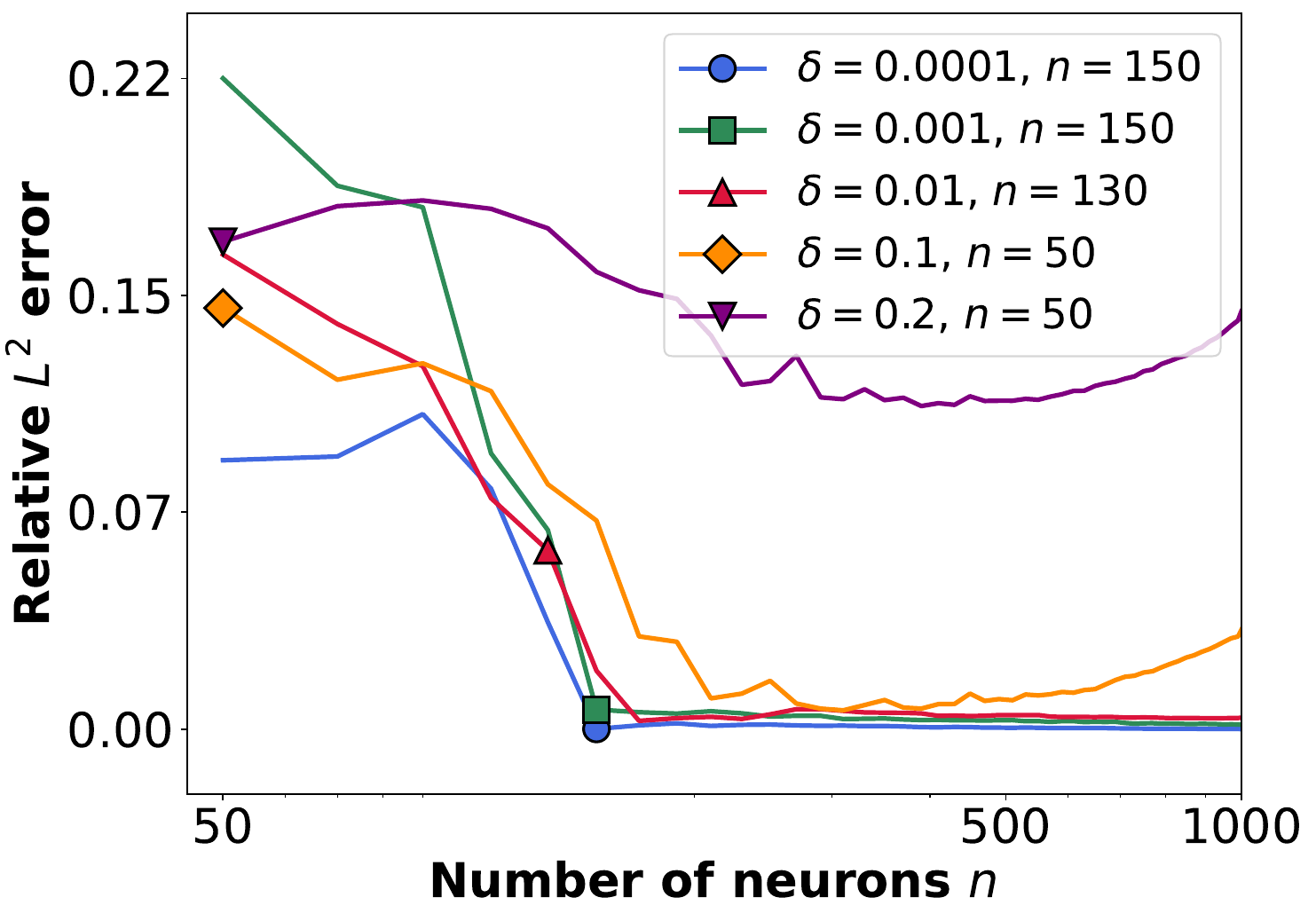}
	\end{minipage}
	\begin{minipage}{0.3\linewidth}
		\centering
		\includegraphics[width=\linewidth]{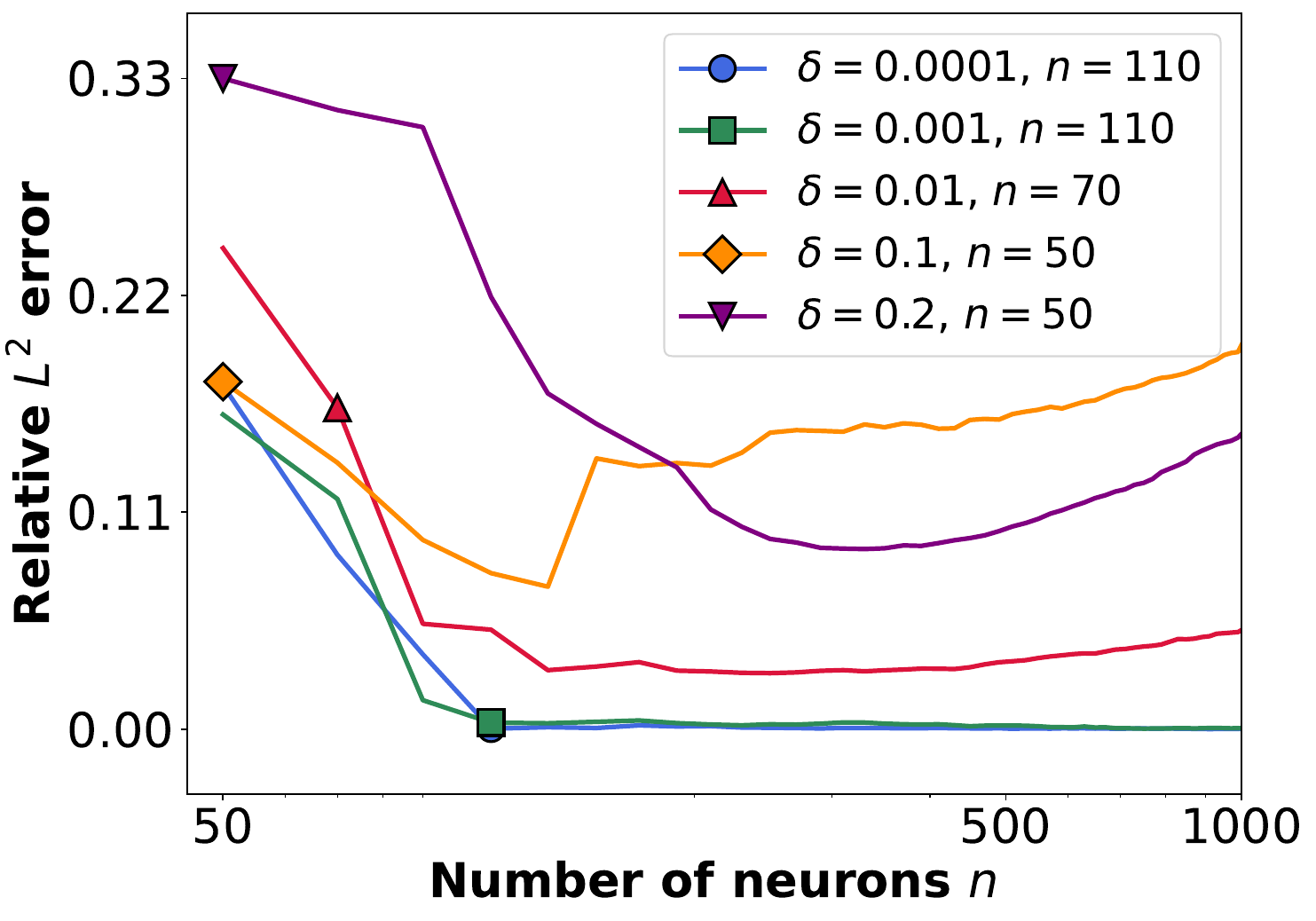}
	\end{minipage}
    \begin{minipage}{0.3\linewidth}
		\centering
		\includegraphics[width=\linewidth]{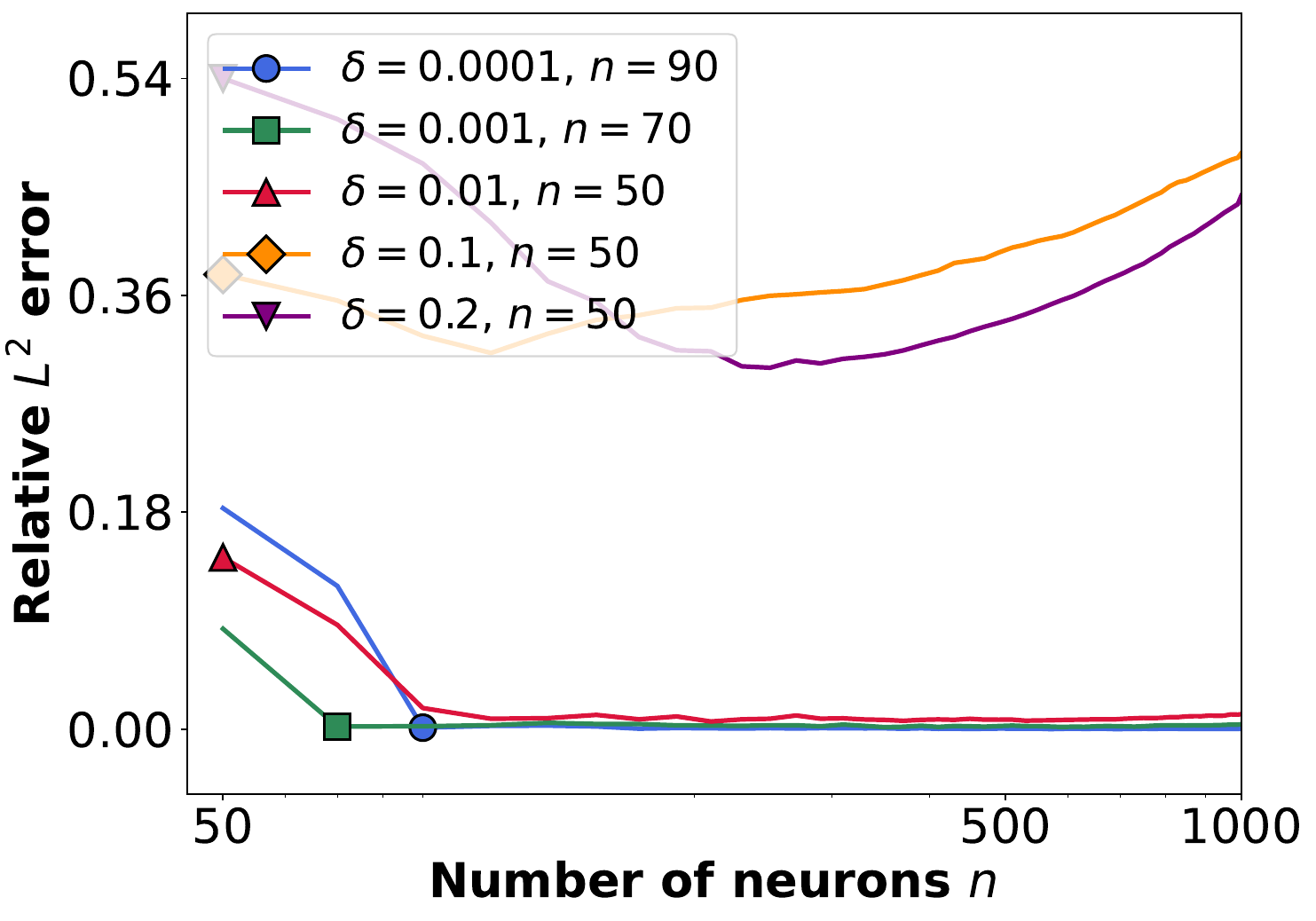}
	\end{minipage}

    \begin{minipage}{0.3\linewidth}
		\centering
		\includegraphics[width=\linewidth]{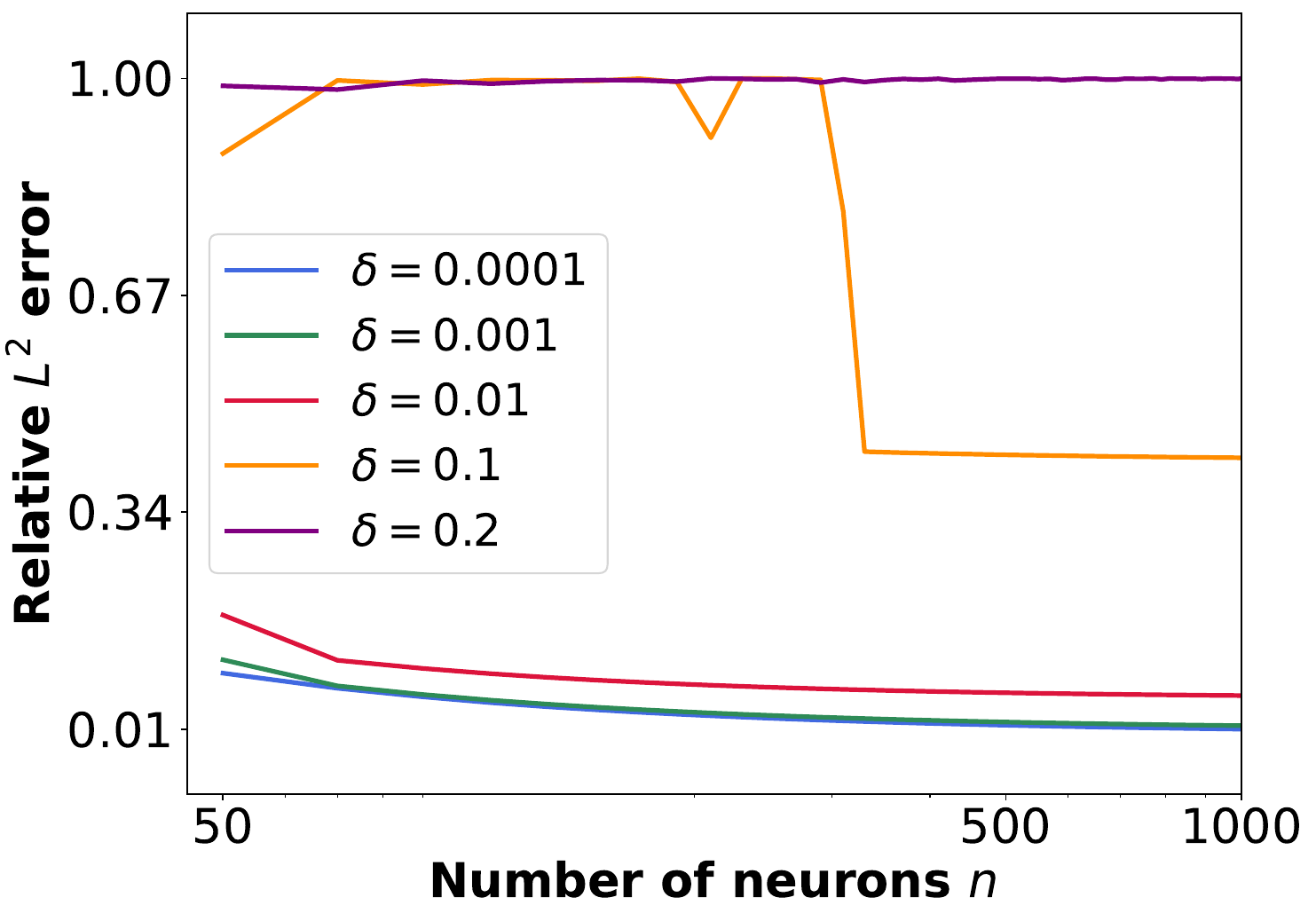}
	\end{minipage}
	\begin{minipage}{0.3\linewidth}
		\centering
		\includegraphics[width=\linewidth]{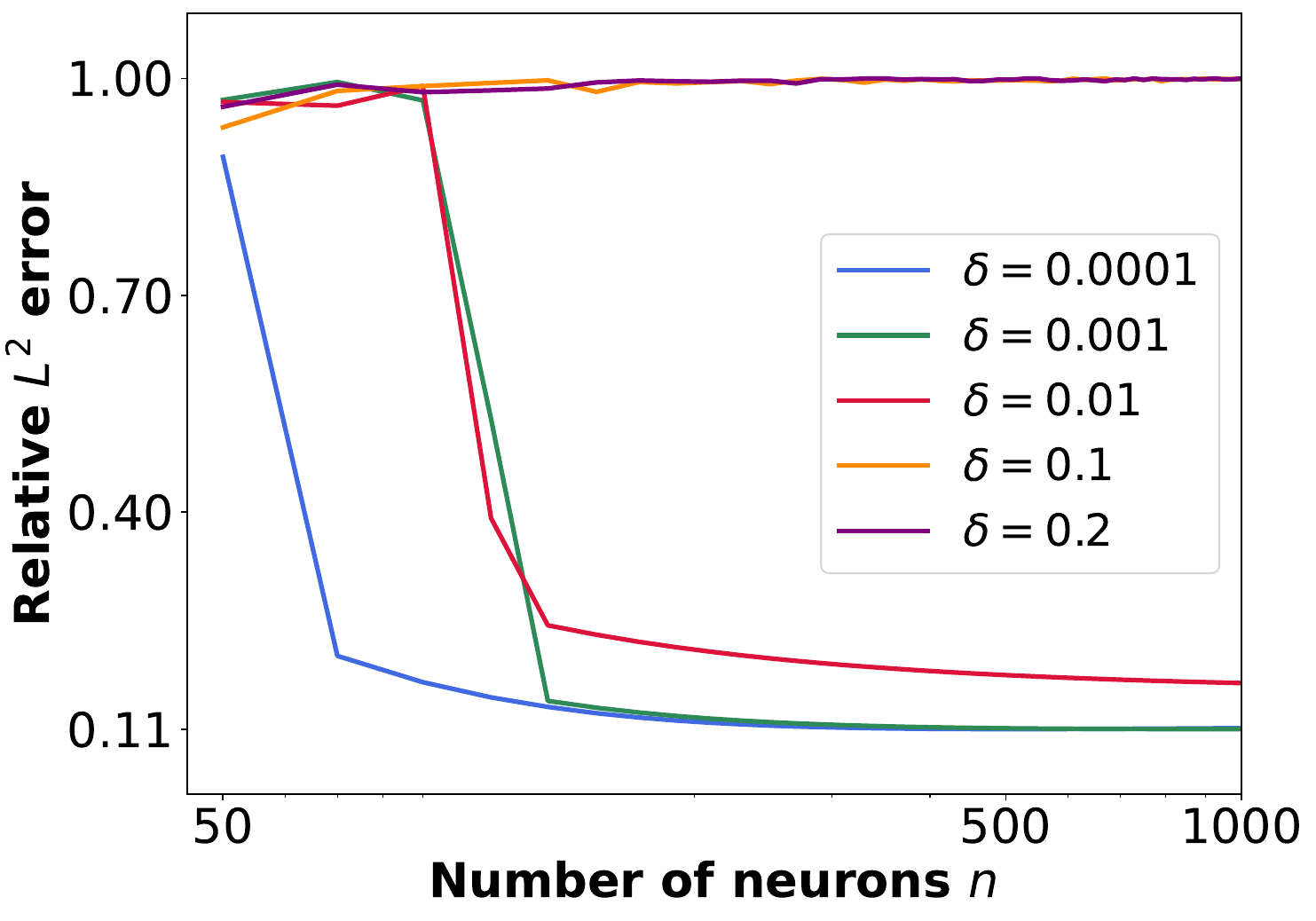}
	\end{minipage}
   \begin{minipage}{0.3\linewidth}
		\centering
		\includegraphics[width=\linewidth]{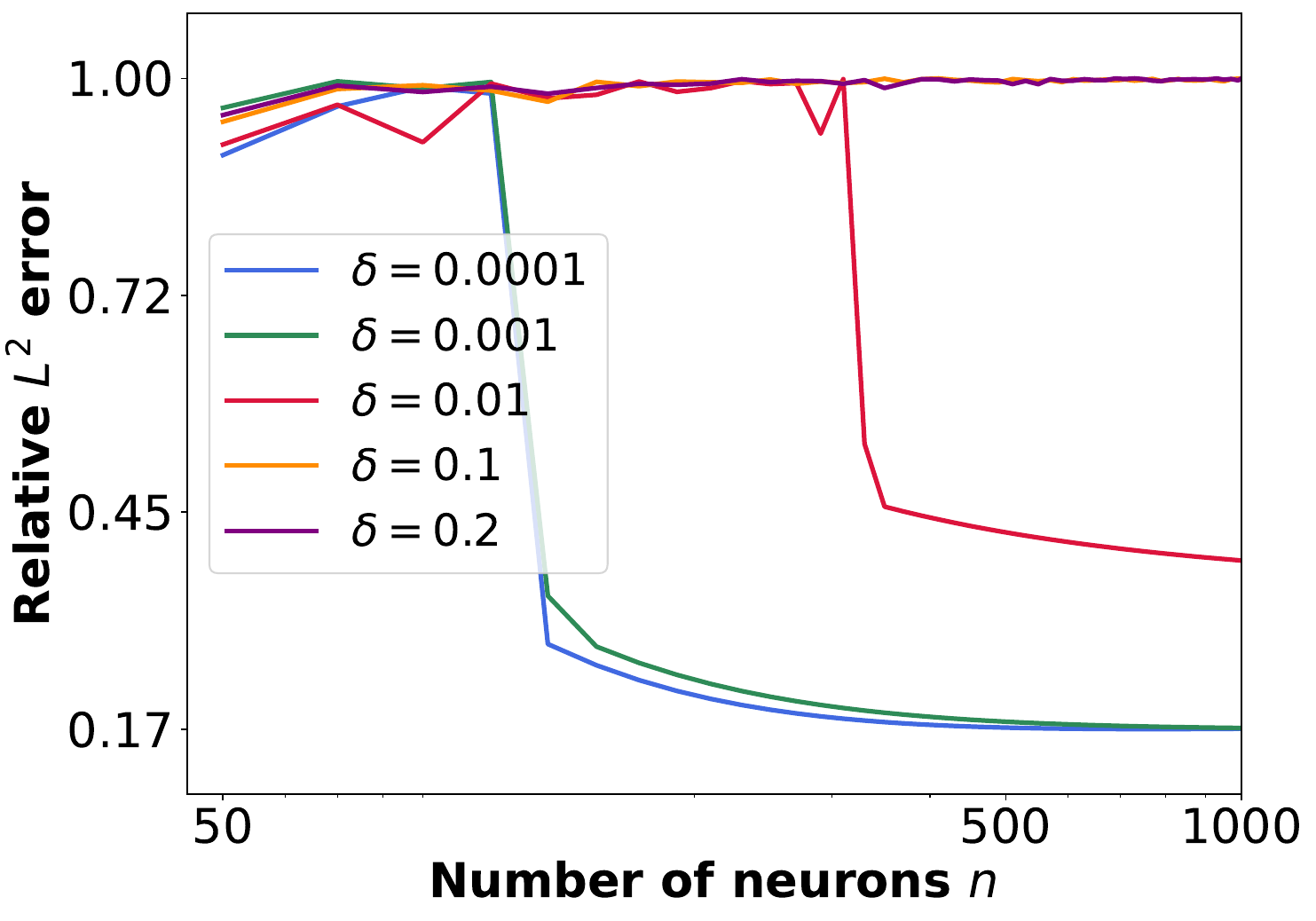}
	\end{minipage}

	\caption{Auto-convolution equation: Relative $L^2$ error of $f_n^{\delta}$ obtained by ENN1 (top row), ENN2 (middle row), and Tikhonov-NN (bottom row). Each column corresponds to a different random seed: 678, 765, 987 (from left to right). The marker on each curve highlights the first network architecture that satisfied the stopping criterion for each noise level, with $\tau=1.0001$. The corresponding value of $n$ (the stopping number) is indicated in the legend. In some figures, the blue circles ($\delta=0.0001$) and green squares ($\delta=0.001$) coincide.}
	\label{fig:auto_relative_L2_error_comparison}
\end{figure}

\begin{example}
\label{Ex3}
Electrical impedance tomography (EIT) on the unit square $\Omega=(0,1)^2$:
\begin{equation}
\label{EIT_equation}
\nabla\cdot(f\nabla u)=0
\quad\text{in }\Omega,
\qquad
u=h
\quad\text{on }\partial\Omega.
\end{equation}
where $f$ denotes the electrical conductivity, and $u$ is the electrical potential. The corresponding boundary current is
$\left.f\partial_\nu u\right|_{\partial\Omega}$, where $\nu$ denotes the outward unit normal. The DN map $\Lambda_f:H^{1/2}(\partial\Omega)\to H^{-1/2}(\partial\Omega)$ is defined by $\Lambda_fh:=\left.f\partial_\nu u\right|_{\partial\Omega}$. The inverse problem is to determine $f$ from $g=A(f)$, where $A(f):=\Lambda_f$ and $g:=\Lambda_{f^\dagger}$.
Consider the admissible class
\[
\mathcal D_M
:=
\left\{
f\in W^{1,\infty}(\Omega):
\alpha\le f(\boldsymbol{x})\le \beta
\ \text{for a.e. } \boldsymbol{x}\in\Omega,\ 
\|f\|_{W^{1,\infty}(\Omega)}\le M
\right\},
\qquad 0<\alpha<\beta<\infty .
\]
For $f\in\mathcal D_M$, the boundary value problem
\eqref{EIT_equation} is uniformly elliptic, and hence the DN map is well defined. Moreover, the uniqueness result in \cite{AstalaPaivarinta2006} implies that $A:f\mapsto\Lambda_f$ is injective on $\mathcal D_M$. By Alessandrini's identity and standard energy estimates, together with the Gagliardo--Nirenberg interpolation inequality, for
$f,f^\dagger\in\mathcal D_M$,
\[
\begin{aligned}
&\|\Lambda_f-\Lambda_{f^\dagger}\|_{
\mathcal L(H^{1/2}(\partial\Omega),H^{-1/2}(\partial\Omega))}
\le C\|f-f^\dagger\|_{L^\infty(\Omega)} \\
&\quad \quad \le
C_\Omega
\|f-f^\dagger\|_{L^2(\Omega)}^{1/2}
\|f-f^\dagger\|_{W^{1,\infty}(\Omega)}^{1/2} \le
C_{\Omega,M}
\|f-f^\dagger\|_{H^1(\Omega)}^{1/2}.
\end{aligned}
\]
 Thus \(A\) is locally H\"older with respect to the \(H^1(\Omega)\) metric on \(\mathcal D_M\), with an exponent \(1/2\). In practice, the DN map is replaced by a finite number of boundary measurements. For example, for one prescribed voltage pattern $h_1$, we use $A_1(f)=\Lambda_f h_1$, where
\[
h_1(x,y)=
\begin{cases}
\sin(\pi x), & y=0,\\
0, & \text{otherwise on } \partial\Omega .
\end{cases}
\]
The true conductivity is taken as $f^\dagger(x,y)=x+y+1$. Using multiple independent voltage patterns would provide richer boundary information and generally improve the reconstruction accuracy.
\end{example}

\begin{figure}[htbp]
	\centering
	\begin{minipage}{0.3\linewidth}
		\centering
		\includegraphics[width=\linewidth]{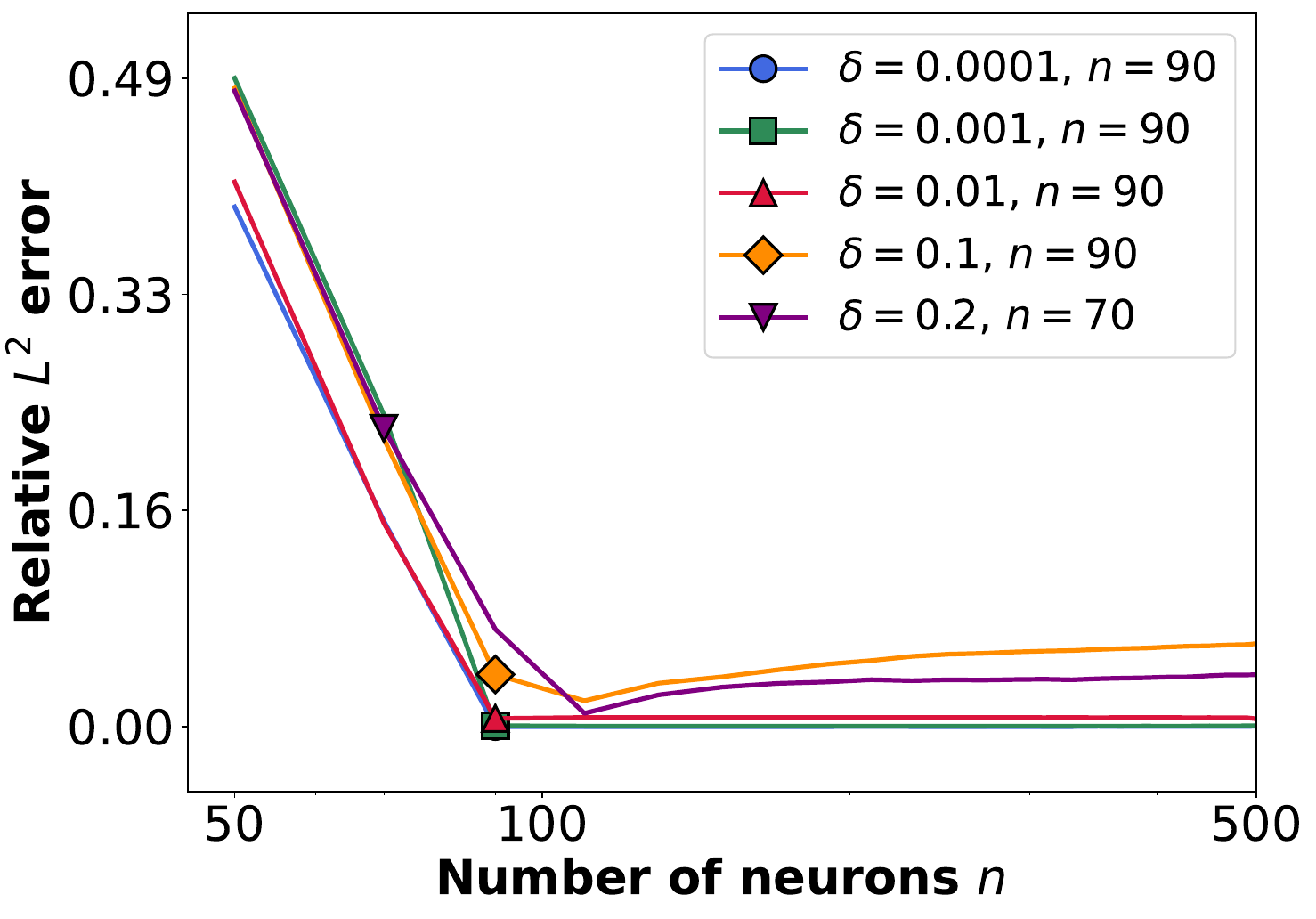}
	\end{minipage}
	\begin{minipage}{0.3\linewidth}
		\centering
		\includegraphics[width=\linewidth]{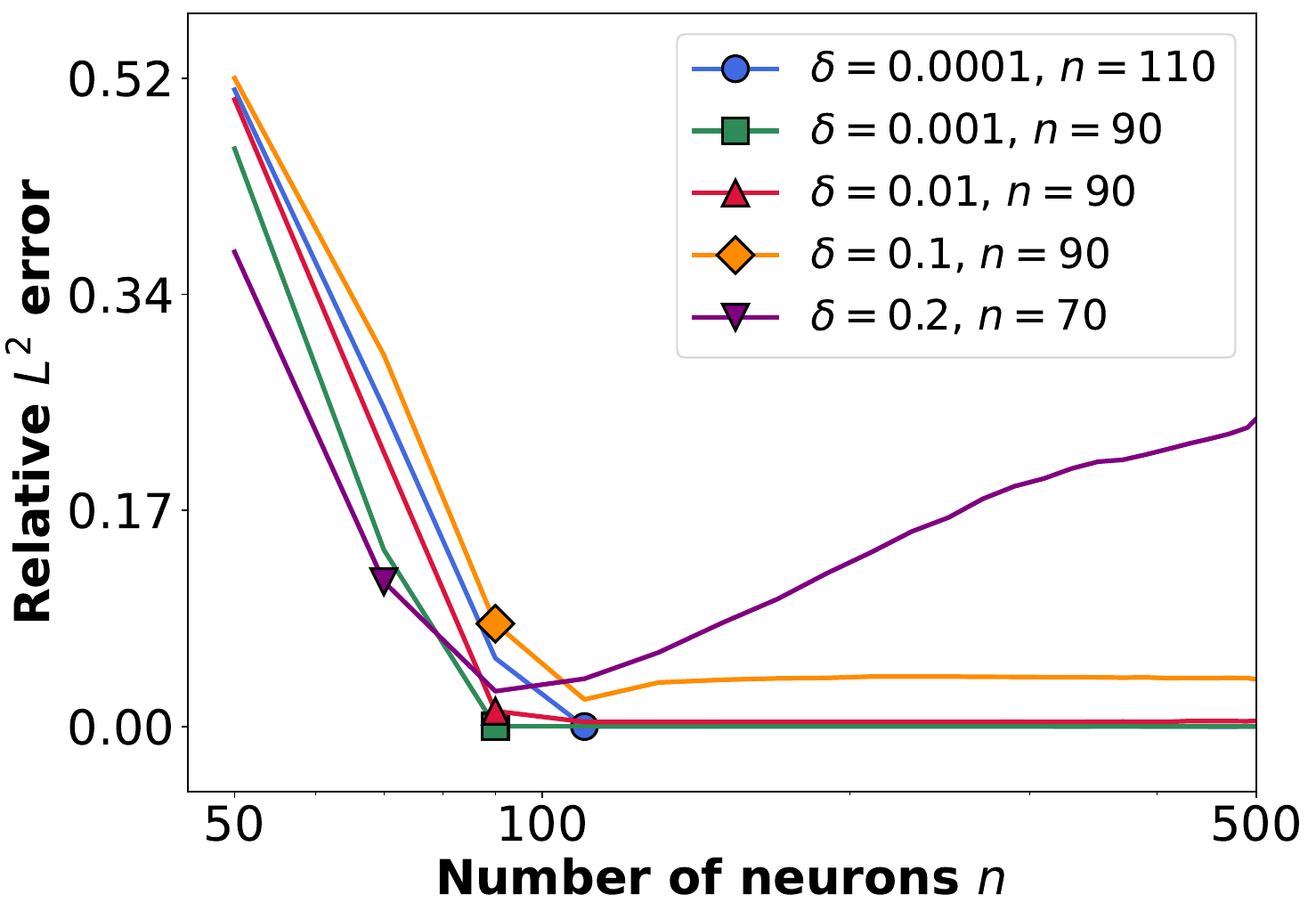}
	\end{minipage}
    \begin{minipage}{0.3\linewidth}
		\centering
		\includegraphics[width=\linewidth]{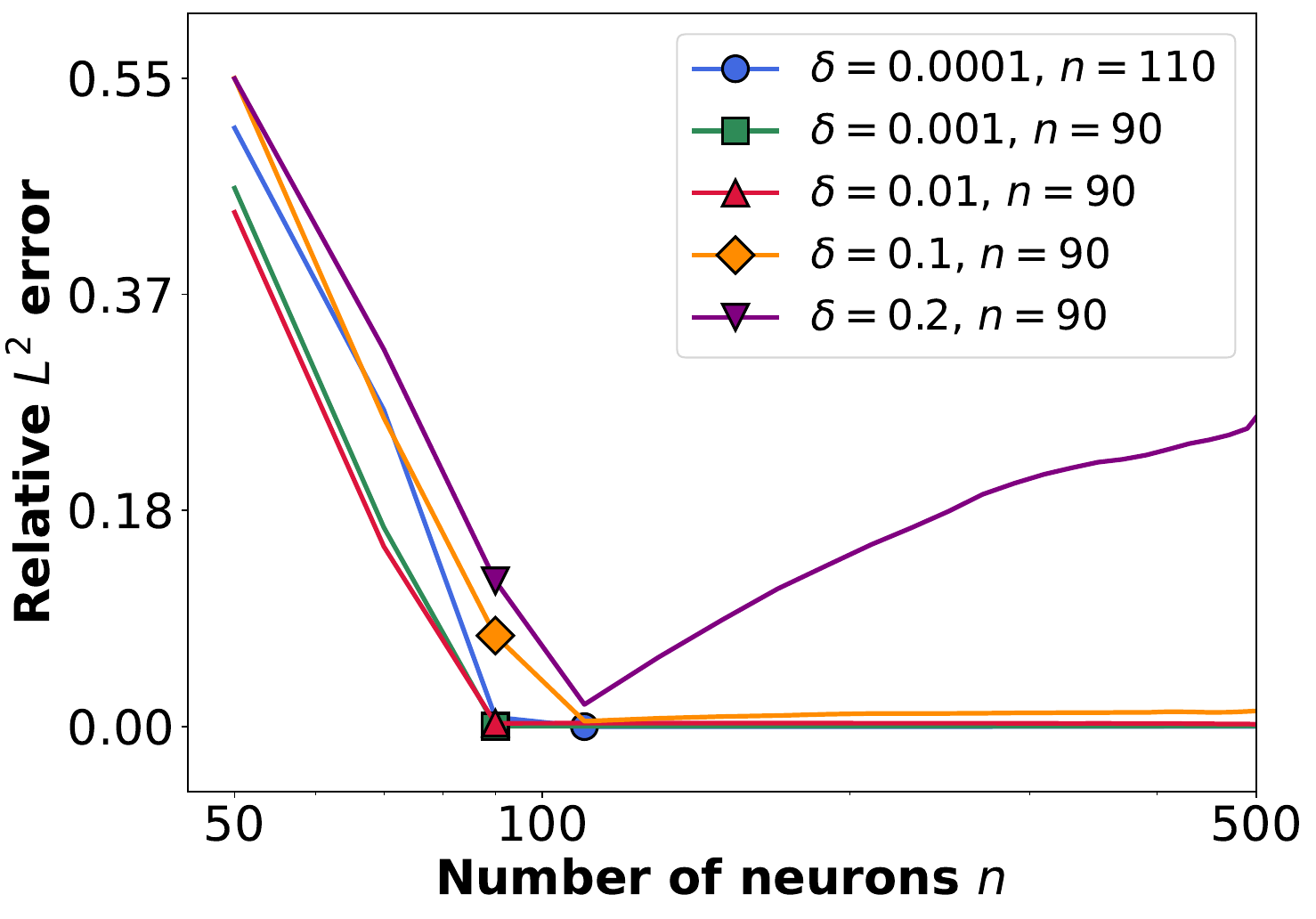}
	\end{minipage}
	
	\begin{minipage}{0.3\linewidth}
		\centering
		\includegraphics[width=\linewidth]{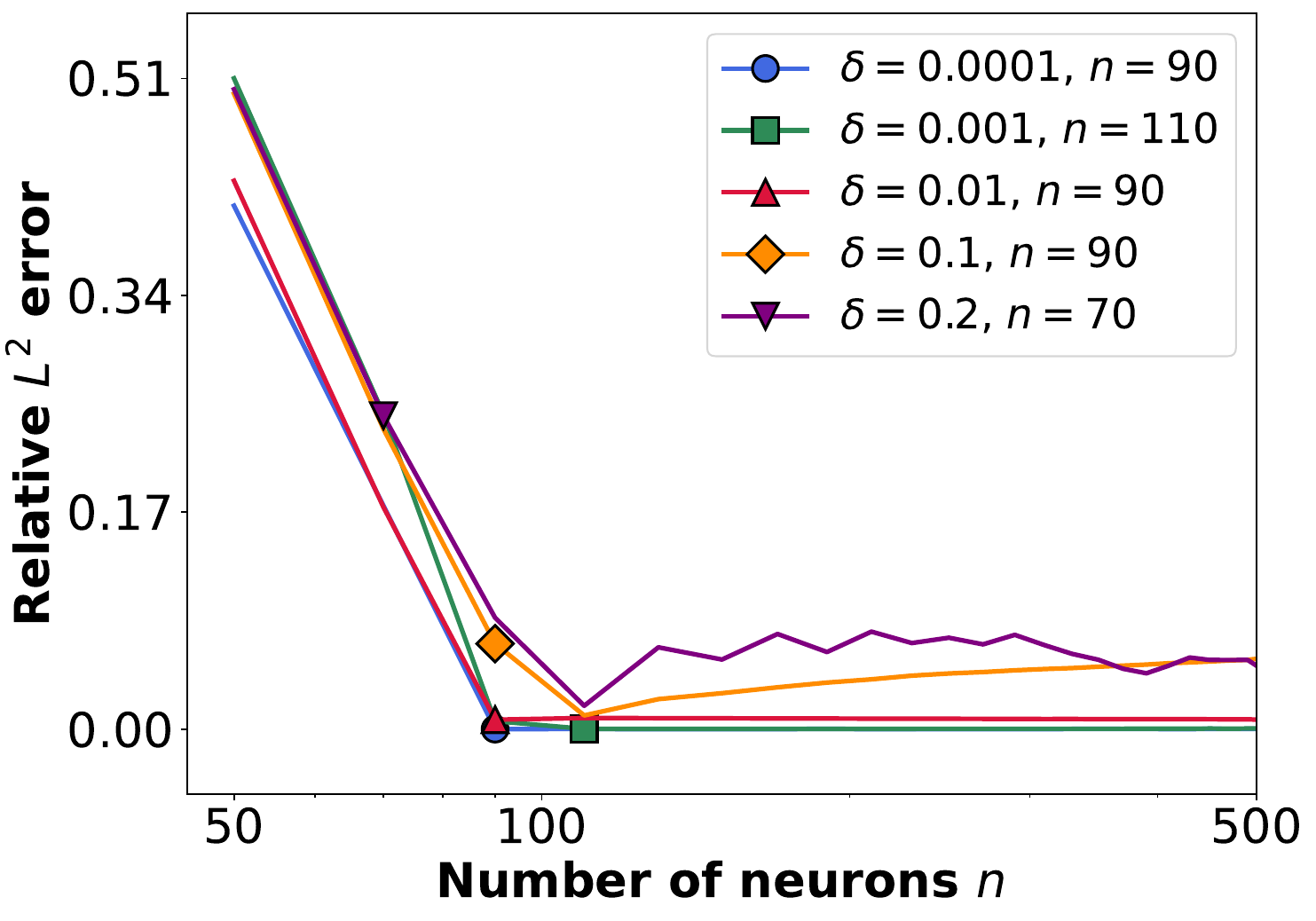}
	\end{minipage}
	\begin{minipage}{0.3\linewidth}
		\centering
		\includegraphics[width=\linewidth]{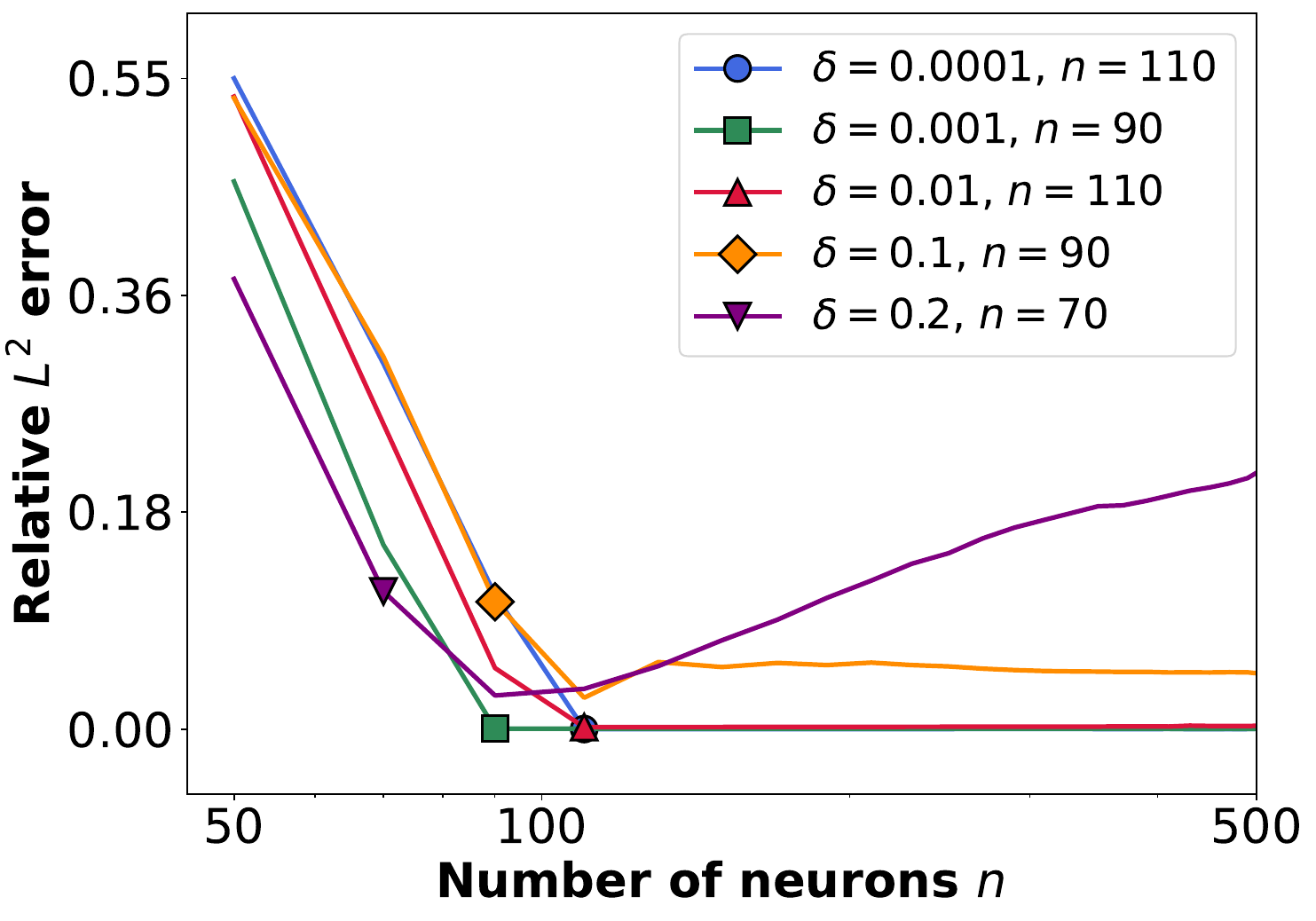}
	\end{minipage}
    \begin{minipage}{0.3\linewidth}
		\centering
		\includegraphics[width=\linewidth]{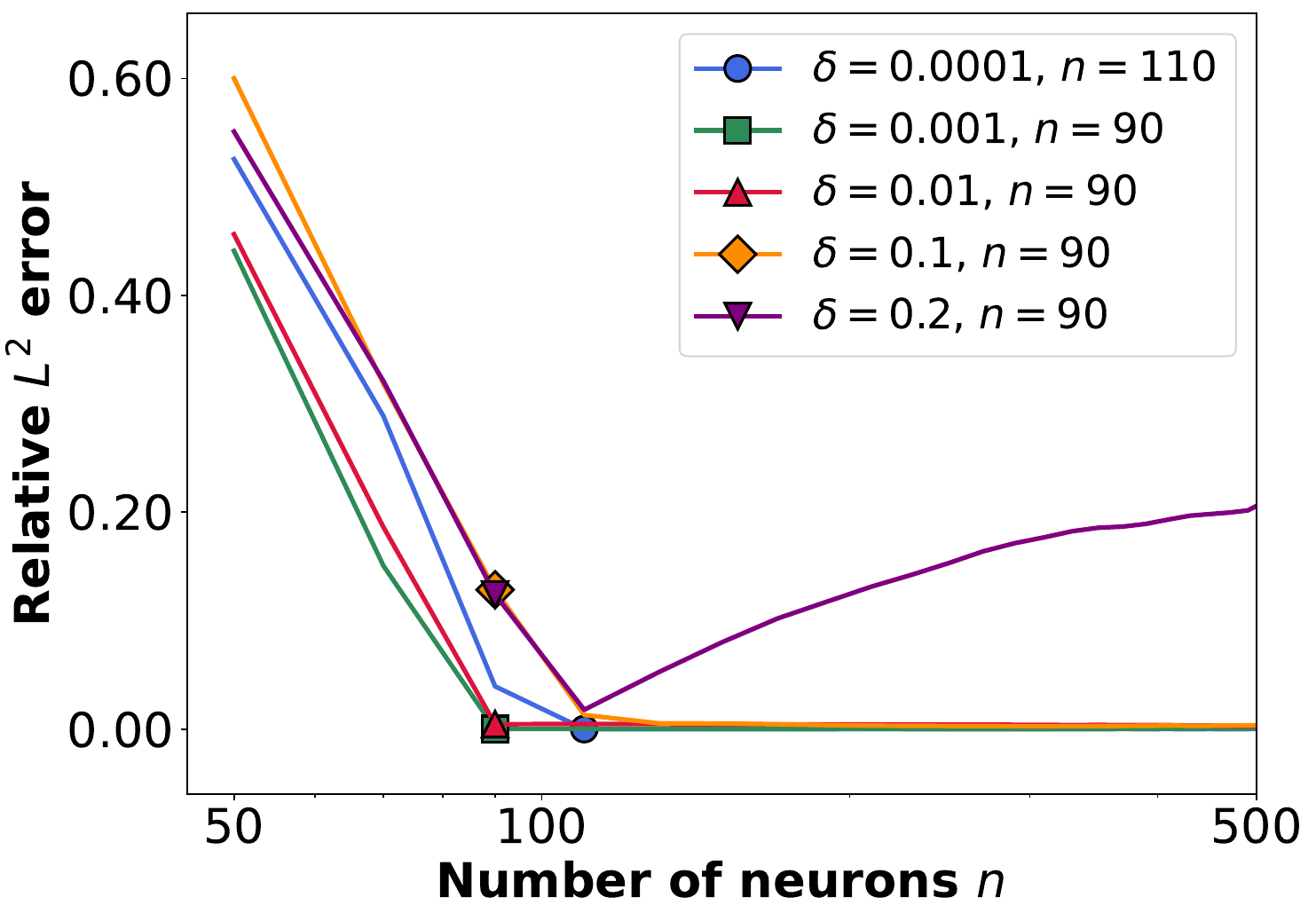}
	\end{minipage}
	
	\begin{minipage}{0.3\linewidth}
		\centering
		\includegraphics[width=\linewidth]{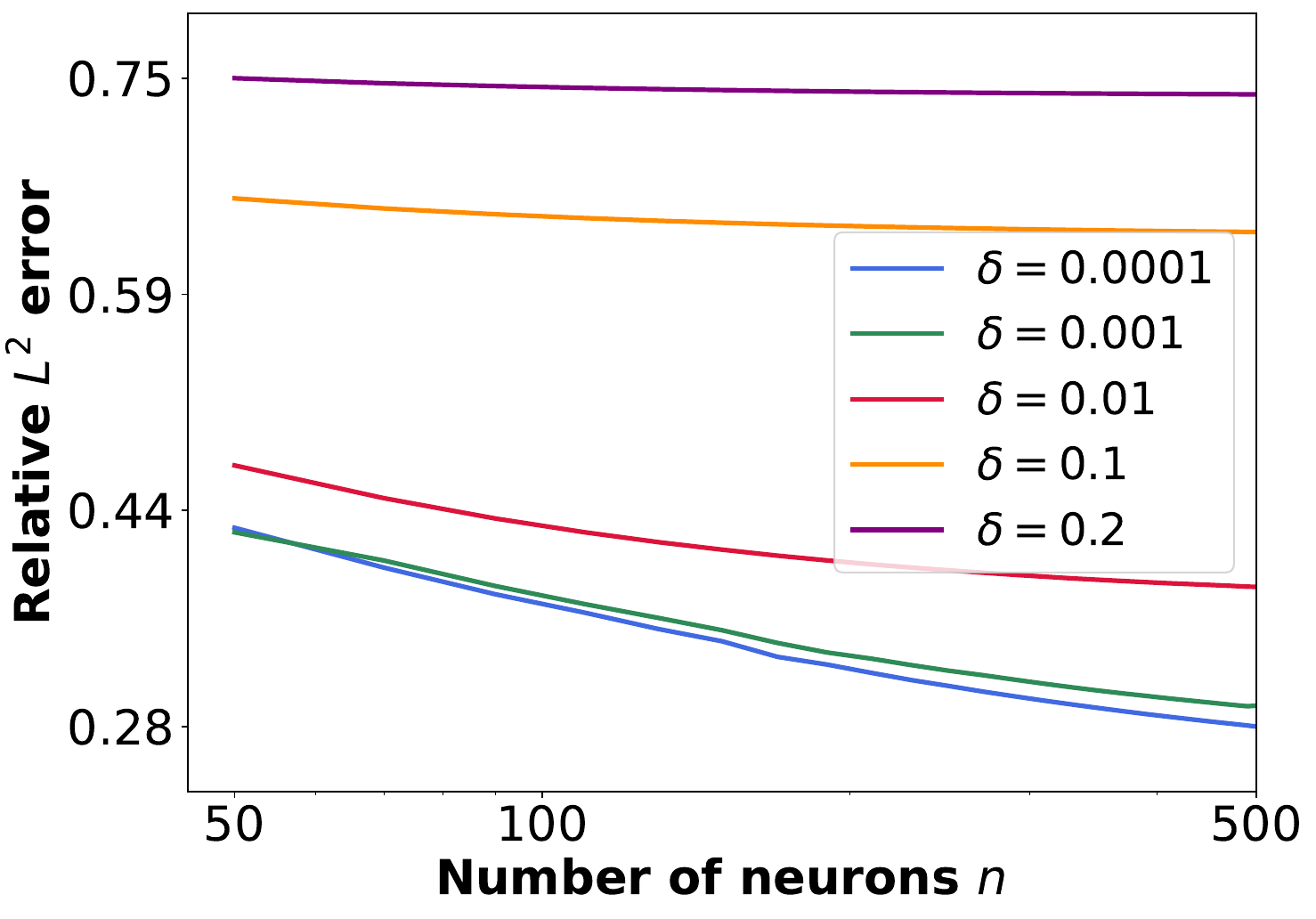}
	\end{minipage}
	\begin{minipage}{0.3\linewidth}
		\centering
		\includegraphics[width=\linewidth]{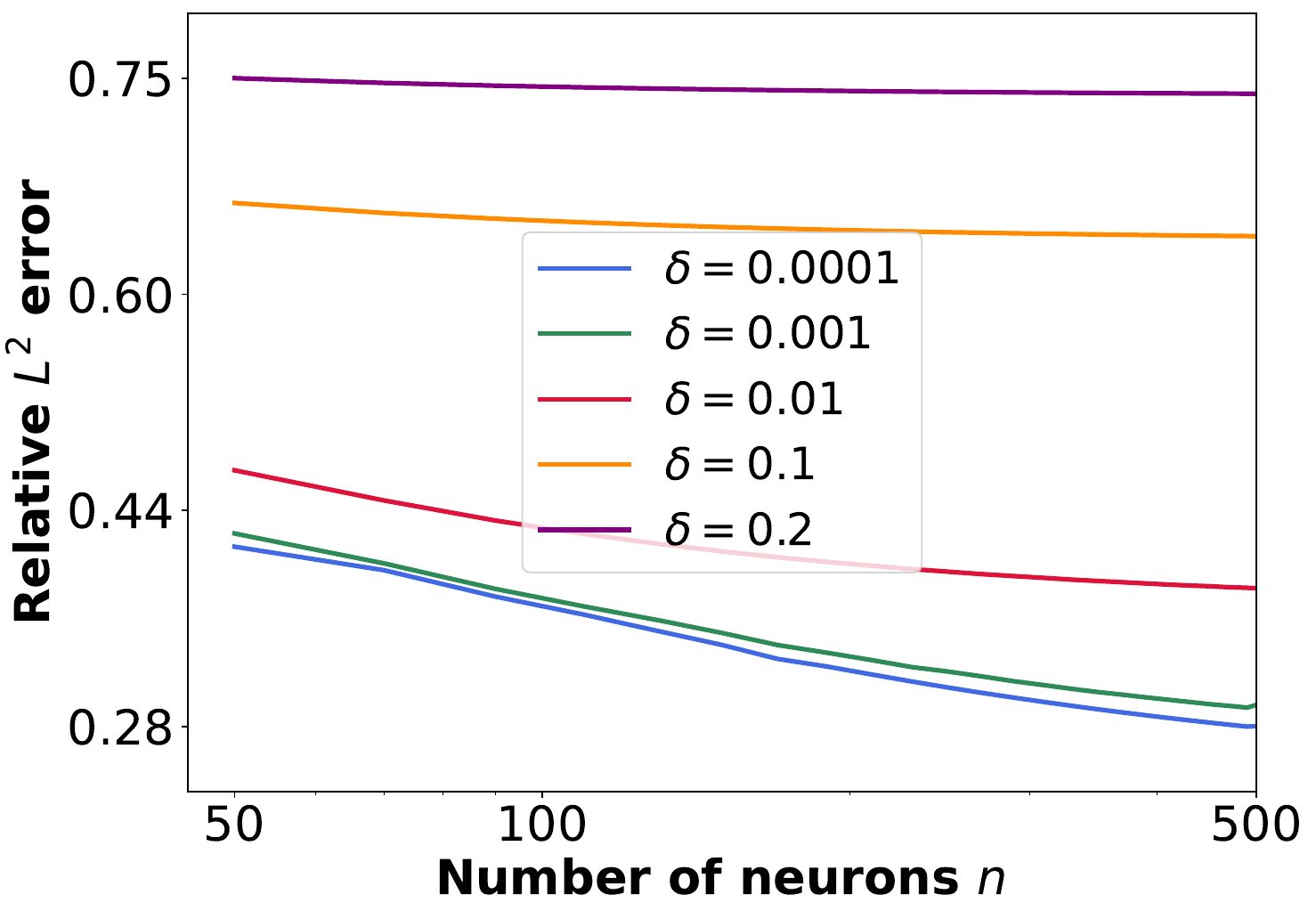}
	\end{minipage}
    \begin{minipage}{0.3\linewidth}
		\centering
		\includegraphics[width=\linewidth]{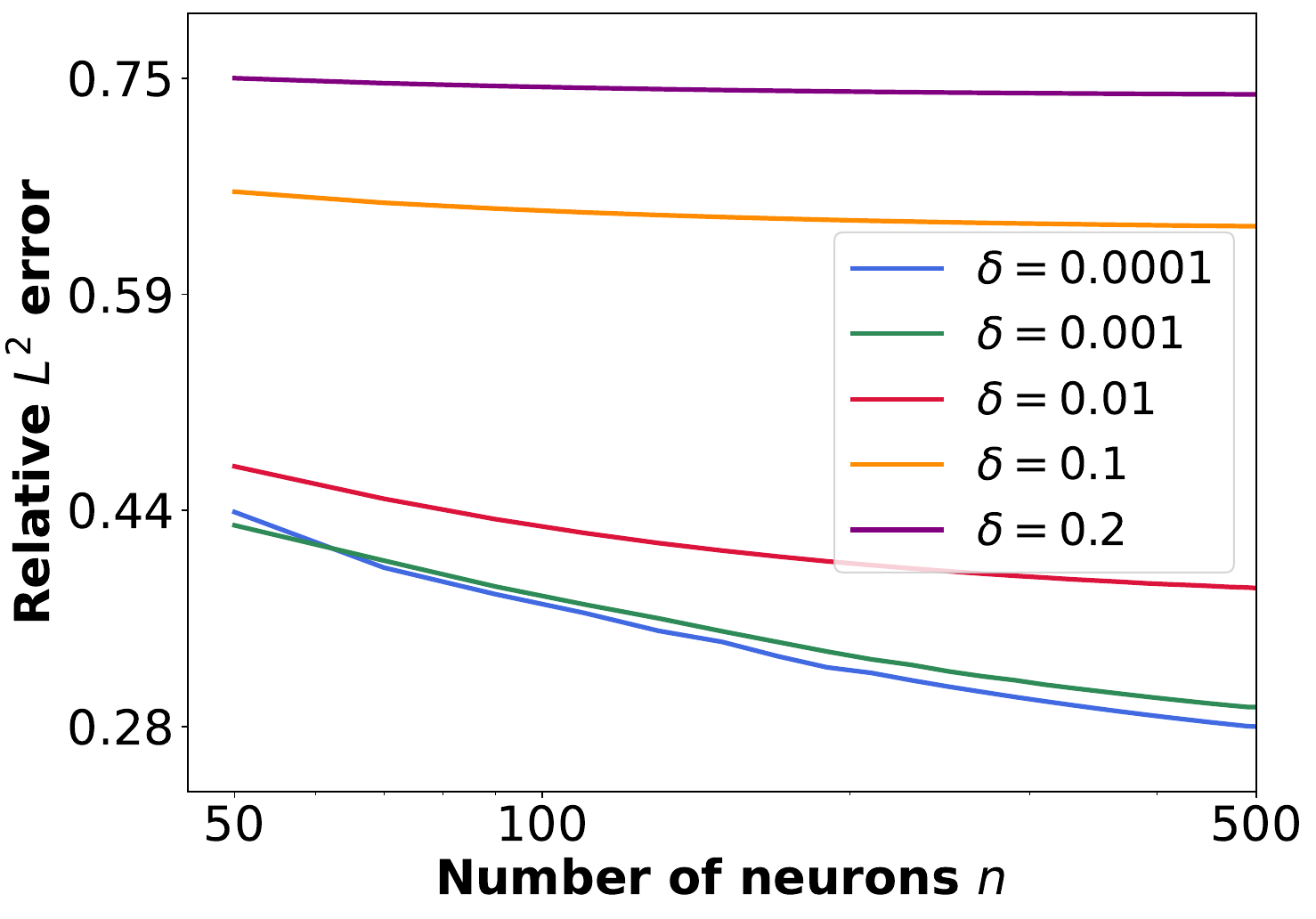}
	\end{minipage}

	\caption{EIT: Relative $L^2$ error of $f_n^{\delta}$ obtained by ENN1  (top row), ENN2  (middle row), and Tikhonov-NN (bottom row). Each column corresponds to a different random seed: 20, 30, 40 (from left to right). The marker on each curve highlights the first network architecture that satisfied the stopping criterion for each noise level, with $\tau=1.0001$. The corresponding value of $n$ (the stopping number) is indicated in the legend. In some figures, the blue circles ($\delta=0.0001$) and green squares ($\delta=0.001$) coincide.}
	\label{fig:EIT_relative_L2_error_comparison}
\end{figure}

The numerical results for Examples \ref{Ex1}-\ref{Ex3} are reported in Figures \ref{fig:noa_error_comparison}, \ref{fig:auto_relative_L2_error_comparison}, and \ref{fig:EIT_relative_L2_error_comparison}, respectively. These plots show how the relative $L^2$ error of the neural network approximations produced by the three methods evolves with the number of neurons under different random seeds.

In the first column of Fig. \ref{fig:noa_error_comparison}, the solutions generated by ENN1 and ENN2 display a pronounced instability, namely, the relative error increases with the number of neurons $n$ rather than decreasing. This reflects the intrinsic ill-posedness of the underlying inverse problem: as the network capacity grows, the approximation becomes increasingly sensitive to noise, which manifests as instability and error amplification. By contrast, Tikhonov-NN effectively suppresses this phenomenon. The improvement arises from the presence of the regularization term, which imposes an additional constraint that balances the fidelity to the data with the stability of the solution, thereby damping spurious oscillations and preventing uncontrolled error growth. This underscores the stabilizing effect of Tikhonov regularization, which provides a stronger balance between accuracy and stability than the alternative regularization strategies.

Beyond this specific behavior, more general trends can be identified across Figures~\ref{fig:noa_error_comparison}, \ref{fig:auto_relative_L2_error_comparison}, and \ref{fig:EIT_relative_L2_error_comparison}. In general, the error of the neural network approximation decreases as the noise level $\delta$ decreases and as the number of neurons $n$ increases. However, for larger noise levels (e.g., $\delta=0.1,0.2$), the approximation error first decreases and then increases with growing $n$, leading to a typical semi-convergence, or U-shaped, behavior. This observation suggests that, under high noise conditions, a simpler and smaller network architecture can already achieve satisfactory performance, while blindly increasing the network size does not necessarily improve the reconstruction accuracy and may instead lead to overfitting or degraded performance. 

In addition, the points marked for the stopping index $n$ of ENN1 and ENN2 indicate that both algorithms successfully terminate at finite iteration steps in all reported experiments. In high-noise cases, the discrepancy threshold is larger, and the algorithm may stop before the true reconstruction error reaches its minimum. This is consistent with the classical semi-convergence: the discrepancy principle is designed to ensure stability rather than to minimize the unknown true error. Moreover, the stopping $n$ shows that $n=n(\delta)$ generally increases as $\delta$ decreases, and is qualitatively consistent with the estimate $n(\delta)=\mathcal O(\delta^{-2/\theta})$  in Theorems \ref{main} and \ref{main2}. 

Regarding the Tikhonov-NN method, we examine the last row of the three figures. For small noise levels, two types of behavior are observed as the number of neurons increases. In some cases, the reconstruction error decreases initially and then approaches a plateau, whereas in others it continues to decrease up to the largest network width considered. The latter behavior suggests that the stabilization regime may not yet have been reached within the prescribed maximum width $n_{\max}$.
 Moreover, for relatively large noise levels (e.g., $\delta=0.1,0.2$), the reconstruction error is dominated by noise and therefore cannot be significantly reduced by increasing the number of neurons. Overall, in the non-blow-up cases, ENN1 and ENN2 as well as the Tikhonov-NN method are all capable of constructing neural network approximations with good accuracy. In comparison, ENN1 and ENN2 tend to achieve approximations with smaller errors using relatively fewer neurons than the Tikhonov-NN approach, whereas the Tikhonov-NN method provides greater robustness and stability.

\section{Conclusion and outlook}
\label{sec:Con}

In this work, we have investigated solving ill-posed inverse problems using shallow neural networks, including expanding two-layer network classes with increasing width and controlled representation cost, so that the network size can be coupled with the noise level. For iterative regularization, we develop ENNs and prove its regularization properties under different assumptions. We show that for data with high noise levels, a small network architecture is sufficient to obtain a good approximate solution. For variational regularization, we analyze neural network approximation within the framework of Tikhonov regularization. Under variational source conditions, we establish convergence rate results. In summary, neural networks, when integrated with a slightly modified conventional regularization framework, can serve as effective regularization methods with theoretical guarantees.

\section*{Acknowledgements}
This work was funded by the Shenzhen Sci-Tech Fund (No. RCJC20231211090030059), National Key Research and Development Program of China (No. 2025YFE0113400) and National Natural Science Foundation of China (No. W2421102). The authors would like to thank Prof. Bangti Jin for valuable discussions and constructive suggestions.

\appendix

\section{Proof of lemmas and proposition in Section \ref{Preliminaries}}
\label{App:pre}
\subsection{Proof of Lemma \ref{lemma compact set}}
\begin{proof}
It suffices to verify that, for every sequence 
$\{f_n(a^k,\boldsymbol b^k,c^k;\mathbf x)\}_{k=1}^{\infty}
\subset X_{n,r}$, there exists a convergent subsequence whose limit
belongs to $X_{n,r}$. The parameters of 
$f_n(a^k,\boldsymbol b^k,c^k;\mathbf x)$ are given by $(a^k,\boldsymbol b^k,c^k)
:=
\bigl((a_j^k,\boldsymbol b_j^k,c_j^k)\bigr)_{j=1}^n
\in (M_r)^n $. Since $M_r$ is compact, $(M_r)^n$ is compact. Hence there exists a subsequence
$\{(a^{k_\ell},\boldsymbol b^{k_\ell},c^{k_\ell})\}_{\ell=1}^{\infty}$
and a limit $(\tilde a,\tilde{\boldsymbol b},\tilde c)
:=
\bigl((\tilde a_j,\tilde{\boldsymbol b}_j,\tilde c_j)\bigr)_{j=1}^n
\in (M_r)^n$ such that
\[
(a^{k_\ell},\boldsymbol b^{k_\ell},c^{k_\ell})
\to
(\tilde a,\tilde{\boldsymbol b},\tilde c)
\quad\text{as } \ell\to\infty
\quad\text{in }(\mathbb R\times\mathbb R^d\times\mathbb R)^n .
\]
Since $\Omega$ is bounded, there is $R>0$ with $\overline{\Omega}\subset\{x:\|x\|_\infty\le R\}$. Let $C_R:=\max\{R,1\}$. Since $t\mapsto t_+$ is $1$-Lipschitz and $\|\boldsymbol b\|_1+|c|=1$, for every $x\in \overline\Omega$, we have
\begin{align}
&\bigl|f_n(a^{k_{\ell}},\boldsymbol b^{k_{\ell}},c^{k_{\ell}};\mathbf{x})-f_n(\tilde a,\tilde{\boldsymbol b},\tilde c;\mathbf{x})\bigr|\notag\\
&\le \frac{1}{n}\sum_{j=1}^n \bigl|a_j^{(k_{\ell})}-\tilde a_j\bigr|\,
        \bigl( \bigl|(\boldsymbol b_j^{(k_{\ell})})^T \mathbf{x}\bigr| + \bigl|c_j^{(k_{\ell})}\bigr| \bigr)
   \;+\; \frac{1}{n}\sum_{j=1}^n \bigl|\tilde a_j\bigr|\,
        \bigl( \bigl|(\boldsymbol b_j^{(k_{\ell})}-\tilde{\boldsymbol b}_j)^T \mathbf{x}\bigr|
              + \bigl|c_j^{(k_{\ell})}-\tilde c_j\bigr| \bigr) \notag\\
              \label{eqlemma1}
&\le C_R\Biggl[
      \frac{1}{n}\sum_{j=1}^n \bigl|a_j^{(k_{\ell})}-\tilde a_j\bigr|
    + \frac{1}{n}\sum_{j=1}^n \bigl|\tilde a_j\bigr|
      \bigl(\|\boldsymbol b_j^{(k_{\ell})}-\tilde{\boldsymbol b}_j\|_1
           + \bigl|c_j^{(k_{\ell})}-\tilde c_j\bigr|\bigr)
    \Biggr],
\end{align}
Hence $||f_n(a^{k_{\ell}},\boldsymbol b^{k_{\ell}},c^{k_{\ell}};\mathbf{x})-f_n(\tilde a,\tilde{\boldsymbol b},\tilde c;\mathbf{x})||_{C(\overline\Omega)}\to 0$. Since $(\tilde a,\tilde{\boldsymbol b},\tilde c)\in(M_r)^n$, we have
$f_n(\tilde a,\tilde{\boldsymbol b},\tilde c;\mathbf{x})\in X_{n,r}$.
Thus $X_{n,r}$ is sequentially compact in
$C(\overline\Omega)$. Moreover,
$\|f\|_{L^2(\Omega)}\le |\Omega|^{1/2}\|f\|_{L^\infty(\Omega))}$. Hence $X_{n,r(n)}$ is also sequentially compact in $L^2(\Omega)$.
\end{proof}

\subsection{Proof of Lemma \ref{lem:barron_compact_lsc}}
\begin{proof}
Let $R:=\sup_{\mathbf x\in \overline \Omega}\|\mathbf x\|_\infty<\infty$. Since
$\|\boldsymbol b_j^{(m)}\|_1+|c_j^{(m)}|=1$, we have
$|(\boldsymbol b_j^{(m)})^T\mathbf x+c_j^{(m)}|\le R+1$ for
$\mathbf x\in \overline\Omega$. Hence
\[
|f_m(\mathbf x)|
\le q_m(R+1)
\le Q(R+1).
\]
Moreover, since $t\mapsto t_+$ is $1$-Lipschitz,
\[
|f_m(\mathbf x)-f_m(\mathbf y)|
\le q_m\|\mathbf x-\mathbf y\|_\infty
\le Q\|\mathbf x-\mathbf y\|_\infty .
\]
Thus $\{f_m\}$ is uniformly bounded and equicontinuous on $\overline \Omega$. By the Arzelà--Ascoli theorem, it is relatively compact in $C(\overline\Omega)$.
Let $\{f_{m_k}\}$ be any subsequence converging in
$C(\overline\Omega)$, and denote its limit by $f$, i.e., $f_{m_k}\to f$ in $C(\overline\Omega)$. By passing to a further
subsequence if necessary, assume that $q_{m_k}\to
L:=\liminf_{k\to\infty}q_{m_k}$. If $L=0$, then
$\|f_{m_k}\|_{C(\overline\Omega)}\le q_{m_k}(R+1)\to0$, so $f=0$, and the conclusion
is direct.
Now assume that $L>0$. For all sufficiently large $k$, $q_{m_k}>0$, define
$\alpha_j^{(m_k)}:=|a_j^{(m_k)}|/(n_{m_k}q_{m_k})$ and the probability
measure
\[
\rho_k
:=
\sum_{j=1}^{n_{m_k}}
\alpha_j^{(m_k)}
\delta_{\left(
\operatorname{sign}(a_j^{(m_k)})q_{m_k},
\boldsymbol b_j^{(m_k)},c_j^{(m_k)}
\right)} .
\]
Then $f_{m_k}(\mathbf x)
=
\int a(\boldsymbol b^T\mathbf x+c)_+
\,d\rho_k(a,\boldsymbol b,c)$. For every $\varepsilon>0$, $q_{m_k}\le L+\varepsilon$ for all large $k$. Hence the supports of $\rho_k$ are eventually contained in the compact set $K_{L+\varepsilon}
=
\{(a,\boldsymbol b,c): |a|\le L+\varepsilon,\ 
\|\boldsymbol b\|_1+|c|\le1\}$. By passing to a further subsequence, $\rho_k\rightharpoonup\rho^*$ as
probability measures on $K_{L+\varepsilon}$. For each fixed $\mathbf x\in\overline\Omega$, the function $(a,\boldsymbol b,c)\mapsto a(\boldsymbol b^T\mathbf x+c)_+$ is continuous and bounded on $K_{L+\varepsilon}$. Therefore,
\[
 f_{m_k}(\mathbf x)
\to
\int a(\boldsymbol b^T\mathbf x+c)_+
\,d\rho^*(a,\boldsymbol b,c).
\]
Since $f_{m_k}\to f$ in $C(\overline\Omega)$, it follows that $f(\mathbf x)
=
\int a(\boldsymbol b^T\mathbf x+c)_+
\,d\rho^*(a,\boldsymbol b,c),$ and hence $f\in\mathcal B_1$. Moreover, $\rho^*$ is supported in
$K_{L+\varepsilon}$, and  $\int |a|(\|\boldsymbol b\|_1+|c|)\,d\rho^* \le L+\varepsilon$. Taking the infimum over all Barron representations and then letting $\varepsilon\downarrow0$ gives
$\|f\|_{\mathcal B_1}
\le L= \liminf_{k\to\infty}q_{m_k}$.
\end{proof}

\subsection{Proof of Proposition \ref{Approximtion theorem}}
\begin{proof}
The estimates follow directly from
\cite[Theorem 8]{LiYuanyuan} by taking $k=1$ and $m=1$. It suffices to verify that the approximating network constructed there belongs to admissible class $X_{n,c_\rho(f)}$. By \cite[Lemma 2 and Theorem 8]{LiYuanyuan}, the approximating network
takes the form
\[
f_n(\mathbf x)
=
\frac1n\sum_{j=1}^n
c_\rho(f)\bar a_j
(\bar{\boldsymbol b}_j^T\mathbf x+\bar c_j)_+,\qquad |\bar a_j|\le1,
\quad
\|\bar{\boldsymbol b}_j\|_1+|\bar c_j|=1.
\]
Setting $a_j=c_\rho(f)\bar a_j,
\boldsymbol b_j=\bar{\boldsymbol b}_j,
c_j=\bar c_j,$ we obtain $|a_j|\le c_\rho(f),$ and $\|\boldsymbol b_j\|_1+|c_j|=1.$ Hence the same network can be written as $f_n(\mathbf x)
=
\frac{1}{n}\sum_{j=1}^n
a_j(\boldsymbol b_j^T\mathbf x+c_j)_+$, where $(a_j,\boldsymbol b_j,c_j)\in M_{c_\rho(f)}$, $j=1,\ldots,n$. Therefore $f_n\in X_{n,c_\rho(f)}.$
\end{proof}

\section{Extension to General Data Fidelity Terms}
\label{app:generalD}
The well-posedness and regularization results in Section \ref{Tikhonov} can be extended to a general data fidelity functional $\mathcal D:\mathcal Y\times\mathcal Y\to[0,\infty)$. For $\alpha>0$, define
\[
\mathcal T_{\alpha,\mathcal D}(f,g)
:=
\mathcal D(A(f),g)+\alpha\mathcal R_n(f).
\]
We impose the following assumptions on $\mathcal D$. These assumptions are standard within the regularization theory of abstract inverse problems formulated in normed spaces  \cite{TikhonovYagola1998}.
\begin{assumption}
\label{ass:generalD}
    The functional  $\mathcal{D}:\mathcal{Y}\times \mathcal{Y} \to \mathbb{R}_+$ satisfies the following conditions: 
    \begin{enumerate}
        \item[(i)] For some $\tau_{\mathcal D} \geq 1$ we have $\forall g_0, g_1, g_2 \in \mathcal{Y}: \mathcal{D}\left(g_0, g_1\right) \le \tau_{\mathcal D} \mathcal{D}\left(g_0, g_2\right)+\tau_{\mathcal D} \mathcal{D}\left(g_2, g_1\right)$; 
        \item[(ii)]  $\forall g_0, g_1 \in \mathcal{Y}: \mathcal{D}\left(g_0, g_1\right)=0 \Longleftrightarrow g_0=g_1$;
        \item[(iii)] $\forall\left\{g_k \right\}^{\infty}_{k=1} \subset \mathcal{Y}: g_k \to g \Longrightarrow \mathcal{D}\left(g_k, g\right) \to 0$;
        \item[(iv)] For every fixed $h\in\mathcal Y$, $g_k\to g$ implies $\mathcal D(h,g_k)\to\mathcal D(h,g)$;
        \item[(v)] The functional $(f,g)\mapsto \mathcal{D}(A(f),g)$ is sequentially lower semi-continuous with respect to convergence in $C(\overline\Omega)\times\mathcal Y$.
    \end{enumerate}
\end{assumption}
For the regularization result, assume further that there exist $\omega_1(n)\to0$ as $n\to\infty$ and $\omega_2(\delta)\to0$ as $\delta\to0$ such that, for every solution $\bar f\in\mathcal B_1$ of $A(\bar f)=g$, the approximating networks $\bar f_n$ given by
Proposition~\ref{Approximtion theorem} satisfy
\begin{equation}
\label{generalD_approx}
\mathcal D(A(\bar f_n),g)
\le C_{\bar f}\omega_1(n),
\qquad
\mathcal D(g,g^\delta)
\le \omega_2(\delta),
\end{equation}
where $C_{\bar f}>0$ is independent of $n$ and $\delta$.
Set $\omega(n,\delta):=\omega_1(n)+\omega_2(\delta)$.

\begin{proposition}
\label{prop:generalD}
Suppose that Assumption~\ref{ass:generalD} holds. Then  Theorem~\ref{well-posedness}(a)--(b) remain valid with $\mathcal T_\alpha$ replaced by $\mathcal T_{\alpha,\mathcal D}$. If, in addition, \eqref{generalD_approx} holds and
\begin{equation}
\label{newchicealpha}
    \alpha_{n,k}\to0,
\qquad
\frac{\omega(n,\delta_k)}{\alpha_{n,k}}\to0
\quad\text{as }n,k\to\infty,
\end{equation}
then Theorem~\ref{well-posedness}(c) also remains valid.
\end{proposition}

\begin{proof}
For part~(a), let $\{f_{n,k}\}$ be a minimizing sequence.
As in the proof of Theorem~\ref{well-posedness}(a), $\{\mathcal R_n(f_{n,k})\}$ are uniformly bounded. Hence, after passing to a subsequence, there exists $f_n^*\in X_{n,\infty}$ such that
$f_{n,k_\ell}\to f_n^*$ in $C(\overline\Omega)$ and $\mathcal R_n(f_{n,k_\ell})\to \mathcal R_n(f_n^*)$. By
Assumption~\ref{ass:generalD}(v),
\[
\mathcal T_{\alpha,\mathcal D}(f_n^*,g)
\le
\liminf_{\ell\to\infty}
\mathcal T_{\alpha,\mathcal D}(f_{n,k_\ell},g)
=
\inf_{f\in X_{n,\infty}}
\mathcal T_{\alpha,\mathcal D}(f,g).
\]
Thus $f_n^*$ is a minimizer of $\mathcal T_{\alpha,\mathcal D}(\cdot,g)$ over $X_{n,\infty}$. 

For part~(b), the minimizing property gives 
\[
\alpha\mathcal R_n(f_{n,k})
\le
\mathcal T_{\alpha,\mathcal D}(f_{n,k},g_k)
\le
\mathcal T_{\alpha,\mathcal D}(0,g_k).
\]
Since $g_k\to g$, Assumption~\ref{ass:generalD}(iv), applied with $h=A(0)$, implies that $\{\mathcal T_{\alpha,\mathcal D}(0,g_k)\}$ is bounded. Hence $\{\mathcal R_n(f_{n,k})\}$ is uniformly bounded, and
the same compactness argument yields a subsequence
$f_{n,k_\ell}\to f_n^*$ in $C(\overline\Omega)$, and hence $\mathcal R_n(f_{n,k_\ell})\to \mathcal R_n(f_n^*)$. By the sequential lower semi-continuity of $(f,g)\mapsto\mathcal D(A(f),g)$ in Assumption~\ref{ass:generalD}(v) and the convergence $g_{k_\ell}\to g$, we have $\mathcal T_{\alpha,\mathcal D}(f_n^*,g)
\le
\liminf_{\ell\to\infty}
\mathcal T_{\alpha,\mathcal D}
(f_{n,k_\ell},g_{k_\ell}).$
On the other hand, for every $f\in X_{n,\infty}$, we have $\mathcal T_{\alpha,\mathcal D}
(f_{n,k_\ell},g_{k_\ell})
\le
\mathcal T_{\alpha,\mathcal D}(f,g_{k_\ell}).$ By Assumption~\ref{ass:generalD}(iv),
$\mathcal T_{\alpha,\mathcal D}(f,g_{k_\ell})
\to\mathcal T_{\alpha,\mathcal D}(f,g)$.
Consequently, $\mathcal T_{\alpha,\mathcal D}(f_n^*,g)
\le
\mathcal T_{\alpha,\mathcal D}(f,g),$ for any $f\in X_{n,\infty}$. Therefore
$f_n^*\in\arg\min_{f\in X_{n,\infty}}
\mathcal T_{\alpha,\mathcal D}(f,g)$.

For part~(c), fix an arbitrary sequence
$\{(n_\ell,k_\ell)\}_{\ell=1}^\infty$ such that
$n_\ell\to\infty$ and $k_\ell\to\infty$, and set
$f_\ell:=f_{n_\ell,k_\ell}$ and
$\alpha_\ell:=\alpha_{n_\ell,k_\ell}$. Let $\bar f\in\mathcal B_1$ be an arbitrary solution of $A(\bar f)=g$. Let $\epsilon\in(0,1)$ be arbitrary, by the definition of the Barron norm, there exists a representing probability measure
$\rho(\epsilon)$ of $\bar f$ such that $c_{\rho(\epsilon)}(\bar f)
\le
(1+\epsilon)\|\bar f\|_{\mathcal B_1}\le 2\|\bar f\|_{\mathcal{B}_1}.$ Let $\bar f_{n_\ell}^\epsilon$ be the approximating networks given by Proposition \ref{Approximtion theorem}. By Assumption~\ref{ass:generalD}(i) and \eqref{generalD_approx}, we have $\mathcal D(A(\bar f_{n_\ell}^\epsilon),g_{k_\ell})
\le
\tau_{\mathcal D}
\left(
C_{\bar f}\omega_1(n_\ell)
+\omega_2(\delta_{k_\ell})
\right)$, where $C_{\bar f}=L_A^2\left(2C(\Omega,d)\|\bar f\|_{\mathcal B_1}\right)^{2\theta}$.
Hence, by the minimizing property of $f_{\ell}$,
\[
\mathcal D(A(f_\ell),g_{k_\ell})
+\alpha_\ell\mathcal R_{n_\ell}(f_\ell)
\le
\tau_{\mathcal D}
\left(
C_{\bar f}\omega_1(n_\ell)
+\omega_2(\delta_{k_\ell})
\right)
+\alpha_\ell
\mathcal R_{n_\ell}(\bar f_{n_\ell}^\epsilon).
\]
By Proposition~\ref{Approximtion theorem} and the monotonicity of
$\varphi$, we have $\mathcal R_{n_\ell}(\bar f_{n_\ell}^\epsilon)
\le
\varphi(c_{\rho(\epsilon)}(\bar f))
\le
\varphi\bigl((1+\epsilon)\|\bar f\|_{\mathcal B_1}\bigr).$ Hence
\[
\mathcal R_{n_\ell}(f_\ell)
\le
\tau_{\mathcal D}
\frac{
C_{\bar f}\omega_1(n_\ell)+\omega_2(\delta_{k_\ell})
}{
\alpha_\ell
}
+
\varphi((1+\epsilon)\|\bar f\|_{\mathcal B_1}).
\] 
This estimate holds for every $\epsilon\in(0,1)$ and every solution $\bar f\in\mathcal B_1$ of $A(\bar f)=g$. In particular, fixing any such $\bar f$ and taking
$\epsilon=1/2$, the parameter choice \eqref{newchicealpha} implies that $\{\mathcal R_{n_\ell}(f_{\ell})\}$ is uniformly bounded. Thus, the same compactness argument as in the proof of Theorem~\ref{well-posedness}(c), combined with Lemma~\ref{lem:barron_compact_lsc}, yields a subsequence, not relabeled,
such that $f_\ell\to f^*$ in $C(\overline\Omega)$ for some $f^*\in\mathcal B_1$. Moreover, the preceding minimizing inequality, together with
$\omega_1(n_\ell)\to0$, $\omega_2(\delta_{k_\ell})\to0$, and $\alpha_{\ell}\to0$ implies $\mathcal D(A(f_{\ell}),g_{k_\ell})\to0$. By
Assumption~\ref{ass:generalD}(v),
$\mathcal D(A(f^*),g)
\le
\liminf_{\ell\to\infty}
\mathcal D(A(f_\ell),g_{k_\ell})
=0$, and hence $A(f^*)=g$ by
Assumption~\ref{ass:generalD}(ii). Finally, by Lemma~\ref{lem:barron_compact_lsc} and the properties of $\varphi$, we obtain $\mathcal R(f^*)
\le
\liminf_{\ell\to\infty}
\mathcal R_{n_\ell}(f_\ell)
\le
\limsup_{\ell\to\infty}
\mathcal R_{n_\ell}(f_\ell)
\le
\varphi((1+\epsilon)\|\bar f\|_{\mathcal B_1}).$ Letting $\epsilon\downarrow0$ yields $\mathcal R(f^*)\le\mathcal R(\bar f)$. Since $\bar f$ was arbitrary, $f^*$ is an $\mathcal R$-minimizing solution.
\end{proof}

For the squared norm discrepancy $\mathcal D(g_1,g_2)=\|g_1-g_2\|_{\mathcal Y}^2$ considered in
the main text, Assumption~\ref{ass:generalD} under the standing assumption that $A:L^2(\Omega)\to\mathcal Y$ is continuous, with
$\tau_{\mathcal D}=2$. Moreover, the local H\"older continuity of $A$, Proposition~\ref{Approximtion theorem}, and the noise assumption give
$\omega_1(n)=n^{-\theta}$ and $\omega_2(\delta)=\delta^2$, respectively. Hence $\omega(n,\delta)=n^{-\theta}+\delta^2$, and Proposition~\ref{prop:generalD} reduces to Theorem~\ref{well-posedness}.



\bibliographystyle{plain}
\bibliography{References}
\end{document}